\documentclass[11pt,english,a4paper]{amsart}
\usepackage[utf8]{inputenc}
\usepackage[T1]{fontenc}
\usepackage[english]{babel}
\usepackage{verbatim}

\usepackage{amsmath}
\usepackage{amssymb}
\usepackage{amsthm}
\usepackage{lmodern}
\usepackage{latexsym}
\usepackage{amsfonts}
\usepackage{mathrsfs}
\usepackage[all]{xy}
\usepackage{bm}
\usepackage{yfonts}

\usepackage{hyperref}
\hypersetup{pdftoolbar=true, pdftitle={Amplified moments of zeta}, pdffitwindow=true, colorlinks=true, citecolor=green, filecolor=black, linkcolor=red, urlcolor=red, hypertexnames=false}

\numberwithin{equation}{section}

\usepackage[margin=1in]{geometry}

\DeclareMathOperator*{\Res}{Res}
\newtheorem{thm}{Theorem}

\newtheorem{prop}{Proposition}
\newtheorem{lemma}[prop]{Lemma}
\newtheorem{cor}[prop]{Corollary}

\theoremstyle{definition}

\renewcommand{\pmod}[1]{\,(\textup{mod}\,#1)}

\newcommand{\be}{\begin{equation*}}
\newcommand{\ee}{\end{equation*} }

\newcommand{\ben}{\begin{equation}}
\newcommand{\een}{\end{equation} }

\newcommand{\bs}{\begin{split}}
\newcommand{\es}{\end{split}}

\newcommand{\bmu}{\begin{multline*}}
\newcommand{\emu}{\end{multline*}}

\newcommand{\bmun}{\begin{multline}}
\newcommand{\emun}{\end{multline}}

\usepackage{multicol}
\usepackage{multirow}

\allowdisplaybreaks

\begin{document}
\title[Amplified moments of zeta]{Amplified moments of the Riemann zeta function}

\author[B. Durkan]{Benjamin Durkan}
\address{Department of Mathematics, The University of Manchester, Oxford Road, Manchester, M13 9PL}
\email{benjamin.durkan@manchester.ac.uk}

\author[T. Page]{Timothy Page}
\address{}
\email{tmpage@hotmail.co.uk}

\begin{abstract}
We establish asymptotic formulae for two-piece amplified second and fourth moments of the Riemann zeta function. As applications, we obtain unconditional effective lower bounds for several joint moments of zeta which are in strong agreement with the conjectures of Keating--Wei and Keating--Snaith. In particular, we prove an unconditional lower bound for the sixth moment of zeta $M_3(T)\geq(34.4+o(1))c_3T(\log T)^9$. We further improve some of the lower bounds obtained by Soundararajan, removing the assumption of the Lindel\H{o}f Hypothesis, and we obtain effective lower bounds for all joint integer moments of zeta consistent with the predictions of random matrix theory.
\end{abstract}

\thanks{BD gratefully acknowledges support from the Heilbronn Institute for Mathematical Research. We thank Emma Bailey, Hung Bui, Chris Hughes, Andrew Pearce-Crump and Tim Trudgian for their helpful comments and conversations, which have significantly improved the paper.}

\maketitle

\tableofcontents
\section{Introduction}
The evaluation of moments of the Riemann zeta function is a central problem in analytic number theory. For $k\ge 0$, we write

$$M_k(T)=\int_0^T|\zeta(\tfrac{1}{2}+it)|^{2k}\;dt.$$

Motivated by random matrix theory, Keating and Snaith \cite{KeatingSnaith} conjectured that for fixed complex $k$ with $\Re(k)>-1/2$,

\begin{equation}\label{keating-snaith}
    \int_0^T\left|\zeta(\tfrac{1}{2}+it)\right|^{2k}\;dt\sim a_k\frac{G^2(k+1)}{G(2k+1)}T(\log T)^{k^2},
\end{equation}

\
where $G(z)$ denotes the Barnes $G$-function and the arithmetic factor $a_k$ is given by the Euler product

\begin{equation}
\label{eqn:arithmetic_factor}
a_k=\prod_p\left(1-\frac{1}{p}\right)^{k^2}\sum_{m=0}^{\infty}\left(\frac{\Gamma(m+k)}{\Gamma(m+1)\Gamma(k)}\right)^2p^{-m}.
\end{equation}

At present, asymptotic formulae of the shape \eqref{keating-snaith} are known only for $k=1,2$, the second moment going back to Hardy and Littlewood \cite{hardy} and the fourth moment going back to Ingham \cite{Ingham}. For larger $k$, sharp lower and upper bounds are known (the latter conditional on the Riemann Hypothesis, due to Harper's \cite{harper} refinement of Soundararajan's method \cite{sound}), but an asymptotic formula of the form \eqref{keating-snaith} remains far beyond current methods. General lower-bound methods originate in work of Ramachandra \cite{Ramachandra} and were sharpened by Rudnick and Soundararajan \cite{RudnickSound}, Radziwi{\l}{\l} and Soundararajan \cite{RadziwillSound} and Heap and Soundararajan \cite{heap2022lower}. We note that for the derivative of $\zeta(s)$ a sharp lower bound was obtained by Gao \cite{gao2021lower} for all real $k\ge 0$.

In fact, if the bound \eqref{keating-snaith} were available for all positive integer $k$, this would immediately imply the Lindel\"{o}f Hypothesis. Whilst bounds are available for all real $k\ge 0$, in this paper we shall be concerned solely with integer $k$. It is a well-studied problem to obtain good numerical bounds for these moments, and an active area of research. This work improves upon a number of these bounds, as asked in e.g. \cite[Problem 6]{problem_list}. For a fuller account of the history of the bounds obtained for moments of zeta, the reader is encouraged to consult Florea's article \cite{florea}.

\subsection{Setup and statement of main results}
The first type of results are unconditional lower bounds for moments of $\zeta(s)$, $\zeta^{'}(s)$ and $\zeta''(s)$, which are deduced in Section \ref{section:corollaries}. Here we use the notation of \cite{soundararajan1995mean} whereby we set $c_k = a_k/k^2!$.

\begin{thm}\label{cor:6thmoment}
    We have
    \begin{equation*}
        \int_0^T|\zeta(\tfrac{1}{2}+it)|^6\;dt\ge (34.4+o(1))c_3T(\log T)^9.
    \end{equation*}
\end{thm}

\
This can be compared to the conjectured $42c_3T(\log T)^9$ from the Keating--Snaith conjecture \eqref{keating-snaith} and the conjectures of Conrey-- Ghosh \cite{conrey_ghosh_sixth} and Conrey--Gonek \cite{conrey_gonek_high}. This improves upon previous bounds of Conrey and Ghosh \cite{conrey_ghosh_mean_values_iii}, Soundararajan \cite{soundararajan1995mean} and Page \cite{page}.

\begin{thm}\label{cor:6thmoment_deriv}
    We have
    \begin{equation*}
        \int_0^T|\zeta'(\tfrac{1}{2}+it)|^6\;dt\ge (0.549+o(1))c_3T(\log T)^{15}.
    \end{equation*}
\end{thm}

\
This can be compared to the conjectured $(0.71948...) c_3 T (\log T)^{15}$ due to Hughes \cite{Hughes_thesis}. This improves upon the bound of $0.419c_3T(\log T)^{15}$ in Page \cite{page}.

\begin{thm}\label{cor:6thmoment_secondderiv}
    We have
    \begin{equation*}
        \int_0^T|\zeta''(\tfrac{1}{2}+it)|^6\;dt\ge (0.0231+o(1))c_3T(\log T)^{21}.
    \end{equation*}
\end{thm}

\
This can be compared to the conjectured $(0.03136...) c_3 T (\log T)^{21}$ in \cite{keating2023jointmomentshigherorder}.

We also establish lower bounds for joint moments by proving the following results, although with greater computational power one can obtain similar results for many more joint moments.

\begin{thm}\label{cor:JM(2,4)}
    We have
    \begin{equation*}
        \int_0^T|\zeta(\tfrac{1}{2}+it)|^2|\zeta'(\tfrac{1}{2}+it)|^4\;dt\ge  (2.216+o(1))c_3T(\log T)^{13}.
    \end{equation*}
\end{thm}

\
This can be compared to the conjectured $2.78\dot{3} c_3T(\log T)^{13}$ from Keating--Wei  \cite{keating2023jointmomentshigherorder}.

\begin{thm}\label{cor:JM(4,2)}
    We have
    \begin{equation*}
        \int_0^T|\zeta(\tfrac{1}{2}+it)|^4|\zeta'(\tfrac{1}{2}+it)|^2\;dt\ge (7.828+o(1))c_3T(\log T)^{11}.
    \end{equation*}
\end{thm}

\
This can be compared to the conjectured $10.8c_3T(\log T)^{11}$ from Keating--Wei \cite{keating2023jointmomentshigherorder}.

\begin{thm}\label{thm:joint_02}
    We have
    \begin{equation*}
        \int_0^T|\zeta(\tfrac{1}{2}+it)|^2|\zeta''(\tfrac{1}{2}+it)|^4\;dt\ge (0.286+o(1))c_3T(\log T)^{17}.
    \end{equation*}
\end{thm}
This can be compared to the conjectured $(0.337\cdots+o(1))c_3T(\log T)^{17}$ from Keating--Wei \cite{keating2023jointmomentshigherorder}.

\begin{thm}\label{0412}
    We have
    \begin{equation*}
        \int_0^T|\zeta''(\tfrac{1}{2}+it)|^2|\zeta(\tfrac{1}{2}+it)|^4\;dt\ge (2.448+o(1))c_3T(\log T)^{13}.
    \end{equation*}
\end{thm}
This can be compared to the conjectured $(3.7\dot{3}+o(1))c_3T(\log T)^{13}$ from Keating--Wei \cite{keating2023jointmomentshigherorder}.

\begin{thm}\label{thm:0224}
    We have
    \begin{equation*}
        \int_0^T|\zeta'(\tfrac{1}{2}+it)|^2|\zeta''(\tfrac{1}{2}+it)|^4\;dt\ge (0.0726+o(1))c_3T(\log T)^{19}.
    \end{equation*}
\end{thm}
This can be compared to the conjectured $(0.088\cdots+o(1))c_3T(\log T)^{19}$ from Keating--Wei \cite{keating2023jointmomentshigherorder}.

\begin{thm}\label{thm:2204}
    We have
    \begin{equation*}
        \int_0^T|\zeta''(\tfrac{1}{2}+it)|^2|\zeta'(\tfrac{1}{2}+it)|^4\;dt\ge (0.175+o(1))c_3T(\log T)^{17}.
    \end{equation*}
\end{thm}
This can be compared to the conjectured $(0.250\cdots+o(1))c_3T(\log T)^{17}$ from Keating--Wei \cite{keating2023jointmomentshigherorder}.

The approach resulting in polytope integrals gives strong lower bounds for the moments above, but becomes computationally difficult for higher moments, and is also limited by the sharp decay in the length of the Dirichlet polynomials. For higher moments, we can obtain better bounds by modifying Soundararajan's arguments in \cite{soundararajan1995mean}.

Our next result is an unconditional refinement of the bounds obtained in \cite{soundararajan1995mean}.

\begin{thm}\label{thm:sound_bounds_unconditional}
Unconditionally,
$$
        M_5(T)\geq (6484+o(1))c_5T(\log T)^{25},
$$
$$
        M_6(T)\geq(56260+o(1))c_6T(\log T)^{36},
$$
$$
        M_7(T)\geq(597.122892023379+o(1))c_7T(\log T)^{49},
$$
$$
        M_8(T)\geq(139.258401860194+o(1))c_8T(\log T)^{64},
$$
$$
        M_9(T)\geq(12.570855534299+o(1))c_9T(\log T)^{81},
$$
$$
        M_{10}(T)\geq(2.259685508647+o(1))c_{10}T(\log T)^{100},
$$
and
$$
        M_{11}(T)\geq(2.001973286783+o(1))c_{11}T(\log T)^{121}.
$$
\end{thm}

The cases $k=5,6$ were established by Soundararajan under the assumption of the Lindel\"{o}f Hypothesis. By modifying his argument slightly we obtain these results unconditionally. For higher values of $k$ we introduce a second piece to Soundararajan's amplifier, which yields improved lower bounds for $7\le k\le 11$. In \cite{soundararajan1995mean} the bounds obtained previously were $2c_kT(\log T)^{k^2}$. Theorem \ref{thm:sound_bounds_unconditional} delivers stronger bounds for these low moments, but as $k\to\infty$ our bounds approach those of Soundararajan's.

We next establish a general lower bound for all joint integer moments. We first introduce some notation. 
Let
$$
(a_1,\ldots,a_K)
=
(\underbrace{n_1,\ldots,n_1}_{k_1\text{ times}},
\underbrace{n_2,\ldots,n_2}_{k_2\text{ times}},
\ldots,
\underbrace{n_j,\ldots,n_j}_{k_j\text{ times}}),
$$
where $K=k_1+\cdots+k_j$. Thus
$$
N:=a_1+\cdots+a_K=k_1n_1+\cdots+k_jn_j.
$$

For $m\geq1$ define

$$
\lambda_{a_1,\ldots,a_K}(m)=\sum_{m_1\cdots m_K=m}(\log m_1)^{a_1}\cdots(\log m_K)^{a_K},
$$
the sum being over ordered positive integer $K$-tuples; when an exponent is zero we use the usual convention $x^0=1$, also for $x=0$.  For $\Re s>1$,

\begin{equation}\label{eqn:prod_zeta_joint}
\prod_{r=1}^{K}\zeta^{(a_r)}(s)=(-1)^N\sum_{m=1}^{\infty}\frac{\lambda_{a_1,\ldots,a_K}(m)}{m^s}.       
\end{equation}

For the fixed ordered tuple $a_1,\ldots,a_K$, write $C(a_1,\ldots,a_K)$ for the following constant:

\begin{equation}\label{eqn:C_def}
\begin{aligned}
C(a_1,\ldots,a_K)=  & \frac{1}{\Gamma(K^2+2N+1)}\left(\prod_p\left(1-\frac1p\right)^{K^2}
       \sum_{\ell=0}^{\infty}\frac{\binom{\ell+K-1}{K-1}^2}{p^\ell}\right)                                      \\
&\times\left.
\frac{\partial^{2N}}{\partial x_1^{a_1}\cdots\partial x_K^{a_K}\partial y_1^{a_1}\cdots\partial y_K^{a_K}}
\prod_{r=1}^{K}\prod_{q=1}^{K}\frac{1}{1+x_r+y_q}
\right|_{x_1=\cdots=x_K=y_1=\cdots=y_K=0} .
\end{aligned}                  
\end{equation}

The differentiated factor in \eqref{eqn:C_def} is finite and may be made explicit.  Indeed, expanding
$$
(1+x_r+y_q)^{-1}
=\sum_{i,j\geq0}(-1)^{i+j}\binom{i+j}{i}x_r^iy_q^j .
$$
Thus the derivative in \eqref{eqn:C_def} is
\begin{equation}\label{eqn:C_derivative_explicit}
\left(\prod_{r=1}^{K}(a_r!)^2\right)
\sum_{\substack{i_{rq},j_{rq}\ge0\\
\sum_{q=1}^{K}i_{rq}=a_r\ (1\le r\le K)\\
\sum_{r=1}^{K}j_{rq}=a_q\ (1\le q\le K)}}
\prod_{r=1}^{K}\prod_{q=1}^{K}\binom{i_{rq}+j_{rq}}{i_{rq}} .
\end{equation}
This is a weighted count of pairs of non-negative $K\times K$ arrays: the first array has prescribed row sums $a_1,\ldots,a_K$, and the second has prescribed column sums $a_1,\ldots,a_K$.  Equivalently, the derivative in \eqref{eqn:C_def} is
$$
\int_{[0,\infty)^{K^2}}e^{-\sum_{r,q}t_{rq}}
\prod_{r=1}^{K}\left(\sum_{q=1}^{K}t_{rq}\right)^{a_r}
\prod_{q=1}^{K}\left(\sum_{r=1}^{K}t_{rq}\right)^{a_q}
dt_{1,1} \cdots dt_{K, K},
$$
which follows by expanding the row and column sums and integrating each monomial.  In particular there is no cancellation, because every term has total sign $(-1)^{2N}=1$.  The diagonal choice $i_{rr}=j_{rr}=a_r$, all other entries zero, gives the contribution $\prod_r(2a_r)!$.  On the other hand, using $\binom{i+j}{i}\le 2^{i+j}$ gives the simple bounds
$$
\prod_{r=1}^{K}(2a_r)!
\le
\text{the derivative in \eqref{eqn:C_def}}
\le
4^N\left(\prod_{r=1}^{K}(a_r!)^2\binom{a_r+K-1}{K-1}^2\right).
$$
This is a coefficient problem of the kind appearing in MacMahon's partition analysis, or equivalently a weighted vector-partition count \cite{AndrewsPauleRiese,SturmfelsVectorPartitions}.

The $p$-factor in the Euler product in \eqref{eqn:C_def} is positive and is $1+O_K(p^{-2})$, since the $\ell=0,1$ terms in the sum are $1+K^2/p$; hence the Euler product converges to a positive number.  This Euler product is the arithmetic factor in \eqref{eqn:arithmetic_factor} at $k=K$, namely $(K^2)!c_K$, so the preceding bounds give the explicit estimate
$$
\frac{(K^2)!c_K}{(K^2+2N)!}\prod_{r=1}^{K}(2a_r)!
\le C(a_1,\ldots,a_K)
\le
\frac{(K^2)!c_K}{(K^2+2N)!}
4^N\left(\prod_{r=1}^{K}(a_r!)^2\binom{a_r+K-1}{K-1}^2\right).
$$
In particular $C(a_1,\ldots,a_K)>0$.  As checks, if $N=0$ then $C(a_1,\ldots,a_K)=c_K$, while if $K=1$ then $C(a_1)=1/(2a_1+1)$.

\begin{thm}\label{thm:joint_lower}
With $C$ as in \eqref{eqn:C_def} we have,
$$
\int_T^{2T}\prod_{\mu=1}^{j}\left|\zeta^{(n_\mu)}\left(\tfrac{1}{2}+it\right)\right|^{2k_\mu}dt
\geq\left(C(a_1,\ldots,a_K)+o(1)\right)T(\log T)^{K^2+2N}.
$$
The same lower bound holds with $[T,2T]$ replaced by $[1,T]$.

\end{thm}

\subsection{Amplified moments}
A natural way to probe higher moments is via twisted and amplified mean values. Given a Dirichlet polynomial

$$A(s)=\sum_{n\le y}\frac{a(n)}{n^s} \qquad (a(n)\ll n^{\varepsilon})$$
with $y$ given by a power of $T$, one would like to choose $a(n)$ so as to model high powers of $\zeta(s)$. There are two natural and complementary ways to encode such powers.

\begin{itemize}
    \item In one approach, used recently by Bui, Hall and Subira Jorge \cite{bui2025amplified}, one chooses $A(s)$ so that its coefficients mimic those of $\zeta(s)^r$ for some $r>0$, for instance $a(n)\approx d_r(n)$ where $d_r(n)$ are the $r$-fold divisor coefficients.
    \item In the present paper we adopt a different point of view. We fix a relatively short Dirichlet polynomial $A(s)$ of length $y=T^{\theta_k}$ or $y=T^{\vartheta_k}$ for twisted second and fourth moments respectively, with smooth polynomial coefficients 
    
    \begin{equation}\label{eqn:smooth_coeffs}
        P[n]=P\left(\frac{\log(y/n)}{\log y}\right),
    \end{equation}
for $1\le n\le y$ where $P(x)=\sum_{j\ge 0}c_jx^j$ is a certain polynomial with real coefficients. We set also $P[n]=0$ for $n\ge y$ by convention.
    \
    We allow arbitrary even powers $|A|^{2k}$ to appear inside twisted moments of zeta. From the perspective of modelling $\zeta(s)^{2k}$, this amounts to representing large powers of $\zeta(s)$ by the product $|A(\tfrac{1}{2}+it)|^{2k}$, rather than by a single long Dirichlet polynomial with divisor-type coefficients.
\end{itemize}

\vspace{2mm}

These two frameworks behave rather differently analytically. In our setup, the presence of the high power $|A|^{2k}$ forces the admissible exponent to decrease with $k$ for the twisted second moment we can take $\theta_k< \tfrac{1}{2k}$, while for the twisted fourth moment we require $\vartheta_k< \tfrac{1}{4k}$. In particular, our method does not recover the phenomenon of \cite{bui2025amplified} that one can take amplifiers of length $T^{1/8-\varepsilon}$ uniformly in $k$. On the other hand, the polynomial structure of the coefficients and the flexibility in the amplifier power make our approach very effective for acquiring bounds on joint moments. We note that our coefficients are chosen to be polynomial, largely for technical convenience -- with a much more technical argument one should be able to modify our results to handle more general coefficients $a(n)$.

In what follows we focus on twisted second and fourth moments with a high power of the amplifier attached, namely integrals of the form

\begin{equation}\label{eqn:twists}
\int_0^T|\zeta(\tfrac{1}{2}+it)|^2|A(\tfrac{1}{2}+it)|^{2k}\;dt, \ \ \ \ \ \ \ \int_0^T|\zeta(\tfrac{1}{2}+it)|^4|A(\tfrac{1}{2}+it)|^{2k}\;dt
\end{equation}

with $A$ of length $T^{\theta_k}$ in the twisted second moment and length $T^{\vartheta_k}$ in the twisted fourth moment. The case $k=1$ in the second moment was treated by Balasubramanian, Conrey and Heath-Brown \cite{BCHB} to handle Dirichlet polynomials with length up to $T^{1/2-\delta}$; for general coefficients this range has since been pushed slightly beyond $T^{1/2}$ by Bettin, Chandee and Radziwi{\l}{\l} \cite{BCR}. Any improvement to the length of the permissible Dirichlet polynomial for the twisted second moment would automatically yield a stronger result for our twisted second moment.

For the twisted fourth moment, Hughes and Young \cite{Hughes_2010} obtained an asymptotic in the case $k=1$ with $y\le T^{1/11-\delta}$ and Bettin, Bui, Li and Radziwi{\l}{\l} \cite{BBLR} later extended this to $y\le T^{1/4-\delta}$. This is conjecturally the longest one can take the Dirichlet polynomial to be, according to Selberg's eigenvalue conjecture. More recently, Bui, Hall and Subira-Jorge \cite{bui2025amplified} handled the case of a fourth power of the amplifier with coefficients modelled on $d_r(n)$, allowing length $y\le T^{1/8-\delta}$ and used this result to establish a large-gaps result for consecutive zeta zeros.

Our first results deal with further twisted analogues of \eqref{eqn:twists}, where we consider

$$\int_0^T|\zeta(\tfrac{1}{2}+it)|^2|A(\tfrac{1}{2}+it)+\chi(\tfrac{1}{2}+it)A(\tfrac{1}{2}-it)|^{2k}\;dt$$

and

$$\int_0^T|\zeta(\tfrac{1}{2}+it)|^4|A(\tfrac{1}{2}+it)+\chi(\tfrac{1}{2}+it)A(\tfrac{1}{2}-it)|^{2k}\;dt,$$

these twists representing the form of the approximate functional equation. This is reminiscent of Soundararajan's approach in \cite{soundararajan1995mean}.

We establish the first results in the literature for asymptotics in the case of general $k$, at the cost of a shrinking amplifier length as $k$ grows.

\subsection{Amplified moment results}
We start by observing that one may utilise the amplified second and fourth moments to establish a lower bound for the $2k$th moment of $\zeta(s)$. In what follows we shall tacitly assume $k\ge 2$. We have, by H\"{o}lder's inequality,

\begin{equation}\label{eq:holder_one_piece}
    M_k(T)\ge\frac{\left(\int_0^T|\zeta(\tfrac{1}{2}+it)|^4|A(\tfrac{1}{2}+it)|^{2k-4}\;dt\right)^{k-1}}{\left(\int_0^T|\zeta(\tfrac{1}{2}+it)|^2|A(\tfrac{1}{2}+it)|^{2k-2}\;dt\right)^{k-2}},
\end{equation}

where

\begin{equation*}
    A(s)=\sum_{n\le y}\frac{P[n]}{n^s}
\end{equation*}

is an approximation of $\zeta(s)$ by a Dirichlet polynomial and where $P[n]$ is as in \eqref{eqn:smooth_coeffs}. Introducing the two-piece twist improves this lower bound. More precisely,

\begin{equation}\label{eqn:holder}
     M_k(T)\ge\frac{\left(\int_0^T|\zeta(\tfrac{1}{2}+it)|^4|A(\tfrac{1}{2}+it)+\chi(\tfrac{1}{2}+it)A(\tfrac{1}{2}-it)|^{2k-4}\;dt\right)^{k-1}}{\left(\int_0^T|\zeta(\tfrac{1}{2}+it)|^2|A(\tfrac{1}{2}+it)+\chi(\tfrac{1}{2}+it)A(\tfrac{1}{2}-it)|^{2k-2}\;dt\right)^{k-2}}.
\end{equation}

\vspace{3mm}

The motivation for this twist is that we would like $A(s)$ to approximate $\zeta(s)$, and the combination $A(s)+\chi(s)A(1-s)$ is the natural short analogue of the approximate functional equation. This formulation is also well adapted to differentiated moments, since shifts $\alpha_i \in\mathbb{C}$ with $|\alpha_i|\ll 1/ \log T$ introduced to the zeta factors can later be differentiated and then specialised at the origin. A similar use of H\"older's inequality also leads to lower bounds for joint moments, but we first discuss the classical moments to make the main mechanism transparent. With this in mind, we make the definitions

\begin{align*}
    I_k(\alpha_1,\alpha_2)&=\int_0^T\zeta(\tfrac{1}{2}+\alpha_1+it)\zeta(\tfrac{1}{2}+\alpha_2-it)|A(\tfrac{1}{2}+it)+\chi(\tfrac{1}{2}+it)A(\tfrac{1}{2}-it)|^{2k}\;dt,\\
    J_k(\alpha_1,\alpha_2,\alpha_3,\alpha_4)&=\int_0^T\zeta(\tfrac{1}{2}+\alpha_1+it)\zeta(\tfrac{1}{2}+\alpha_2+it)\zeta(\tfrac{1}{2}+\alpha_3-it)\zeta(\tfrac{1}{2}+\alpha_4-it)\\&\hspace{10mm}\times|A(\tfrac{1}{2}+it)+\chi(\tfrac{1}{2}+it)A(\tfrac{1}{2}-it)|^{2k}\;dt.
\end{align*}

\vspace{3mm}
The first simplification is that the $2k$-th power of the two-piece amplifier may be decomposed into a finite linear combination of terms $|A(s)|^{2k-2\nu}A(s)^{2\nu}\chi(1-s)^\nu$. This separates the genuinely oscillatory pieces from the main terms.

\begin{prop}\label{prop:binomial_expansion}
    Let $s=1/2+it$. Then by the Binomial Theorem
       $$ |A(s)+\chi(s) A(1-s) |^{2k}={\binom{2k}{k}} |A(s)|^{2k}+ 2 \sum_{\nu=1} ^{k} {\binom{2k}{k+\nu}} \Re ( |A(s)|^{2k-2\nu} A(s)^{2\nu} \chi(1-s)^\nu). $$
\end{prop}

This gives us two series of integrals to calculate. For convenience we introduce the following further notation:

\begin{align*}
    \mathcal{I}_{k,\nu}(\alpha_1,\alpha_2)&=\int_{-\infty}^{\infty}\zeta(\tfrac{1}{2}+\alpha_1+it)\zeta(\tfrac{1}{2}+\alpha_2-it)|A(\tfrac{1}{2}+it)|^{2k-2\nu}\\&\ \ \ \ \ \ \ \ \ \times A(\tfrac{1}{2}+it)^{2\nu}\chi(\tfrac{1}{2}-it)^\nu\Phi(\tfrac{t}{T})\;dt,
    \\
\mathcal{J}_{k,\nu}(\alpha_1,\alpha_2,\alpha_3,\alpha_4)&=\int_{-\infty}^{\infty}\zeta(\tfrac{1}{2}+\alpha_1+it)\zeta(\tfrac{1}{2}+\alpha_2+it)\zeta(\tfrac{1}{2}+\alpha_3-it)\zeta(\tfrac{1}{2}+\alpha_4-it)\\& \ \ \ \ \ \ \ \ \ \times|A(\tfrac{1}{2}+it)|^{2k-2\nu}A(\tfrac{1}{2}+it)^{2\nu}\chi(\tfrac{1}{2}-it)^\nu\Phi(\tfrac{t}{T})\;dt.
\end{align*}

In both cases, the index $\nu$ counts the number of $\chi(1-s)$ factors appearing within the integrands. We work first with smooth weights $\Phi(t/T)$ supported in $[1/4,2]$, with the derivative bounds stated in Lemma~\ref{thm:young_lem_5}, and write $\widehat\Phi(0)=\int_{\mathbb R}\Phi(u)\,du$.  The passage to sharp intervals is as follows.  For any fixed dyadic interval one chooses smooth minorants and majorants of its characteristic function, with total mass differing from the interval length by $o(T)$.  The signed pieces $\mathcal I_{k,\nu}$ and $\mathcal J_{k,\nu}$ are first recombined using Proposition~\ref{prop:binomial_expansion}; the resulting full amplified moments are non-negative, so the minorant is used for the numerator and the majorant for the denominator.  The asymptotic formulae below are uniform for these weights and differ only by the factor $\widehat\Phi(0)$.  The transition width is fixed while $T\to\infty$, then sent to zero, and only then are the dyadic intervals summed.  This is the standard smoothing argument used implicitly in \cite{page} and \cite{BM}.

By Proposition \ref{prop:binomial_expansion}, the same binomial decomposition applies to the sharp integrals $I_k,J_k$ and, after inserting the weight $\Phi(t/T)$, to the smoothed pieces $\mathcal{I}_{k,\nu},\mathcal{J}_{k,\nu}$.

Our first step is to dispose of the integrals with large powers of $\chi(1-s)$ appearing in the integrand. This is a natural consequence of the rapid oscillation of $\chi(1-s)$, which is amplified by introducing high powers. This oscillation forces cancellation within our integrals. This is consistent with the prediction of the Ratios Conjecture for $L$-functions \cite{conrey2005integralmomentslfunctions}.

\begin{thm}\label{thm:O(T)_ge_2,4}
Uniformly in the shifts $\alpha_1,\alpha_2,\alpha_3,\alpha_4$, in the ranges $\theta_k<1/(2k)$ and $\vartheta_k<1/(4k)$ respectively, we have
    
\begin{align*}
    \mathcal{I}_{k,\nu}(\alpha_1,\alpha_2)&=O(T^{1-\varepsilon}) \hspace{5mm}(2\leq\nu\leq k),
    \\
    \mathcal{J}_{k,\nu}(\alpha_1,\alpha_2,\alpha_3,\alpha_4)&=O(T^{1-\varepsilon}) \hspace{5mm}(3\leq\nu\leq k).
\end{align*}

\end{thm}

Due to this first theorem, there are only five distinct integrals that we need to evaluate, as all of the other integrals which appear from Proposition \ref{prop:binomial_expansion} are lower order terms. These are addressed in the following five theorems.

The following theorems all hold uniformly for shifts of size $O(1/\log T)$.  Formulae containing apparent poles in the shifts are first obtained for distinct generic shifts and then extended to the stated range by holomorphic continuation; expressions with shift denominators are interpreted through these removable continuations before specialisation or differentiation.  The error terms are established uniformly on the boundaries of fixed polydiscs of radius $\asymp 1/\log T$, and Cauchy's formula preserves these bounds after continuation.

\begin{thm}\label{thm:Ik0}

We have, for $\theta_k < \frac{1}{2k}$,

\begin{align*}
\mathcal{I}_{k,0}(\alpha_1,\alpha_2)&= a_{k+1}T \log T \, (\log y)^{k^2+2k}\,\widehat{\Phi}(0)  \int_0^1 \int_{\substack{
0 \leq t_{i,j} \leq 1 \\0 \leq x_i, w_i \leq 1 }} 
\int_{\substack{
\sum_{j=1}^k t_{i,j} +w_i\le 1 \\
\sum_{i=1}^k t_{i,j} +x_j \le 1 }} 
\\
&\qquad \times y^{-\alpha_1\sum_{i=1}^k w_i-\alpha_2\sum_{i=1}^k x_i} \left(1-\theta_k  \sum_{i=1}^{k} (w_i +x_i)  \right)
\left(Ty^{-\sum_{i=1}^{k} (w_i +x_i)}\right)^{-v (\alpha_1+\alpha_2)}
\\
&\qquad \times
\prod_{i =1}^k
P\left(1-\sum_{j=1}^k t_{i,j} -w_i \right)
\prod_{j=1}^k
P\left(1-\sum_{i=1}^k t_{i,j} -x_j \right)
\,
\\
&\qquad 
dv \, 
dt_{1,1}\cdots dt_{k,k} \, 
dw_1 \cdots dw_{k} \,
dx_1 \cdots dx_k   
+ O\left(T(\log T)^{(k+1)^2-1}\right).
\end{align*}
\end{thm}

\begin{thm}\label{thm:Ik1}
We have, for $\theta_k < \frac{1}{2k}$,
\begin{align*}
\mathcal{I}_{k,1}(\alpha_1,\alpha_2)
&=
a_{k+1}T
\widehat{\Phi}(0)
\left(\frac{T}{2\pi}\right)^{-\alpha_1}
(\log y)^{(k+1)^2} 
\int_{\substack{
0 \leq t_{i,j} \leq 1 \\
0 \leq x_i, w_i \leq1 }}
\int_{\substack{
\sum_{j=1}^{k+1} t_{i,j} \le 1 \\
\sum_{i =1}^{k-1}t_{i,j} + w_j + x_j \le 1
}}
\\
& \times
y^{\alpha_1\sum_{i=1}^{k+1}w_i-\alpha_2\sum_{i=1}^{k+1} x_i}
\prod_{i=1}^{k-1}
P\left(1-\sum_{j=1}^{k+1} t_{i,j}  \right)
\prod_{j =1}^{k+1}
P\left(1-\sum_{i=1}^{k-1} t_{i,j}  -w_j - x_j \right)
\\
& dt_{1,1}\cdots dt_{k-1,k+1} \,
dw_1 \cdots dw_{k+1} \,
dx_1 \cdots dx_{k+1}
+ O\left(T(\log T)^{(k+1)^2 - 1}\right).
\end{align*}
\end{thm}

\begin{thm}\label{thm:Jk0}

We have, for $\vartheta_k < \frac{1}{4k}$ 

\begin{align*}
&\mathcal{J}_{k,0}(\alpha_1,\alpha_2,\alpha_3,\alpha_4)\\
&=
\frac12 \widehat{\Phi}(0)\,a_{k+2}T(\log T)^4(\log y)^{k^2+4k} \\
&\quad \times
\int_{[0,1]^4} \int_{\substack{
0 \leq t_{i,j}\leq 1\\
0 \leq  w_i,x_i,y_i,z_i \leq 1
}}
\int_{\substack{
\sum_{j=1}^k t_{i,j} + w_{i}+x_{i} \le 1\\
\sum_{i =1}^k t_{i,j}  + y_{j} +z_{j} \le 1
}} y^{-\alpha_3\sum_{i=1}^k w_i-\alpha_4\sum_{i=1}^k x_i-\alpha_1\sum_{i=1}^k y_i-\alpha_2\sum_{i=1}^k z_i}
\\
& \qquad \times  \left(Ty^{-\sum_{i=1}^k(w_i+y_i)}\right)^{-(\alpha_1+\alpha_3)v_1} \left(
\frac{
y^{\sum_{i=1}^k y_i-\sum_{i=1}^k z_i}
\left(Ty^{-\sum_{i=1}^k(w_i+y_i)}\right)^{v_1}
}{
\left(Ty^{-\sum_{i=1}^k (x_i+z_i)}\right)^{v_2}
}
\right)^{-(\alpha_2-\alpha_1)v_3}
\\
& \qquad  \times \left(Ty^{-\sum_{i=1}^k (x_i+z_i)}\right)^{-(\alpha_2+\alpha_4)v_2} \left(
\frac{
y^{\sum_{i=1}^kw_i-\sum_{i=1}^k x_i}
\left(Ty^{-\sum_{i=1}^k(w_i+y_i)}\right)^{v_1}
}{
\left(Ty^{-\sum_{i=1}^k(x_i+z_i)}\right)^{v_2}
}
\right)^{-(\alpha_4-\alpha_3)v_4} 
\\
&\qquad \times
\biggl(
\vartheta_k\left(\sum_{i=1}^k w_i-\sum_{i=1}^k x_i \right)
+v_1\left(1-\vartheta_k\sum_{i=1}^k (w_i +y_i)\right)
-v_2\left(1-\vartheta_k\sum_{i=1}^k (x_i+z_i)\right)
\biggr) 
\\
&\qquad \times
\biggl(
\vartheta_k\left(\sum_{i=1}^k y_i-\sum_{i=1}^k z_i\right)
+v_1\left(1-\vartheta_k\sum_{i=1}^k(w_i+y_i)\right)
-v_2\left(1-\vartheta_k\sum_{i=1}^k(x_i+z_i)\right)
\biggr) \\
&\qquad \times
\left(1-\vartheta_k\sum_{i=1}^k (w_i+y_i)\right)
\left(1-\vartheta_k\sum_{i=1}^k (x_i+z_i)\right)
\\
&\qquad \times
\prod_{i =1}^k
P\left(1- \sum_{j =1}^k t_{i,j}  -w_{i}-x_{i} \right)
\prod_{j=1}^k
P\left(1-\sum_{i=1}^k t_{i,j} - y_{j}-z_{j}\right)
\\
&\qquad \times
dv_1 dv_2 dv_3 dv_4 \, 
dt_{1,1}\cdots dt_{k,k} \,
dw_1 \cdots dw_k \,
dx_1 \cdots dx_k \,
dy_1 \cdots dy_k \,
dz_1 \cdots dz_k \\
&\qquad + O\left(T(\log T)^{(k+2)^2-1}\right).
\end{align*}
\end{thm}

\begin{thm}\label{thm:Jk1}
We have, for $k\ge1$ and $\vartheta_k < \frac{1}{4k}$, with the convention that all sums, products and differentials over $1\leq j\leq k-1$ are omitted when $k=1$,
\begin{align*}
\mathcal{J}_{k,1}(\alpha_1,\alpha_2,\alpha_3,\alpha_4)
&=
a_{k+2}T\widehat{\Phi}(0)(\log y)^{k^2+4k+1} 
\int_{\substack{
0\leq t_{i, j} \le1 \ \\
0\le x_i, y_i,z_i\leq 1\ (1\leq i\leq k+1)\\
0\le w_j\leq 1\ (1\leq j\leq k-1)
}}
\int_{\substack{
\sum_{j =1}^{k-1} t_{i, j} +x_i +y_i +z_i \leq 1 \\
\sum_{i=1}^{k+1} t_{i, j} + w_j \le1 
}}
\\
& \qquad \times 
\prod_{i=1}^{k+1}
P\left(1-\sum_{j=1}^{k-1} t_{i, j} -x_i-y_i-z_i \right)
\prod_{j=1}^{k-1}
P\left(1-\sum_{i=1}^{k+1} t_{i, j} -w_j \right)
\\
&\qquad \times
\Biggl[
\left(\frac{T}{2\pi}\right)^{-\alpha_1-\alpha_3-\alpha_4}
\frac{
y^{\alpha_1\sum_{i=1}^{k-1} w_i
-\alpha_2\sum_{i=1}^{k+1} x_i
+\alpha_3\sum_{i=1}^{k+1} y_i
+\alpha_4\sum_{i=1}^{k+1} z_i}
}{
(\alpha_2-\alpha_1)(\alpha_1+\alpha_3)(\alpha_1+\alpha_4)
} \\
&\qquad\qquad +
\left(\frac{T}{2\pi}\right)^{-\alpha_2-\alpha_3-\alpha_4}
\frac{
y^{\alpha_2\sum_{i=1}^{k-1} w_i
-\alpha_1\sum_{i=1}^{k+1} x_i
+\alpha_3\sum_{i=1}^{k+1} y_i
+\alpha_4\sum_{i=1}^{k+1} z_i}
}{
(\alpha_1-\alpha_2)(\alpha_2+\alpha_3)(\alpha_2+\alpha_4)
} \\
&\qquad\qquad +
\left(\frac{T}{2\pi}\right)^{-\alpha_4}
\frac{
y^{-\alpha_3\sum_{i=1}^{k-1} w_i
-\alpha_2\sum_{i=1}^{k+1} x_i
-\alpha_1\sum_{i=1}^{k+1} y_i
+\alpha_4\sum_{i=1}^{k+1} z_i}
}{
(\alpha_2+\alpha_3)(\alpha_1+\alpha_3)(\alpha_3-\alpha_4)
} \\
&\qquad\qquad +
\left(\frac{T}{2\pi}\right)^{-\alpha_3}
\frac{
y^{-\alpha_4\sum_{i=1}^{k-1} w_i
-\alpha_2\sum_{i=1}^{k+1} x_i
+\alpha_3\sum_{i=1}^{k+1} y_i
-\alpha_1\sum_{i=1}^{k+1} z_i}
}{
(\alpha_2+\alpha_4)(\alpha_4-\alpha_3)(\alpha_1+\alpha_4)
}
\Biggr] \\
&\qquad
dt _{1,1}\cdots dt _{k+1,k-1}\,
dx_1 \cdots dx_{k+1} \,
dw_1 \cdots dw_{k-1} \,
dy_1 \cdots dy_{k+1} \,
dz_1 \cdots dz_{k+1} \, \\
&\qquad + O\left(T(\log T)^{(k+2)^2-1}\right).
\end{align*}
\end{thm}

\begin{thm}\label{thm:Jk2}
We have, for $k\geq2$ and $\vartheta_k < \frac{1}{4k}$,
where for $k=2$ all products, sums and integrals indexed by $1\leq j\leq k-2$ are interpreted as empty,
\begin{align*}
\mathcal{J}_{k,2}(\alpha_1,\alpha_2,\alpha_3,\alpha_4)
&=
a_{k+2}T(\log y)^{(k+2)^2}\widehat{\Phi}(0)
\left(\frac{T}{2\pi}\right)^{-\alpha_1-\alpha_2}  \int_{\substack{
0 \leq t_{i,j} \leq 1 \\
0 \leq w_i,x_i,y_i,z_i \leq 1 }}
\\
&\qquad \times
\int_{\substack{
\sum_{j =1}^{k-2} t_{i,j} + w_i +x_i +y_i +z_i \leq 1 \\
\sum_{i=1}^{k+2} t_{i,j} \leq 1 }} y^{\alpha_1 \sum_{i=1}^{k+2} w_i+\alpha_2 \sum_{i=1}^{k+2} x_i-\alpha_3 \sum_{i=1}^{k+2} y_i-\alpha_4 \sum_{i=1}^{k+2} z_i}
\\
& \qquad \times \prod_{i =1}^{k+2}
P\left(1-\sum_{j =1}^{k-2} t_{i,j} - w_i-x_i-y_i-z_i \right)  \prod_{j =1}^{k-2}
P\left(1-\sum_{i =1}^{k+2} t_{i,j} \right)
 \\
&\qquad \times
dt_{1,1}\cdots dt_{k+2,k-2} \,
dw_1 \cdots dw_{k+2} \,
dx_1 \cdots dx_{k+2} \,
dy_1 \cdots dy_{k+2} \,
dz_1 \cdots dz_{k+2} \\
&\qquad + O\left(T(\log T)^{(k+2)^2-1}\right).
\end{align*}
\end{thm}

\subsection{Structure of the paper}
In Section \ref{section:preparatory_lemmas} we collect the lemmas used throughout the paper. In Section \ref{section:initial_reductions} we prove Proposition \ref{prop:binomial_expansion} and Theorem \ref{thm:O(T)_ge_2,4}. Section \ref{section:second_moments} proves the amplified second moment results of Theorem \ref{thm:Ik0} and \ref{thm:Ik1}. Section \ref{section:amplified_fourth} proves the amplified fourth moment results of Theorem \ref{thm:Jk0}, \ref{thm:Jk1} and \ref{thm:Jk2}.  In Section \ref{section:corollaries} we deduce the numerical results of Theorems \ref{cor:6thmoment}--\ref{thm:2204}. In Section \ref{sect:sound_unconditional} we prove Theorem \ref{thm:sound_bounds_unconditional} and in Section \ref{sect:joint_lower} we prove Theorem \ref{thm:joint_lower}.

\section{Preparatory Lemmas}\label{section:preparatory_lemmas}

Our first lemma enables us to recast a discrete sum into a polytope integral.

\begin{lemma}\label{t_c}
Let $k,\ell\geq 1$ be fixed integers, let $y\geq 2$, and put $L=\log y$.
Let $f_1,\ldots,f_k,g_1,\ldots,g_{\ell}\in C^1([0,1])$. Then
\begin{align*}
& \sum_{\substack{
m_{i,j}\geq 1\\
\prod_{j=1}^{\ell}m_{i,j}\leq y \hspace{2mm} (1\leq i\leq k)\\
\prod_{i=1}^{k}m_{i,j}\leq y \hspace{2mm} (1\leq j\leq \ell)
}}
\left(\prod_{i=1}^{k}\prod_{j=1}^{\ell}\frac{1}{m_{i,j}}\right)
\prod_{i=1}^{k}
f_i\bigg(
\frac{\log\left(y/\prod_{j=1}^{\ell}m_{i,j}\right)}{\log y}
\bigg)
\prod_{j=1}^{\ell}
g_j\bigg(
\frac{\log\left(y/\prod_{i=1}^{k}m_{i,j}\right)}{\log y}
\bigg)
\\
&\quad =
(\log y)^{k\ell}
\int_{\substack{
0\leq t_{i,j}\leq 1\\
\sum_{j=1}^{\ell}t_{i,j}\leq 1 \hspace{2mm} (1\leq i\leq k)\\
\sum_{i=1}^{k}t_{i,j}\leq 1 \hspace{2mm} (1\leq j\leq \ell)
}}
\prod_{i=1}^{k}
f_i\bigg(1-\sum_{j=1}^{\ell}t_{i,j}\bigg)
\prod_{j=1}^{\ell}
g_j\bigg(1-\sum_{i=1}^{k}t_{i,j}\bigg)
dt_{1,1} \cdots dt_{k,\ell}
\\
&\quad\quad
+O_{k,\ell}\left(
\left(\prod_{i=1}^{k}\|f_i\|_{C^1}\right)
\left(\prod_{j=1}^{\ell}\|g_j\|_{C^1}\right)
(\log y)^{k\ell-1}
\right),
\end{align*}
where $\|h\|_{C^1}=\|h\|_{\infty}+\|h'\|_{\infty}$. In particular, if the
functions $f_i$ and $g_j$ are fixed, the error term is
$O_{k,\ell,f_i,g_j}((\log y)^{k\ell-1})$.
\end{lemma}

\begin{proof}
Set
$$
u_{i,j}=\frac{\log m_{i,j}}{\log y}.
$$
Then the conditions $\prod_{j=1}^{\ell}m_{i,j}\leq y$ and
$\prod_{i=1}^{k}m_{i,j}\leq y$ are equivalent to
$$
\sum_{j=1}^{\ell}u_{i,j}\leq 1,
\qquad
\sum_{i=1}^{k}u_{i,j}\leq 1.
$$
Thus the discrete points $u_{i,j}$ lie in the polytope
$$
\mathcal R=
\left\{
(t_{i,j})\in[0,1]^{k\ell}:
\sum_{j=1}^{\ell}t_{i,j}\leq 1 \hspace{2mm} (1\leq i\leq k),
\quad
\sum_{i=1}^{k}t_{i,j}\leq 1 \hspace{2mm} (1\leq j\leq \ell)
\right\}.
$$
Define
$$
F((t_{i,j}))=
\prod_{i=1}^{k}
f_i\bigg(1-\sum_{j=1}^{\ell}t_{i,j}\bigg)
\prod_{j=1}^{\ell}
g_j\bigg(1-\sum_{i=1}^{k}t_{i,j}\bigg).
$$
On $\mathcal R$ we have
$$
\|F\|_{\infty}+\mathrm{Lip}(F)
\ll_{k,\ell}
\left(\prod_{i=1}^{k}\|f_i\|_{C^1}\right)
\left(\prod_{j=1}^{\ell}\|g_j\|_{C^1}\right).
$$

We shall use the following one-dimensional form of partial summation. If
$0\leq a\leq 1$ and $\phi$ is Lipschitz on $[0,a]$, then
$$
\sum_{n\leq y^a}\frac{1}{n}\phi\left(\frac{\log n}{\log y}\right)
=
(\log y)\int_0^a\phi(t)dt
+
O\left(\|\phi\|_{\infty}+\mathrm{Lip}(\phi)\right),
$$
uniformly in $a$. To prove this, write
$$
H(x)=\sum_{n\leq x}\frac{1}{n},
\qquad
E(x)=H(x)-\log x.
$$
The function $E(x)$ is bounded for $x\geq 1$. Hence, with $Y=y^a$,
$$
\sum_{n\leq Y}\frac{1}{n}\phi\left(\frac{\log n}{L}\right)
=
\int_{1^-}^{Y}\phi\left(\frac{\log x}{L}\right)dH(x).
$$
The contribution of $d\log x$ is
$$
\int_1^Y\phi\left(\frac{\log x}{L}\right)\frac{dx}{x}
=
L\int_0^a\phi(t)dt.
$$
The contribution of $dE(x)$ is, by integration by parts,
$$
\ll
\|\phi\|_{\infty}
+
\int_1^Y
\mathrm{Lip}(\phi)\frac{dx}{Lx}
\ll
\|\phi\|_{\infty}+\mathrm{Lip}(\phi),
$$
which proves the estimate.

We now apply this estimate successively to the $k\ell$ variables $m_{i,j}$.
At any stage of the iteration, once the other variables have been fixed, the
conditions involving the current variable $m_{i,j}$ give an interval
$$
1\leq m_{i,j}\leq y^a
$$
with $0\leq a\leq 1$. More explicitly, the exponent $a$ is the minimum of the
remaining row allowance and the remaining column allowance. The one-dimensional
estimate above is therefore applicable uniformly. The weight in the current
variable is obtained from $F$ by fixing or integrating the other variables over
sections of $\mathcal R$; these sections have endpoints given by minima of
finitely many affine functions, so the resulting one-variable weight is
Lipschitz with norm
$$
\ll_{k,\ell}
\left(\prod_{i=1}^{k}\|f_i\|_{C^1}\right)
\left(\prod_{j=1}^{\ell}\|g_j\|_{C^1}\right).
$$

Each main term replaces one sum over $m_{i,j}$ by the integral
$L\,dt_{i,j}$. After all $k\ell$ variables have been treated, this gives
$$
L^{k\ell}\int_{\mathcal R}F((t_{i,j}))
dt_{1,1} \cdots dt_{k, \ell}.
$$
This is exactly the main term in the statement.

It remains only to record the size of the accumulated error. Suppose the
one-dimensional error occurs when $r-1$ variables have already been replaced by
integrals. The preceding main terms have contributed at most a factor
$L^{r-1}$. The error from the current one-dimensional summation is
$$
O_{k,\ell}\left(
\left(\prod_{i=1}^{k}\|f_i\|_{C^1}\right)
\left(\prod_{j=1}^{\ell}\|g_j\|_{C^1}\right)
\right).
$$
The remaining $k\ell-r$ variables are bounded by harmonic sums, since
each remaining variable is at most $y$. Therefore their total contribution is
$$
\ll_{k,\ell} L^{k\ell-r}.
$$
Thus the error produced at this stage is
$$
O_{k,\ell}\left(
\left(\prod_{i=1}^{k}\|f_i\|_{C^1}\right)
\left(\prod_{j=1}^{\ell}\|g_j\|_{C^1}\right)
L^{k\ell-1}
\right).
$$
There are only $k\ell$ stages, so the total error has the same order. This
proves the lemma.
\end{proof}
In order to successfully bound various error terms, we shall need the following standard estimates for the sizes of some functions which will appear throughout the paper.

\begin{lemma}\label{lemma:bounds}

We have the following estimates.

    \begin{enumerate}
        \item In the critical strip,
        
        \begin{equation*}
        \zeta(\sigma + it) \ll
        \begin{cases}
            T^{\frac{1}{2} - \frac{2\sigma}{3}} 
            \hspace{1cm} \text{for} \hspace{2mm} \sigma \in [0 , 1/2]
            \\
            \\
             T^{\frac{1-\sigma}{3}} 
            \hspace{1.2cm} \text{for} \hspace{2mm} \sigma \in [1/2 , 1]
        \end{cases}
        \end{equation*}
        
        \item For $\sigma$ fixed and $t$ large,
        
        \begin{equation*}
            \chi(\sigma+it)=\left(\frac{t}{2\pi}\right)^{1/2-\sigma-it}e^{it+i\pi/4}\left(1+O\left(\frac{1}{t}\right)\right).
        \end{equation*}
        
        \item For complex shifts $\alpha,\beta$
        
        \begin{equation*}
            \chi(s+\alpha)\chi(1-s+\beta)=\left(\frac{t}{2\pi}\right)^{-(\alpha+\beta)}\left(1+O\left(\frac{1}{t}\right)\right).
        \end{equation*}
        
        \item The truncated polynomial sum may be bounded
        \begin{equation*}
            \sum_{n\le y}\frac{P[n]}{n^{\sigma+it}}\ll 1+y^{1-\sigma+\varepsilon}.
        \end{equation*}
        
    \end{enumerate}
\end{lemma}

\begin{proof}
    Parts (1)--(3) follow from the classical Weyl bound, the functional equation, and Stirling's formula for $\chi$, as in \cite[Chs. 4--5]{titchmarsh}; Part (4) follows by partial summation and an integral comparison.
\end{proof}

Next, we require a lemma which enables us to evaluate a smoothed twisted second moment.

\begin{lemma}[\cite{YNG}, Lemma 5]\label{thm:young_lem_5}

Suppose $\Phi(t/T)$ is a smooth weight satisfying the following conditions;

\begin{enumerate}
    \item $0\le\Phi(t/T)\le 1$ for all $t\in\mathbb{R}$,
    \item $\Phi(t/T)$ is compactly supported in $[1/4,2]$,
    \item $\Phi^{(j)}(t/T)\ll_j(T/\log T)^{-j}$ for each $j\in\mathbb{Z}_{\ge 0}$.
\end{enumerate}

For positive integers $h',k'$ with $h'k'\le T^{2\theta}$ with $\theta<1/2$ and $\alpha,\beta\ll 1/\log T$ we have

\begin{align*}
   & \int_{-\infty}^{\infty}\left(\frac{h'}{k'}\right)^{-it}\zeta(\tfrac{1}{2}+\alpha+it)\zeta(\tfrac{1}{2}+\beta-it)\Phi\left(\frac{t}{T}\right)\;dt\\&=\sum_{h'm=k'n}\frac{1}{m^{\frac{1}{2}+\alpha}n^{\frac{1}{2}+\beta}}\int_{-\infty}^{\infty}V_{\alpha,\beta}(mn,t)\Phi\left(\frac{t}{T}\right)\;dt\\&+\sum_{h'm=k'n}\frac{1}{m^{\frac{1}{2}-\beta}n^{\frac{1}{2}-\alpha}}\int_{-\infty}^{\infty}V_{-\beta,-\alpha}(mn,t)\chi_{\alpha,\beta}(t)\Phi\left(\frac{t}{T}\right) + O_{A,\theta}(T^{-A}),
\end{align*}
where 
$$V_{\alpha,\beta}(x,t)=\frac{1}{2\pi i}\int_{(1)}\frac{G(s)}{s}g_{\alpha,\beta}(s,t)x^{-s}\;ds$$
and 
$$g_{\alpha,\beta}(s,t)=\pi^{-s}\frac{\Gamma\left(\frac{1/2+\alpha+s+it}{2}\right)\Gamma\left(\frac{1/2+\beta+s-it}{2}\right)}{\Gamma\left(\frac{1/2+\alpha+it}{2}\right)\Gamma\left(\frac{1/2+\beta-it}{2}\right)}.$$
    
\end{lemma}

We shall also need a lemma which allows us to rewrite the product of zeta factors as a weighted Dirichlet series. The following lemma generalises \cite[Lemma 4.1]{Bui_2011} which considered the product of two zeta factors.

\begin{lemma}\label{lemma:bcy_sigma} 

    Given a set $A=\{\alpha_1,\ldots,\alpha_k\}$, let
    $$
        \sigma_A(\ell)=\sum_{\ell=a_1\cdots a_k}a_1^{-\alpha_1}\cdots a_k^{-\alpha_k}.
    $$
    For $(\log T)^2 \le|t|\le 2T$ and uniformly for $\alpha_j \ll (\log T)^{-1}$,
    
    \begin{equation*}
        \prod_{j=1}^{k} \zeta(\tfrac{1}{2}+\alpha_j+it)=
        \sum_{\ell=1}^{\infty}\frac{\sigma_{A}(\ell)}{\ell^{1/2+it}}e^{-\ell/T^{k+1}}+O(T^{-1+\varepsilon}).
    \end{equation*}
    
\end{lemma}

\vspace{5mm}

\begin{proof}
    Consider the sum

    $$ S = \sum_{\ell=1}^{\infty}\frac{\sigma_{A}(\ell)}{\ell^{1/2+it}}e^{-\ell/T^{k+1}}.$$

    \
    Put $V=T^{k+1}$. Writing the exponential factor using the inverse Mellin transform we have

    $$ S = \frac{1}{2 \pi i} \int_{(1)} V^s \Gamma(s) \prod_{j=1}^{k} \zeta(1/2 + \alpha_j + s + it) ds, $$
    where as usual $\int_{(1)}$ is used to denote the line integral $\int_{1-i\infty}^{1+i\infty}$. We first argue for distinct shifts; when shifts coincide, the nonzero-pole contribution is interpreted by the small contour $D$ below. We now move the line of integration to $\sigma = -1 + \varepsilon$, encountering poles at $s=0$ and $s = 1/2 - \alpha_j - it$ for every value of $j$. It follows that
    
    $$ S = \prod_{j=1}^{k} \zeta(1/2 + \alpha_j  + it) + A_1 + A_2 $$
    
    where
    \begin{align*}
        A_1&=\sum_{j=1}^{k} T^{(k+1)(1/2-\alpha_j -it)} \Gamma(1/2 - \alpha_j -it) \prod_{\substack{\ell=1 \\ \ell \neq j}}^{k} \zeta(1 + (\alpha_{\ell} - \alpha_j)),\\
        A_2&=\frac{1}{2 \pi i} \int_{(-1 + \varepsilon)} T^{s(k+1)} \Gamma(s) \prod_{j=1}^{k} \zeta(1/2 + \alpha_j + s + it) ds.
    \end{align*}

We handle $A_2$ first. Writing $s=-1+\varepsilon+iu$ we see that $\Re(1/2+\alpha_j+it+s)=-1/2+\varepsilon+\alpha_j<0$ for large $T$. We now invoke the convexity bound in the form $\zeta(\sigma+i\tau)\ll_{\varepsilon}(1+|\tau|)^{1/2-\sigma+\varepsilon}$ for $\sigma<0$, which gives
    
    $$\zeta(1/2+\alpha_j+it+s)\ll_{\varepsilon}(1+|t+u|)^{1+O(\varepsilon)}.$$
    
    By Stirling's approximation,
    
    $$\Gamma(-1+\varepsilon+iu)\ll e^{-\pi |u|/2}(1+|u|)^{-3/2+\varepsilon}$$
    
    so that
    
    $$A_2\ll_{\varepsilon}V^{-1+\varepsilon}\int_{-\infty}^{\infty}e^{-\pi|u|/2}(1+|u|)^{O(1)}(1+|t+u|)^{k+\varepsilon}du\ll_{\varepsilon}V^{-1+\varepsilon}(1+|t|)^{k+\varepsilon}.$$
    
    Since $|t|\le 2T$ and $V=T^{k+1}$ we have
    
    $$A_2\ll_{\varepsilon}T^{-(k+1)+\varepsilon}T^{k+\varepsilon}\ll_{\varepsilon}T^{-1+\varepsilon}.$$
    
    For the nonzero poles we can let $c>0$ be such that $|\alpha_j|\le c/\log T$ for all $j$, and let $D$ be the small circle
    
    $$D=\left\{s:|s-1/2+it|=3c/\log T\right\}.$$
    
    All poles $s_j=1/2-\alpha_j-it$ then lie in $D$. By Cauchy's Residue Theorem, $A_1$ equals the integral of the same integrand over $D$ divided by $2\pi i$. 
    
    On $D$ we have $\Re(s)=1/2+O(1/\log T)$ so $|V^s|\ll V^{1/2}$, and by Stirling $\Gamma(s)\ll e^{-\pi|t|/2}$ uniformly for $s\in D$. Also for each $i$, 
    
    $$|1/2+\alpha_i+it+s-1|=|s-(1/2-\alpha_i-it)|\asymp 1/\log T$$
    
    so uniformly on $D$ we have 
    
    $$\zeta(1/2+\alpha_i+it+s)\ll \log T.$$ 
    
    Since the length of $D$ is $\asymp 1/\log T$ we have
    
    $$A_1\ll V^{1/2}e^{-\pi|t|/2}(\log T)^k(\log T)^{-1}\ll V^{1/2}e^{-\pi|t|/2}(\log T)^{k-1}.$$
    
    Since $V=T^{k+1}$ and $|t|\ge (\log T)^2$ we have $A_1\ll_A T^{-A}$ for any fixed $A>0$. Therefore $A_1\ll T^{-1+\varepsilon}$. Combining the estimates for $A_1$ and $A_2$ completes the proof.
\end{proof}

Our next lemma handles the smoothed twisted fourth moment, and may be seen as an analogue of Lemma \ref{thm:young_lem_5} for the fourth moment.

\begin{lemma}
[\cite{BBLR}, Theorem 1.2]\label{bmthm5.1}
Let $G(s)$ be an even entire function of rapid decay in any fixed strip $|\Re(s)|\le C$ satisfying $G(0)=1$, and let

\begin{equation*}
    V(x)=\frac{1}{2\pi i}\int_{(1)}G(s)(2\pi)^{-2s}x^{-s}\frac{ds}{s}.
\end{equation*}

Let $y=T^\theta$.  Then, for $T$ large and $0<\theta<1/4$, we have

\begin{align*}
    &\sum_{h,k\le y}\frac{a_h\overline{a_k}}{\sqrt{hk}}\int_{-\infty}^{\infty}\zeta(\tfrac{1}{2}+\alpha_1+it)\zeta(\tfrac{1}{2}+\alpha_2+it)\zeta(\tfrac{1}{2}+\beta_1-it)\zeta(\tfrac{1}{2}+\beta_2-it)\left(\frac{h}{k}\right)^{-it}\Phi\left(\frac{t}{T}\right)\;dt
    \\
    &=\sum_{h,k\le y}\frac{a_h\overline{a_k}}{\sqrt{hk}}\int_{-\infty}^{\infty}\Phi\left(\frac{t}{T}\right)\{Z_{\alpha_1,\alpha_2,\beta_1,\beta_2,h,k}(t)+\left(\frac{t}{2\pi}\right)^{-(\alpha_1+\beta_1)}Z_{-\beta_1,\alpha_2,-\alpha_1,\beta_2,h,k}(t)
    \\
    &+\left(\frac{t}{2\pi}\right)^{-(\alpha_1+\beta_2)}Z_{-\beta_2,\alpha_2,\beta_1,-\alpha_1,h,k}(t)+\left(\frac{t}{2\pi}\right)^{-(\alpha_2+\beta_1)}Z_{\alpha_1,-\beta_1,-\alpha_2,\beta_2,h,k}(t)
    \\
    &+\left(\frac{t}{2\pi}\right)^{-(\alpha_2+\beta_2)}Z_{\alpha_1,-\beta_2,\beta_1,-\alpha_2,h,k}(t)+\left(\frac{t}{2\pi}\right)^{-(\alpha_1+\alpha_2+\beta_1+\beta_2)}Z_{-\beta_1,-\beta_2,-\alpha_1,-\alpha_2,h,k}(t)
    \}\;dt
    \\
    &+O_{\varepsilon}\left(T^{\frac{1}{2}+2\theta+\varepsilon}+T^{\frac{3}{4}+\theta+\varepsilon}\right)
\end{align*}

    uniformly for $\alpha_1,\alpha_2,\beta_1,\beta_2\ll (\log T)^{-1}$, where $\varepsilon>0$ and 
    
    \begin{equation*}
        Z_{\alpha,\beta,\gamma,\delta,h,k}(t)=\sum_{hm_1m_2=kn_1n_2}\frac{1}{m_1^{1/2+\alpha}m_2^{1/2+\beta}n_1^{1/2+\gamma}n_2^{1/2+\delta}}V\left(\frac{m_1m_2n_1n_2}{t^2}\right).
    \end{equation*}
    
\end{lemma}

\begin{lemma}\label{bmthm4.1}

    Let $j\ge 0$ and $1\le n<y$ be integers, and let
    
    \begin{equation}\label{eqn:K_j}
        K_j(\alpha_1,\ldots,\alpha_k)=\frac{1}{2\pi i}\int_{(1/\log T)}\left(\frac{y}{n}\right)^u\prod_{\ell=1}^k\zeta(1+\alpha_{\ell}+u)\frac{du}{u^{j+1}}
    \end{equation}
    
    Then we have
    
    \begin{multline*}
        K_j(\alpha_1,\ldots,\alpha_k)=\frac{(\log(y/n))^{j+k}}{j!}
        \int_{\substack{x_{\ell}\ge0\\ \sum x_{\ell}\le 1}}
        \left(\frac{y}{n}\right)^{-\sum \alpha_{\ell} x_{\ell}} (1-\sum x_{\ell})^j\;dx_1\cdots dx_k +O\left((\log y)^{j+k-1}\right)
    \end{multline*}
    
    uniformly for $|\alpha_{\ell}|\ll (\log y)^{-1}$.
\end{lemma}

\vspace{5mm}

\begin{proof}
    Set $q=y/n>1$.  Replacing each $\zeta(1+\alpha_\ell+u)$ by its principal part gives
    $$
        K_j(\alpha_1,\ldots,\alpha_k)
        =\frac{1}{2\pi i}\int_{(c)} q^u\prod_{\ell=1}^{k}\frac{1}{u+\alpha_\ell}\frac{du}{u^{j+1}}
        +O((\log y)^{j+k-1}),
    $$
    where $c\asymp1/\log y$ is chosen to the right of the poles.  Every term containing at least one regular part of a zeta-factor has one fewer pole at the origin.  For the main term we use
    $$
        \frac{1}{u+\alpha_\ell}=\int_{0}^{\infty}e^{-(u+\alpha_\ell)w_\ell}\,dw_\ell
    $$
    on a line to the right of the poles, with the general shifted case obtained by analytic continuation.  Interchanging the absolutely convergent integrals reduces the inner integral to Perron's formula,
    $$
        \frac{1}{2\pi i}\int_{(c)} e^{u(\log q-\sum w_\ell)}\frac{du}{u^{j+1}}
        =
        \begin{cases}
            \frac{(\log q-\sum w_\ell)^j}{j!},& \sum w_\ell\le \log q,\\
            0,& \sum w_\ell>\log q.
        \end{cases}
    $$
    The change of variables $w_\ell=(\log q)x_\ell$ gives the asserted simplex integral.
\end{proof}

We quote two lemmas in the form used in \cite[Lemmas 3.2 and 3.3]{page}, following the method of \cite{cgg}.

\begin{lemma}[\cite{cgg}, Lemma 2]\label{lemma:page_3.2}
    Given two series
    \begin{equation*}
        A(s)=\sum_{n=1}^{\infty}\frac{a(n)}{n^s}, \ \ \ \ \ \ \ B(s)=\sum_{n\le y}\frac{b(n)}{n^s}
    \end{equation*}
    with coefficients satisfying $a(n)\ll d_i(n)(\log n)^j$ and $b(n)\ll d_k(n)(\log n)^{\ell}$ with non-negative $i,j,k,\ell$, a value $T^{\varepsilon}\ll y\ll T$ and $c:=1+1/\log T$, we have
    \begin{multline*}
        \frac{1}{2\pi i}\int_{c+i}^{c+iT}\chi(1-s)A(s)B(1-s)\;ds=\sum_{n\le y}\frac{b(n)}{n}\sum_{m\le\frac{nT}{2\pi}}a(m)e\left(-\frac{m}{n}\right)+O(T^{1/2}y(\log T)^{i+j+k+\ell}).
    \end{multline*}
\end{lemma}

\begin{lemma}[\cite{cgg}, Lemma 3]\label{lemma:page_3.3}
    Given a Dirichlet series
    \begin{equation*}
        A(s)=\sum_{n=1}^{\infty}\frac{a(n)}{n^s}
    \end{equation*}
    which may be decomposed into the form
    \begin{equation*}
        A(s)=\prod_{j=1}^JA_j(s)
    \end{equation*}
    with \begin{equation*}
        A_j(s)=\sum_{m=1}^{\infty}\frac{a_j(m)}{m^s}
    \end{equation*}
    for some natural number $J\ge 1$, then for any integer $d>0$ and any completely multiplicative function $f$, we have
    \begin{equation*}
        \sum_{m=1}^{\infty}\frac{a(md)f(m)}{m^s}=\sum_{d=d_1\cdots d_J}\prod_{j=1}^J\sum_{(m,d_1\cdots d_{j-1})=1}^{\infty}\frac{a_j(md_j)f(m)}{m^s}.
    \end{equation*}
\end{lemma}
We shall also require a version of Perron's formula with an error term suitable for our applications. The following is a variant of Perron's formula, given by \cite[Lemma 4]{cgg}.

\begin{lemma}\label{lem:perron}
    Suppose that $A(s)=\sum_{n=1}^{\infty}a(n)n^{-s}$ for $\sigma>1$, where
    $$|a(n)|\le Kd_k(n)(\log n)^{\ell}.$$
    Then for any $\varepsilon>0$ we have
    \begin{equation*}
        R:=\sum_{n\le x}a(n)-\frac{1}{2\pi i}\int_{c-iU}^{c+iU}A(s)\frac{x^s}{s}\;ds\ll_{\varepsilon}Kx^{\varepsilon}\left(1+\frac{x}{U}\right),
    \end{equation*}
    where $c=1+1/\log x$ and the implied constant is independent of $K$. If $U\ll x^{1-\varepsilon}$ for some $\varepsilon>0$, then
    \begin{equation*}
        R\ll_{\varepsilon}K\frac{x}{U}(\log x)^{k+\ell}.
    \end{equation*}
\end{lemma}

\section{Some initial reductions}\label{section:initial_reductions}
In this section we prove Proposition \ref{prop:binomial_expansion} and Theorem \ref{thm:O(T)_ge_2,4}.

\subsection{Proof of Proposition \ref{prop:binomial_expansion}}
    We have
\begin{equation*}
    |A(s)+\chi(s)A(1-s)|^{2k}=(A(s)+\chi(s)A(1-s))^k(A(1-s)+\chi(1-s)A(s))^k.
\end{equation*}

Applying the binomial theorem to each factor we see
\begin{equation*}
    |A(s)+\chi(s)A(1-s)|^{2k}=\sum_{r=0}^k\sum_{q=0}^k\binom{k}{r}\binom{k}{q}\chi(s)^r\chi(1-s)^qA(s)^{k-r+q}A(1-s)^{r+k-q}.
\end{equation*}

Denoting the summand 

$$B(r,q)=\binom{k}{r}\binom{k}{q}\chi(s)^r\chi(1-s)^qA(s)^{k-r+q}A(1-s)^{r+k-q}.$$

\vspace{3mm}
\
We may deal with the diagonal terms $B(r,r)$ and the off-diagonals separately. For the diagonal term, 

\begin{align*}B(r,r)&=\binom{k}{r}^2\chi(s)^r\chi(1-s)^rA(s)^kA(1-s)^k=\binom{k}{r}^2|\chi(s)|^{2r}|A(s)|^{2k}.\end{align*}

But on the critical line $|\chi(s)|=1$ so the diagonals contribute 

\begin{equation*}\sum_{r=0}^kB(r,r)=|A|^{2k}\sum_{r=0}^k\binom{k}{r}^2=|A|^{2k}\binom{2k}{k},
\end{equation*}
having applied Vandermonde's identity.

For the off-diagonal contribution, we have $r\ne q$, $\overline{B(r,q)}=B(q,r)$ so that 

\begin{equation*}
B(r,q)+B(q,r)=2\Re(B(r,q))
\end{equation*}

and thus 

\begin{equation*}
\sum_{\substack{0\le r,q\le k\\ r\ne q}}B(r,q)=2\sum_{0\le r<q\le k}\Re(B(r,q)).
\end{equation*}

Fix an integer $\nu$ with $1\le \nu\le k$ and consider the pairs with $q=r+\nu$ with $0\le r\le k-\nu$. For such a pair,

\begin{align*}
B(r,r+\nu)=\binom{k}{r}\binom{k}{r+\nu}\chi(s)^r\chi(1-s)^{r+\nu}A(s)^{k-r+(r+\nu)}A(1-s)^{r+k-(r+\nu)},
\end{align*}

which simplifies to 

\begin{align*}
B(r,r+\nu)&=\binom{k}{r}\binom{k}{r+\nu}A(s)^{k+\nu}A(1-s)^{k-\nu}\chi(1-s)^\nu \\&=\binom{k}{r}\binom{k}{r+\nu}|A(s)|^{2k-2\nu}A(s)^{2\nu}\chi(1-s)^\nu.
\end{align*}

Summing over $r$ gives

\begin{equation*}\sum_{r=0}^{k-\nu}B(r,r+\nu)=\left(\sum_{r=0}^{k-\nu}\binom{k}{r}\binom{k}{r+\nu}\right)|A(s)|^{2k-2\nu}A(s)^{2\nu}\chi(1-s)^\nu.
\end{equation*}

Hence the contribution to $|A(s)+\chi(s)A(1-s)|^{2k}$ from distance $\nu$ is

\begin{align*}
& 2\Re\left(\left(\sum_{r=0}^{k-\nu}\binom{k}{r}\binom{k}{r+\nu}\right)|A(s)|^{2k-2\nu}A(s)^{2\nu}\chi(1-s)^\nu\right)
\\ 
= & 2\binom{2k}{k+\nu}\Re\left(|A(s)|^{2k-2\nu}A(s)^{2\nu}\chi(1-s)^\nu\right),
\end{align*}

once again applying Vandermonde's identity to simplify the binomial coefficient.

\subsection{Proof of Theorem \ref{thm:O(T)_ge_2,4}}
We consider our integrals on the intervals $[T,2T]$, and the result for $[1,T]$ follows by summing over dyadic intervals.  We first record the second-moment case, since the right edge of the contour is slightly different.  Let
$$
F_\nu(s)=\zeta(1-s+\alpha_1)\zeta(s+\alpha_2)\chi(s)^\nu A(1-s)^{k+\nu}A(s)^{k-\nu}.
$$
Taking $H=T^{1/3}$, the two end pieces on the critical line are
$$
\ll y^k\left(\int_T^{T+H}+\int_{2T-H}^{2T}\right)
|\zeta(\tfrac12+\alpha_1+iu)\zeta(\tfrac12+\alpha_2-iu)|\,du
\ll y^k(H+T^{1/2})(\log T)^C\ll T^{1-\varepsilon},
$$
by Cauchy--Schwarz, the classical short-interval second moment estimate (see, for example, \cite{Ivic}), and
$\theta_k<1/(2k)$; the shifted form again follows from Cauchy's formula in discs of radius $\asymp1/\log T$.  On the horizontal pieces, for $1/2\leq\sigma\leq1$ the
functional equation and Lemma~\ref{lemma:bounds} give the exponent
$$
        \frac16+\frac{\sigma}{3}
        +\nu\left(\frac12-\sigma\right)
        +\theta_k(k-\nu+2\nu\sigma),
$$
which is decreasing for $\nu\ge2$ and is at most $1/3+k\theta_k<1$ at
$\sigma=1/2$.  For $\sigma\ge1$, the exponent is
$$
        (\nu-1)\left(\frac12-\sigma\right)+\theta_k\sigma(k+\nu),
$$
which is $<1$ at $\sigma=1$ and decreases thereafter, since
$\theta_k(k+\nu)<\nu-1$.  Taking the right edge at a fixed sufficiently large
$\mu>1$ makes this last exponent negative, and hence the vertical edge is
$O(T^{1-\varepsilon})$.  This proves
$\mathcal I_{k,\nu}=O(T^{1-\varepsilon})$ for $2\leq\nu\leq k$.

We now turn to $\mathcal{J}_{k,\nu}(\alpha_1,\alpha_2,\alpha_3,\alpha_4)$.

We observe that 
\begin{align*}
    \mathcal{J}_{k,\nu}(\underline{\alpha})=&\int_T^{2T}\zeta(\tfrac{1}{2}+\alpha_1+it)\zeta(\tfrac{1}{2}+\alpha_2+it)\zeta(\tfrac{1}{2}+\alpha_3-it)\zeta(\tfrac{1}{2}+\alpha_4-it)
    \\
    & \qquad \times A(\tfrac{1}{2}+it)^{k+\nu}A(\tfrac{1}{2}-it)^{k-\nu}\chi(\tfrac{1}{2}-it)^\nu\;dt.
\end{align*}

With $s=1/2-it$ we have $dt=i\;ds$, and as $t$ runs from $T$ to $2T$, $s$ runs from $1/2-iT$ to $1/2-2iT$. Reversing orientation, we have
\begin{equation*}
    \mathcal{J}_{k,\nu}(\underline{\alpha})=\frac{1}{i}\int_{1/2-2iT}^{1/2-iT}F_\nu(s)\;ds
\end{equation*}
where
\begin{equation*}
    F_\nu(s)=\zeta(1-s+\alpha_1)\zeta(1-s+\alpha_2)\zeta(s+\alpha_3)\zeta(s+\alpha_4)\chi(s)^\nu A(1-s)^{k+\nu}A(s)^{k-\nu}.
\end{equation*}
Choose a parameter $H=T^{3/4}$, and set $t_0=T+H, t_1=2T-H$. By Cauchy's Theorem, for any fixed $\mu\in (1/2,1)$ we have
\begin{align*}
    \int_{1/2-2iT}^{1/2-iT}F_\nu(s)\;ds&=\left(\int_{1/2-2iT}^{1/2-it_1}+\int_{1/2-it_1}^{\mu-it_1}+\int_{\mu-it_1}^{\mu-it_0}+\int_{\mu-it_0}^{1/2-it_0}+\int_{1/2-it_0}^{1/2-iT}\right)F_\nu(s)\;ds\\&:=J_1+J_2+J_3+J_4+J_5.
\end{align*}
We start with $J_1$ and $J_5$. On $\Re(s)=1/2$ we have $|\chi(s)|=1$ and 
\begin{equation*}
    |A(1-s)^{k+\nu}A(s)^{k-\nu}|=|A(\tfrac{1}{2}+it)|^{2k}\ll y^kT^{\varepsilon}.
\end{equation*}
Therefore
\begin{equation*}
    |J_1|+|J_5|\ll y^k\left(\int_T^{T+H}+\int_{2T-H}^{2T}\right)\prod_{m=1}^4|\zeta(\tfrac{1}{2}+\beta_m+iu)|\;du,
\end{equation*}
where each $\beta_m$ is one of $\alpha_1,\alpha_2,\overline{\alpha_3},\overline{\alpha_4}$, so $|\beta_m|\ll 1/\log T$. By the AM-GM inequality we have that for non-negative $a_1,a_2,a_3,a_4$,
\begin{equation*}
    \prod_{i=1}^4a_i\le\frac{1}{4}\sum_{i=1}^4a_i^4,
\end{equation*}
from which it follows that
\begin{equation*}
    |J_1|+|J_5|\ll y^k\sum_{m=1}^4\left(\int_T^{T+H}|\zeta(\tfrac{1}{2}+\beta_m+iu)|^4\;du+\int_{2T-H}^{2T}|\zeta(\tfrac{1}{2}+\beta_m+iu)|^4\;du\right).
\end{equation*}
For $H=T^{3/4}$, the standard short-interval fourth moment estimate \cite[Chapter 8]{Ivic} gives
\begin{equation*}
    \int_T^{T+H}|\zeta(\tfrac{1}{2}+\beta+iu)|^4\;du\ll H(\log T)^4
\end{equation*}
uniformly for $|\beta|\ll 1/\log T$; the shifted form follows from the unshifted estimate by Cauchy's formula in discs of radius $\asymp1/\log T$. Thus
\begin{equation*}
    |J_1|+|J_5|\ll y^kH(\log T)^4=T^{\vartheta_kk+3/4+o(1)}\ll T^{1-\varepsilon}.
\end{equation*}
We remark that it is precisely at this point that we are forced to take $\theta_k<\tfrac{1}{4k}$, since otherwise this error term would dominate the main term.

We turn now to the horizontal pieces $J_2$ and $J_4$. We focus on the former since the argument for $J_4$ is identical. Write $s=\sigma-it_1$ with $1/2\le\sigma\le\mu$. Then by Lemma \ref{lemma:bounds} we see that
$$|F_\nu(\sigma-it_1)|\ll T^{E(\sigma)+\varepsilon}$$
where
\begin{equation*}
    E(\sigma)=\frac{1}{3}+\frac{2\sigma}{3}+\nu\left(\frac{1}{2}-\sigma\right)+\vartheta_k(k-\nu+2\nu\sigma).
\end{equation*}
Since $E'(\sigma)=2/3-\nu+2\nu\vartheta_k<0$ for every $3\leq\nu\leq k$ and $\vartheta_k<1/(4k)$, we have that $E$ is decreasing on $[1/2,\mu]$. Thus
\begin{equation*}
    E(\sigma)\le E(1/2)=\frac{2}{3}+k\vartheta_k<\frac{11}{12}<1.
\end{equation*}
Hence
\begin{equation*}
    |J_2|\ll\int_{1/2}^{\mu}T^{E(\sigma)+\varepsilon}\;d\sigma\ll T^{2/3+\vartheta_kk+\varepsilon}\ll T^{11/12+\varepsilon}.
\end{equation*}
Similarly $J_4\ll T^{1-\varepsilon}$. 

Finally we must handle $J_3$. On $s=\mu-it$ we have $t_0\le t\le t_1$, where the bounds as in $J_2$ give 
$$|F_\nu(\mu-it)|\ll T^{E(\mu)+\varepsilon}$$
with the same $E(\mu)$ as above. Therefore
$$|J_3|\ll T\cdot T^{E(\mu)+\varepsilon},$$
so it will suffice to choose $\mu\in (1/2,1)$ such that $E(\mu)<0$. But $E(1)=1-\nu/2+\vartheta_k(k+\nu)<0$ for $3\leq\nu\leq k$, so by continuity of $E$ there exists a fixed $\mu\in (1/2,1)$ with $E(\mu)<0$. For this choice, $$J_3\ll T^{1-\varepsilon}$$
for some $\delta>0$. This completes the proof.

\vspace{8mm}

\section{Amplified second moments}\label{section:second_moments}

\subsection{Proof of Theorem \ref{thm:Ik0}}

As in Lemma 5 of \cite{YNG} we define

$$ V_{\alpha_1, \alpha_2}(x, t) := \frac{1}{2 \pi i} \int_{(2)} \frac{G(s)}{s} g_{\alpha_1,\alpha_2}(s,t)  x^{-s} ds $$

\
for a suitable damping factor $G(s)$. A permissible choice here is to let

$$ G(s) = e^{s^2} \frac{(\alpha_1 + \alpha_2)^2 - (2s)^2}{(\alpha_1 + \alpha_2)^2}.$$

\
We first argue for generic shifts, so that the auxiliary factor is defined and the indicated poles are distinct.  The estimates are uniform on the boundaries of polydiscs of radius $\asymp1/\log T$; possible powers of $(\log T)$ coming from the auxiliary factor are absorbed by the final error term.  The final formula follows for all shifts in the stated range by holomorphic continuation and Cauchy's formula.

\
We further define

$$g_{\alpha_1,\alpha_2}(s,t)=\pi^{-s}\frac{\Gamma\left(\tfrac{1/2+\alpha_1+s+it}{2}\right)\Gamma\left(\tfrac{1/2+\alpha_2+s-it}{2}\right)}{\Gamma\left(\tfrac{1/2+\alpha_1+it}{2}\right)\Gamma\left(\tfrac{1/2+\alpha_2-it}{2}\right)}$$

\
and

$$\chi_{\alpha_1,\alpha_2}(t)=\pi^{\alpha_1+\alpha_2}\frac{\Gamma\left(\tfrac{1/2-\alpha_1-it}{2}\right)\Gamma\left(\tfrac{1/2-\alpha_2+it}{2}\right)}{\Gamma\left(\tfrac{1/2+\alpha_1+it}{2}\right)\Gamma\left(\tfrac{1/2+\alpha_2-it}{2}\right)}.$$

With these definitions, an application of Lemma \ref{thm:young_lem_5} yields (taking $h'=a_1\cdots a_k$ and $k'=b_1\cdots b_k$)

\begin{align*}
    \mathcal{I}_{k,0}(\alpha_1,\alpha_2)&=\sum_{\substack{a_1,\cdots,a_k\le y\\ b_1,\cdots,b_k\le y}}\frac{P[a_1]\cdots P[a_k]P[b_1]\cdots P[b_k]}{\sqrt{a_1\cdots a_kb_1\cdots b_k}}\\&\times\bigg(\sum_{a_1\cdots a_km=b_1\cdots b_kn}\frac{1}{m^{1/2+\alpha_1}n^{1/2+\alpha_2}}\int_{-\infty}^{\infty}V_{\alpha_1,\alpha_2}(mn,t)\Phi(\tfrac{t}{T})\;dt\\&+\sum_{a_1\cdots a_km=b_1\cdots b_kn}\frac{1}{m^{1/2-\alpha_2}n^{1/2-\alpha_1}}\int_{-\infty}^{\infty}V_{-\alpha_2,-\alpha_1}(mn,t)\chi_{\alpha_1,\alpha_2}(t)\Phi(\tfrac{t}{T})\;dt\bigg)\\&+O_{A,\theta}\left(T^{-A}\right).
\end{align*}

\
This splits $\mathcal{I}_{k,0}(\alpha_1,\alpha_2)$ into three parts,

$$
\mathcal{I}_{k,0}(\alpha_1,\alpha_2)=\mathcal{I}_{k,0,1}(\alpha_1,\alpha_2)+\mathcal{I}_{k,0,2}(\alpha_1,\alpha_2)+O_{A,\theta}(T^{-A}).$$

We observe that the second piece is the same as the first upon the permutation $(\alpha_1,\alpha_2)\leftrightarrow (-\alpha_2,-\alpha_1)$ together with the introduction of the $\chi_{\alpha_1,\alpha_2}(t)$ term. As such we begin by considering $\mathcal{I}_{k,0,1}(\alpha_1,\alpha_2)$ and later show how one can handle $\mathcal{I}_{k,0,2}(\alpha_1,\alpha_2)$ similarly and combine the pieces. We have

\begin{align*}
    \mathcal{I}_{k,0,1}(\alpha_1,\alpha_2)&=\sum_{\substack{a_1,\cdots,a_k\le y\\ b_1,\cdots,b_k\le y}}\frac{P[a_1]\cdots P[a_k]P[b_1]\cdots P[b_k]}{\sqrt{a_1\cdots a_kb_1\cdots b_k}}\sum_{a_1\cdots a_km=b_1\cdots b_kn}\frac{1}{m^{1/2+\alpha_1}n^{1/2+\alpha_2}}\\&\times\frac{1}{2\pi i}\int_{-\infty}^{\infty}\int_{(2)}\frac{G(s)}{s}g_{\alpha_1,\alpha_2}(s,t)(mn)^{-s}\Phi(\tfrac{t}{T})\;ds\;dt.
\end{align*}

\
By the Mellin inversion formula we may write

\begin{equation}\label{mellin_polynomial}
    P[n]=\sum_j\frac{c_jj!}{(\log y)^j}\frac{1}{2\pi i}\int_{(2)}\left(\frac{y}{n}\right)^u\frac{du}{u^{j+1}}.
\end{equation}

We remark that $P[n]=0$ for $n\ge y$, which will later assist us in forcing a truncation of resulting Dirichlet series. Substituting \eqref{mellin_polynomial} for each occurrence of $P[a_i], P[b_j]$ enables us to rewrite $\mathcal{I}_{k,0,1}(\alpha_1,\alpha_2)$ as

\begin{align*}
    \mathcal{I}_{k,0,1}(\alpha_1,\alpha_2)&=\sum_{i_1,\cdots,i_k}\sum_{j_1,\cdots,j_k}\frac{c_{i_1}\cdots c_{i_k}c_{j_1}\cdots c_{j_k}i_1!\cdots i_k!j_1!\cdots j_k!}{(\log y)^{i_1+\cdots+i_k+j_1+\cdots+j_k}}
    \frac{1}{(2\pi i)^{1+2k}}\int_{-\infty}^{\infty}\int_{(2)^{1+2k}}\frac{G(s)}{s}g_{\alpha_1,\alpha_2}(s,t)
    \\
    &\times \Phi(\tfrac{t}{T}) \sum_{a_1\cdots a_km=b_1\cdots b_kn}\frac{y^{u_1+\cdots+u_k+v_1+\cdots+v_k}}{m^{1/2+\alpha_1+s}n^{1/2+\alpha_2+s}a_1^{1/2+u_1}\cdots a_k^{1/2+u_k}b_1^{1/2+v_1}\cdots b_k^{1/2+v_k}}
    \\
    &\frac{du_1}{u_1^{i_1+1}}\cdots\frac{du_k}{u_k^{i_k+1}}\cdots\frac{dv_1}{v_1^{j_1+1}}\cdots\frac{dv_k}{v_k^{j_k+1}}\;ds\;dt.
\end{align*}

We now consider the inner sum, which due to the multiplicative constraint $a_1\cdots a_km=b_1\cdots b_kn$ gives rise to an Euler product $\mathcal{A}(\underline{u},\underline{v},\alpha_1,\alpha_2,s)$ which converges absolutely in a product of half-planes containing the origin. In particular, we have

\begin{align*}
   & \sum_{a_1\cdots a_km=b_1\cdots b_kn}\frac{1}{m^{1/2+\alpha_1+s}n^{1/2+\alpha_2+s}a_1^{1/2+u_1}\cdots a_k^{1/2+u_k}b_1^{1/2+v_1}\cdots b_k^{1/2+v_k}}\\&=\zeta(1+\alpha_1+\alpha_2+2s)\prod_{j=1}^k\zeta(1+\alpha_1+s+v_j)\prod_{i=1}^k\zeta(1+\alpha_2+s+u_i)\\&\times \prod_{i=1}^{k} \prod_{j=1}^{k} \zeta(1+u_i+v_j)\mathcal{A}(\underline{u},\underline{v},\alpha_1,\alpha_2,s).    
\end{align*}

\
Substituting this back into the expression for $\mathcal{I}_{k,0,1}(\alpha_1,\alpha_2)$ gives

\begin{align*}
\mathcal{I}_{k,0,1}(\alpha_1, \alpha_2) =  &  \sum_{i_1,\ldots,i_k} \sum_{j_1,\ldots,j_k} \frac{c_{i_1}\cdots c_{i_k} c_{j_1}\cdots c_{j_k} i_1!\cdots i_k! j_1!\cdots j_k!}{ (\log y)^{i_1+\cdots+i_k+j_1+\cdots+j_k}}\\&\times \frac{1}{(2\pi i)^{1 + 2k}}\int_{-\infty}^{\infty} \int_{(2)^{1+2k}} \frac{G(s)}{s} g_{\alpha_1, \alpha_2}(s,t)  \Phi(\tfrac{t}{T}) 
\\
 &  \times y^{u_1+\cdots+u_k+v_1+\cdots+v_k} \mathcal{A}(\underline{u},\underline{v}, \alpha_1, \alpha_2, s) \zeta(1+\alpha_1 + \alpha_2 + 2s) \prod_{j=1}^k \zeta(1 + \alpha_1 + s + v_j) 
 \\
 & \times\prod_{i=1}^k \zeta(1 + \alpha_2 + s + u_i)  \prod_{i=1}^{k} \prod_{j=1}^{k} \zeta(1+ u_i + v_j) \frac{du_1}{u_1^{i_1+1}}\cdots\frac{du_k}{u_k^{i_k+1}}\frac{dv_1}{v_1^{j_1+1}}\cdots \frac{dv_k}{v_k^{j_k+1}}\;ds\;dt.
\end{align*}

We first move the $u_i$- and $v_j$-contours to $\Re(u_i)=\Re(v_j)=\delta$, and then move the $s$-contour to $\Re(s)=-(1-\varepsilon)\delta$. The poles from $\zeta(1+\alpha_1+s+v_j)$ and $\zeta(1+\alpha_2+s+u_i)$ remain to the left of the new $s$-line, since $\Re u_i=\Re v_j=\delta$. We therefore encounter only the candidate poles at $s=0$ and $s=-\frac{(\alpha_1+\alpha_2)}{2}$. However, the latter potential pole is cancelled by construction of our function $G(s)$, leaving only the pole at $s=0$. It follows then that $\mathcal{I}_{k,0,1}(\alpha_1,\alpha_2)=\Res_{s=0}+I_{-(1-\varepsilon)\delta}$.

We may bound the contribution of this line integral by $I_{-(1-\varepsilon)\delta}\ll T^{1-\varepsilon}$ by appealing to Stirling's approximation for the $g_{\alpha_1,\alpha_2}(s)$ term and then bounding by absolute values. In particular, noting that by, for example, \cite[Equation 4.2]{YNG},

\begin{equation*}
    g_{\alpha_1,\alpha_2}(s,t)=\left(\frac{t}{2\pi}\right)^s\left(1+O\left(\frac{1+|s|^2}{t}\right)\right)
\end{equation*}

\
and substituting into our above expression yields, upon taking absolute values, the bound

$$I_{-(1-\varepsilon)\delta}\ll T^{1-(1-\varepsilon)\delta}y^{2k\delta}
=T^{1-\delta(1-2k\theta_k+O(\varepsilon))}\ll T^{1-\varepsilon},$$
after choosing the small parameters with $\theta_k<1/(2k)$.

\
Therefore after additionally integrating over $t$ we have

\begin{align*}
\mathcal{I}_{k,0,1}(\alpha_1, \alpha_2) =  & \widehat{\Phi}(0)T\zeta(1+\alpha_1+\alpha_2) \sum_{i_1,\cdots,i_k} \sum_{j_1,\cdots, j_k} \frac{c_{i_1} \cdots c_{i_k} c_{j_1}\cdots c_{j_k} i_1!\cdots i_k! j_1!\cdots j_k!}{(\log y)^{i_1+\cdots+i_k+j_1+\cdots+j_k}} \\&\times\frac{1}{(2\pi i)^{2k}}\int_{(\delta)^{2k}} 
   y^{u_1+\cdots+u_k+v_1+\cdots+v_k} \mathcal{A}(\underline{u}, \underline{v}, \alpha_1, \alpha_2,0) \prod_{j=1}^k \zeta(1 + \alpha_1 + v_j)  \\
 & \times\prod_{i=1}^k \zeta(1 + \alpha_2 + u_i) \prod_{i=1}^{k} \prod_{j=1}^{k} \zeta(1+ u_i + v_j) \\& \frac{du_1}{u_1^{i_1+1}}\cdots\frac{du_k}{u_k^{i_k+1}}\frac{dv_1}{v_1^{j_1+1}}\cdots \frac{dv_k}{v_k^{j_k+1}}+O(T^{1-\varepsilon}).
\end{align*}

With a view towards eliminating entangled variables in the $i,j$-product to integrate over them, we now shift the contours again, this time to the line $\Re(u_i)=\Re(v_j)\asymp (\log T)^{-1}$. We first truncate the vertical integrals at height a fixed power of $T$, using the usual decay of the Mellin transforms, and then deform inside a fixed small polydisc about the origin. The local boundedness of $\mathcal A$, together with absolute convergence on these contours, shows that at the cost of an error term of size $O(T(\log T)^{(k+1)^2-1})$ we may replace $\mathcal{A}(\underline{u},\underline{v},\alpha_1,\alpha_2,0)$ by $\mathcal{A}(\underline{0},\underline{0},0,0,0)$. To be precise, we have

\begin{align*}
    \mathcal{A}(\underline{u},\underline{v},\alpha_1,\alpha_2,0)&=\mathcal{A}(\underline{0},\underline{0},0,0,0)+O\left(\sum_{r=1}^k|u_r|+|v_r|\right)+O(|\alpha_1|+|\alpha_2|)
    \\
    & =\mathcal{A}(\underline{0},\underline{0},0,0,0)+O\left(\frac{1}{\log T}\right),
\end{align*}

\
the error term following immediately from our imposition on the sizes of the shifts and the fact that the $u_i, v_j$ have real parts $1/\log T$ due to our contour shift. After integration this will yield an error term of size one logarithm smaller than the main term, that is $O(T(\log T)^{(k+1)^2-1})$.

We may now deduce the value of $\mathcal{A}(\underline{0},\underline{0},0,0,0)$ from the Euler product in which $\mathcal{A}$ arose. In particular, since $\mathcal{A}$ is holomorphic with respect to small shifts $\alpha_1,\alpha_2$ near the origin, upon setting $u_i=v_j=s$ and $\alpha_1=\alpha_2=0$ we have

\begin{align*}
   & \sum_{a_1\cdots a_km=b_1\cdots b_kn}\frac{1}{(mna_1\cdots a_kb_1\cdots b_k)^{1/2+s}}
   =\zeta(1+2s)^{(k+1)^2} \mathcal{A}(\underline{s},\underline{s},0, 0,s).    
\end{align*}

\
From this we see that

\begin{align*}
   \mathcal{A}(\underline{0},\underline{0},0, 0, 0) = \lim_{s \to 0} \frac{1}{\zeta(1+2s)^{(k+1)^2}} \sum_{n}\frac{d_{k+1}(n)^2}{n^{1+2s}} = a_{k+1}
 .    
\end{align*}

\
Therefore we may now write

\begin{align*}
\mathcal{I}_{k,0,1}(\alpha_1, \alpha_2) =  & a_{k+1}\widehat{\Phi}(0)T\zeta(1+\alpha_1+\alpha_2) \sum_{i_1,\cdots,i_k} \sum_{j_1,\cdots,j_k} \frac{c_{i_1} \cdots c_{i_k} c_{j_1}\cdots c_{j_k} i_1!\cdots i_k! j_1!\cdots j_k!}{(\log y)^{i_1+\cdots+i_k+j_1+\cdots+j_k}} 
\\
&\times\frac{1}{(2\pi i)^{2k}}\int_{(1/\log T)^{2k}} 
y^{u_1+\cdots+u_k+v_1+\cdots+v_k} \prod_{i=1}^k \zeta(1 + \alpha_2 + u_i)\prod_{j=1}^k\zeta(1+\alpha_1+v_j) 
\\
&\times \prod_{i=1}^{k} \prod_{j=1}^{k} \zeta(1+ u_i + v_j)
 \frac{du_1}{u_1^{i_1+1}}\cdots\frac{du_k}{u_k^{i_k+1}}\frac{dv_1}{v_1^{j_1+1}}\cdots \frac{dv_k}{v_k^{j_k+1}}+O\left(T(\log T)^{(k+1)^2-1}\right).
\end{align*}

We now write each zeta factor with entangled $u_i, v_j$ variables in terms of their absolutely convergent Dirichlet series. Interchanging the order of summation and integration restricts the sizes of the $m_{i, j}$ since we may shift our integrals far to the right. Indeed, we have 

\begin{align*}
\mathcal{I}_{k,0,1}(\alpha_1, \alpha_2) =  & a_{k+1} T \widehat{\Phi}(0) \zeta(1+\alpha_1+\alpha_2)  \sum_{i_1,\cdots,i_k} \sum_{j_1,\cdots,j_k} \frac{c_{i_1}\cdots c_{i_k} c_{j_1} \cdots c_{j_k} i_1!\cdots i_k! j_1!\cdots j_k!}{(\log y)^{i_1+\cdots+i_k+j_1+\cdots+j_k}} 
\\ 
& \times\frac{1}{(2\pi i)^{2k}}\int_{(1/\log T)^{2k}}
y^{u_1+\cdots+u_k+v_1+\cdots+v_k}  \prod_{j=1}^k \zeta(1 + \alpha_1 + v_j) \prod_{i=1}^k \zeta(1 + \alpha_2 + u_i) 
\\
&\times \prod_{r=1}^{k}\prod_{s=1}^{k}
\sum_{m_{r,s}=1}^{\infty}\frac{1}{m_{r,s}^{1+u_r+v_s}}
\\
&  \frac{du_1}{u_1^{i_1+1}}\cdots\frac{du_k}{u_k^{i_k+1}}\frac{dv_1}{v_1^{j_1+1}}\cdots \frac{dv_k}{v_k^{j_k+1}} + O\left(T (\log T)^{(k+1)^2 -1}\right).
\end{align*} 

\
Collecting $m_{i,j}$ terms into those corresponding to a power of $u_i$ or $v_j$ and matching these with a corresponding power of $y$ we have

\begin{align*}
\mathcal{I}_{k,0,1}(\alpha_1, \alpha_2) =  & a_{k+1} T \widehat{\Phi}(0) \zeta(1+\alpha_1+\alpha_2)  \sum_{i_1,\cdots,i_k} \sum_{j_1,\cdots,j_k} \frac{c_{i_1} \cdots c_{i_k} c_{j_1}\cdots c_{j_k} i_1!\cdots i_k! j_1!\cdots j_k!}{(\log y)^{i_1+\cdots+i_k+j_1+\cdots+j_k}}  
\\ 
&  \times\sum_{\substack{ \prod_{j=1}^k m_{i,j} \leq y \\ \prod_{i=1}^k m_{i,j} \leq y }} \bigg( \prod_{i=1}^{k} \prod_{j=1}^{k}  \frac{1}{m_{i,j}} \bigg)\frac{1}{(2\pi i)^{2k}}\int_{(1/\log T)^{2k}}
\prod_{j=1}^k \zeta(1 + \alpha_1 + v_j) \prod_{i=1}^k \zeta(1 + \alpha_2 + u_i)   
\\
&  \times\prod_{i=1}^k \bigg(\frac{y}{\prod_{j=1}^k m_{i,j}}\bigg)^{u_i}  \prod_{j=1}^k \bigg(\frac{y}{\prod_{i=1}^k m_{i,j}}\bigg)^{v_j}\frac{du_1}{u_1^{i_1+1}}\cdots\frac{du_k}{u_k^{i_k+1}}\frac{dv_1}{v_1^{j_1+1}}\cdots \frac{dv_k}{v_k^{j_k+1}} 
\\
& + O\left(T (\log T)^{(k+1)^2 -1}\right).
\end{align*} 

We now decompose the integral into $2k$ separate one-dimensional blocks. To this end we define

$$K_x(n):= \frac{1}{2\pi i} \sum_j \frac{c_j j!}{(\log y)^j}  \int_{(1/\log T)} \bigg( \frac{y}{n} \bigg)^u \zeta(1+x+u) \frac{du}{u^{j+1}}.$$

\
Applying Lemma~\ref{bmthm4.1} with one shift, equivalently the formula used in \cite[p.~337]{page}, gives

\begin{equation*} 
K_x(n)=\sum_j \frac{c_j(\log(y/n))^{j+1}}{(\log y)^j}   \int_{0}^{1}  \bigg( \frac{y}{n} \bigg)^{-w x}(1-w)^j   dw+O(1),
\end{equation*} 

\
uniformly for $n<y$.  Terms in which one $K$-block is replaced by this error are absorbed in the existing $O(T(\log T)^{(k+1)^2-1})$ error term, so the main term may be substituted into our expression to yield

\begin{align*}
\mathcal{I}_{k,0,1}(\alpha_1, \alpha_2) =  & a_{k+1} T \widehat{\Phi}(0) \zeta(1+\alpha_1+\alpha_2)  \sum_{i_1,\cdots,i_k} \sum_{j_1,\cdots,j_k} \frac{c_{i_1} \cdots c_{i_k} c_{j_1}\cdots c_{j_k}}{(\log y)^{i_1+\cdots+i_k+j_1+\cdots+j_k}} 
\\ 
&  \times  \sum_{\substack{ \prod_{j=1}^k m_{i,j} \leq y \\ \prod_{i=1}^k m_{i,j} \leq y }} \bigg( \prod_{i=1}^{k} \prod_{j=1}^{k}  \frac{1}{m_{i,j}} \bigg) \int_{[0,1]^{2k}} \prod_{\ell=1}^k \bigg(\frac{y}{ \prod_{j=1}^k m_{\ell,j}
}\bigg)^{- \alpha_1 w_{\ell}} \log\bigg(\frac{y}{ \prod_{j=1}^k m_{\ell,j}
}\bigg)^{i_\ell + 1} 
\\
&  \times (1-w_{\ell})^{i_\ell} \prod_{\ell=1}^k \bigg(\frac{y}{\prod_{i=1}^k m_{i,\ell}}\bigg)^{- \alpha_2 x_{\ell}} \log\bigg(\frac{y}{\prod_{i=1}^k m_{i,\ell}}\bigg)^{j_{\ell} + 1} (1-x_{\ell})^{j_\ell}  \\& dw_1\cdots dw_{k} dx_1\cdots dx_k + O\left(T (\log T)^{(k+1)^2 -1}\right).
\end{align*}

At this point we may apply Lemma \ref{t_c} and interchange the order of summation and integration, valid by Fubini's Theorem, to obtain

\begin{align*}
\mathcal{I}_{k,0,1}(\alpha_1, \alpha_2) =  & a_{k+1} T \log(y)^{2k + k^2}  \widehat{\Phi}(0) \zeta(1+\alpha_1+\alpha_2)    \sum_{i_1,\cdots,i_k} \sum_{j_1,\cdots,j_k} c_{i_1}\cdots c_{i_k} c_{j_1} \cdots c_{j_k}
\\ 
&   \times \int_{[0,1]^{2k}} \int_{\substack{0\le t_{i,j}\\ \sum_{j=1}^k t_{i,j} \le 1\\ \sum_{i=1}^k t_{i,j} \le 1}} 
 \prod_{\ell=1}^{k} \bigg(1-\sum_{j=1}^k t_{\ell,j} \bigg)^{i_{\ell} + 1} y^{-\alpha_1 w_{\ell} (1-\sum_{j=1}^k t_{\ell,j})} (1-w_{\ell})^{i_{\ell}} 
\\
& \times\prod_{\ell=1}^{k} \bigg(1-\sum_{i=1}^k t_{i,\ell}\bigg)^{j_{\ell} + 1} y^{-\alpha_2 x_{\ell} (1-\sum_{i=1}^k t_{i,\ell})} (1-x_{\ell})^{j_\ell} 
\\
&dt_{1,1}\cdots dt_{k,k} \, 
dw_1 \cdots dw_{k} \,
dx_1 \cdots dx_k  + O\left(T (\log T)^{(k+1)^2 -1}\right).
\end{align*}

We make the following change of variables

$$ w_{\ell}^{'} = w_{\ell} \bigg( 1 - \sum_{j=1}^k t_{\ell,j}\bigg) \hspace{5mm} \text{and} \hspace{5mm} x_{\ell}^{'} = x_{\ell} \bigg( 1 - \sum_{i=1}^k t_{i,\ell}\bigg)  $$

\
This substitution gives

\begin{align*}
\mathcal{I}_{k,0,1}(\alpha_1, \alpha_2) =  &a_{k+1} T \log(y)^{2k + k^2}  \widehat{\Phi}(0) \zeta(1+\alpha_1+\alpha_2)       \int_{\substack{0\le t_{i,j},w_i,x_j\\ \sum_{j=1}^k t_{i,j} + w_i \leq 1 \\ \sum_{i=1}^k t_{i,j} +x_j \leq 1}}   y^{-\alpha_1 \sum_{i=1}^k w_i - \alpha_2 \sum_{i=1}^{k} x_i}   
\\
& \times
\prod_{i=1}^{k} P\bigg(1-\sum_{j=1}^k t_{i,j}-w_i\bigg)
\prod_{j=1}^{k} P\bigg(1-\sum_{i=1}^k t_{i,j}-x_j\bigg)
\\
&dt_{1,1}\cdots dt_{k,k} \, 
dw_1 \cdots dw_{k} \,
dx_1 \cdots dx_k  + O\left(T (\log T)^{(k+1)^2 -1}\right).
\end{align*}

$\mathcal{I}_{k,0,2}(\alpha_1,\alpha_2)$ is identical to $\mathcal{I}_{k,0,1}(\alpha_1,\alpha_2)$ under the permutation $(\alpha_1,\alpha_2)\leftrightarrow (-\alpha_2,-\alpha_1)$ and an additional factor $\chi_{\alpha_1,\alpha_2}(t)$. Applying part (3) of Lemma \ref{lemma:bounds} gives $\chi_{\alpha_1,\alpha_2}(t)=(t/2\pi)^{-(\alpha_1+\alpha_2)}(1+O(1/T))$, so we may factor out the corresponding $T$-term. Using also $\zeta(1+z)=z^{-1}+O(1)$, with all regular parts absorbed in the lower-order error, we have 

\begin{align*}
\mathcal{I}_{k,0}(\alpha_1, \alpha_2) =  & a_{k+1} T \log(y)^{2k + k^2}  \widehat{\Phi}(0)     \int_{\substack{0\le t_{i,j},w_i,x_j\\ \sum_{j=1}^k t_{i,j} + w_i \leq 1 \\ \sum_{i=1}^k t_{i,j} +x_j \leq 1}}    W(\underline{w}, \underline{x}, \alpha_1, \alpha_2) 
\\
& \times
\prod_{i=1}^{k} P\bigg(1-\sum_{j=1}^k t_{i,j}-w_i\bigg)
\prod_{j=1}^{k} P\bigg(1-\sum_{i=1}^k t_{i,j}-x_j\bigg)
\\
&dt_{1,1}\cdots dt_{k,k} \, 
dw_1 \cdots dw_{k} \,
dx_1 \cdots dx_k  + O\left(T (\log T)^{(k+1)^2 -1}\right)
\end{align*} 

\
where

\begin{align*}
 W(\underline{w}, \underline{x}, \alpha_1, \alpha_2)  = & y^{-\alpha_1 \sum_{i=1}^k w_i - \alpha_2 \sum_{i=1}^{k} x_i}   \left[ \frac{1 - \left(T y^{-\sum_{i=1}^{k} (w_i + x_i) }\right)^{-(\alpha_1 + \alpha_2)}}{\alpha_1 + \alpha_2} \right].
\end{align*}

\
Using the fact that

\begin{equation}\label{eqn:integral}
\frac{1-z^{-(\alpha_1+\alpha_2)}}{\alpha_1 + \alpha_2} = \log z \int_0^1 z^{-(\alpha_1+\alpha_2)v}\;dv, 
\end{equation}

\
we have

\begin{align*}
 W(\underline{w}, \underline{x}, \alpha_1, \alpha_2)  = &  y^{-\alpha_1 \sum_{i=1}^k w_i - \alpha_2 \sum_{i=1}^{k} x_i}    \left(1-\theta_k  \sum_{i=1}^{k} (w_i +x_i)  \right)  \\
 &\times\log T \int_0^1 \left(T y^{ -\sum_{i=1}^{k} (w_i + x_i) }\right)^{-v(\alpha_1+\alpha_2)} dv,
\end{align*}

\
which upon substitution completes the proof of Theorem \ref{thm:Ik0}.

\subsection{Proof of Theorem \ref{thm:Ik1}}

Write
$$
    A_{+}(t)=A\left(\frac12+it\right),
    \qquad
    A_{-}(t)=A\left(\frac12-it\right).
$$
Since the coefficients of $P$ are real, we have
$A_{-}(t)=\overline{A_{+}(t)}$.  Moreover, since
$\theta_k<1/(2k)$, there are fixed constants $\eta_1,\eta_2>0$ such
that
\begin{equation}\label{eqn:Ik1_length_bounds}
    y^k\leq T^{1/2-\eta_1},
    \qquad
    y^{k+1}\leq T^{1-\eta_2}.
\end{equation}

Applying the functional equation to the first zeta factor gives
\begin{align*}
    \mathcal{I}_{k,1}(\alpha_1,\alpha_2)
    ={}&
    \int_{-\infty}^{\infty}
    \zeta\left(\frac12-\alpha_1-it\right)
    \zeta\left(\frac12+\alpha_2-it\right)\\
    &\qquad\times
    \chi\left(\frac12+\alpha_1+it\right)
    \chi\left(\frac12-it\right)
    A_{+}(t)^{k+1}A_{-}(t)^{k-1}
    \Phi\left(\frac{t}{T}\right)\,dt.
\end{align*}
On the support of $\Phi(t/T)$ we have $t\asymp T$, and Part (3) of
Lemma \ref{lemma:bounds} gives
\begin{equation}\label{eqn:Ik1_stirling_product}
    \chi\left(\frac12+\alpha_1+it\right)
    \chi\left(\frac12-it\right)
    =
    \left(\frac{t}{2\pi}\right)^{-\alpha_1}
    \left(1+O\left(\frac1T\right)\right).
\end{equation}

We now give a mean-value estimate for the amplifier.  Write
$$
    A_{+}(t)^k
    =
    \sum_{n\leq y^k}\frac{\rho_k(n)}{n^{1/2+it}}.
$$
Since $P$ is fixed and bounded on $[0,1]$, we have
$\rho_k(n)\ll_P d_k(n)$.  The Montgomery--Vaughan mean-value theorem
therefore gives
\begin{align}
    \int_{-\infty}^{\infty}
    \Phi\left(\frac{t}{T}\right)|A_{+}(t)|^{2k}\,dt
    &\ll
    (T+y^k)\sum_{n\leq y^k}\frac{|\rho_k(n)|^2}{n}\notag\\
    &\ll
    T(\log T)^{k^2},
\label{eqn:Ik1_amplifier_mean}
\end{align}
where we used \eqref{eqn:Ik1_length_bounds}.  By Part (1) of
Lemma \ref{lemma:bounds}, the error introduced by
\eqref{eqn:Ik1_stirling_product} is
\begin{align*}
    &\ll
    T^{-1}
    \int_{-\infty}^{\infty}
    \Phi\left(\frac{t}{T}\right)
    \left|
        \zeta\left(\frac12-\alpha_1-it\right)
        \zeta\left(\frac12+\alpha_2-it\right)
    \right|
    |A_{+}(t)|^{2k}\,dt\\
    &\ll_{\varepsilon}
    T^{1/3+\varepsilon}(\log T)^{k^2},
\end{align*}
which is absorbed by the error term in the theorem.  It follows that
\begin{align}
    \mathcal{I}_{k,1}(\alpha_1,\alpha_2)
    ={}&
    \int_{-\infty}^{\infty}
    \Phi\left(\frac{t}{T}\right)
    \left(\frac{t}{2\pi}\right)^{-\alpha_1}
    \zeta\left(\frac12-\alpha_1-it\right)
    \zeta\left(\frac12+\alpha_2-it\right)\notag\\
    &\qquad\qquad\times
    A_{+}(t)^{k+1}A_{-}(t)^{k-1}\,dt
    +
    O\left(T(\log T)^{(k+1)^2-1}\right).
\label{eqn:Ik1_after_stirling}
\end{align}

The support of $\Phi(t/T)$ is contained in $[T/4,2T]$, so
Lemma \ref{lemma:bcy_sigma}, with $t$ replaced by $-t$, gives
\begin{equation}\label{eqn:Ik1_bcy}
    \zeta\left(\frac12-\alpha_1-it\right)
    \zeta\left(\frac12+\alpha_2-it\right)
    =
    \sum_{\ell=1}^{\infty}
    \frac{
        \sigma_{-\alpha_1,\alpha_2}(\ell)e^{-\ell/T^3}
    }{\ell^{1/2-it}}
    +
    O(T^{-1+\varepsilon}).
\end{equation}
By \eqref{eqn:Ik1_amplifier_mean}, the contribution of the error in
\eqref{eqn:Ik1_bcy} is
$$
    \ll
    T^{-1+\varepsilon}
    \int_{-\infty}^{\infty}
    \Phi\left(\frac{t}{T}\right)|A_{+}(t)|^{2k}\,dt
    \ll
    T^\varepsilon(\log T)^{k^2},
$$
and is negligible.

We now define
\begin{align*}
    A(n)
    &:=
    \sum_{\substack{
        d_1\cdots d_{k+1}=n\\
        d_1,\ldots,d_{k+1}\leq y
    }}
    P[d_1]\cdots P[d_{k+1}],\\
    B(m)
    &:=
    \sum_{\substack{
        abc_1\cdots c_{k-1}=m\\
        c_1,\ldots,c_{k-1}\leq y
    }}
    a^{\alpha_1}b^{-\alpha_2}
    e^{-ab/T^3}
    P[c_1]\cdots P[c_{k-1}].
\end{align*}
Substituting \eqref{eqn:Ik1_bcy} into
\eqref{eqn:Ik1_after_stirling}, and expanding the Dirichlet
polynomial blocks, gives
\begin{align}
    \mathcal{I}_{k,1}(\alpha_1,\alpha_2)
    ={}&
    \int_{-\infty}^{\infty}
    \Phi\left(\frac{t}{T}\right)
    \left(\frac{t}{2\pi}\right)^{-\alpha_1}
    \left(
        \sum_{m=1}^{\infty}\frac{B(m)}{m^{1/2-it}}
    \right)
    \left(
        \sum_{n\leq y^{k+1}}\frac{A(n)}{n^{1/2+it}}
    \right)dt\notag\\
    &+
    O\left(T(\log T)^{(k+1)^2-1}\right).
\label{eqn:Ik1_before_diagonal}
\end{align}

We extract the diagonal directly from the smooth weight.  Put
$$
    K_{\alpha_1}(x)
    =
    \int_{-\infty}^{\infty}
    \Phi\left(\frac{t}{T}\right)
    \left(\frac{t}{2\pi}\right)^{-\alpha_1}
    e^{itx}\,dt.
$$
Repeated integration by parts, using the derivative bounds imposed
on $\Phi$, gives, for every fixed $J\geq0$,
\begin{equation}\label{eqn:Ik1_kernel_decay}
    K_{\alpha_1}(x)
    \ll_{J}
    T\left(
        1+\frac{T|x|}{\log T}
    \right)^{-J},
\end{equation}
uniformly for $|\alpha_1|\ll1/\log T$.

The coefficients satisfy
\begin{equation}\label{eqn:Ik1_coeff_bounds}
    A(n)\ll_P d_{k+1}(n),
\end{equation}
and, for every fixed $\varepsilon>0$,
\begin{equation}\label{eqn:Ik1_B_bound}
    B(m)
    \ll_{\varepsilon,P}
    d_{k+1}(m)m^\varepsilon
    \exp\left(
        -\frac{m}{T^3y^{k-1}}
    \right).
\end{equation}
Indeed, in every factorisation contributing to $B(m)$ we have
$c_1\cdots c_{k-1}\leq y^{k-1}$, and hence
$ab\geq m/y^{k-1}$.  Since
$\theta_k(k-1)<1/2$, the series defining $B(m)$ may therefore be
truncated at $m\leq T^4$, with an error $O_A(T^{-A})$ for every fixed
$A>0$.

Put $N=y^{k+1}$.  Expanding the two series in
\eqref{eqn:Ik1_before_diagonal}, we obtain
\begin{equation}\label{eqn:Ik1_kernel_expansion}
    \sum_{n\leq N}\sum_{m\leq T^4}
    \frac{A(n)B(m)}{\sqrt{mn}}\,
    K_{\alpha_1}\left(\log\frac{m}{n}\right)
    +O_A(T^{-A}).
\end{equation}
We show that the terms with $m\neq n$ are negligible.  First suppose
that $m\geq2n$ or $m\leq n/2$.  Then
$|\log(m/n)|\geq\log2$, and
\eqref{eqn:Ik1_kernel_decay}--\eqref{eqn:Ik1_B_bound} give
\begin{align*}
    &\sum_{\substack{n\leq N,\ m\leq T^4\\
        m\geq2n\ {\rm or}\ m\leq n/2}}
    \frac{|A(n)B(m)|}{\sqrt{mn}}
    \left|K_{\alpha_1}\left(\log\frac{m}{n}\right)\right|\\
    &\qquad\ll_{\varepsilon,J}
    T^{1-J}(\log T)^J
    N^{1/2+\varepsilon}T^{2+4\varepsilon}
    =
    O_A(T^{-A}),
\end{align*}
on choosing $J$ sufficiently large.

It remains to consider $n/2<m<2n$ with $m\neq n$.  Write $m=n+h$.
Then
$$
    \left|\log\frac{m}{n}\right|
    \asymp\frac{|h|}{n}.
$$
Using \eqref{eqn:Ik1_coeff_bounds} and the divisor bound, the
contribution of this range is
\begin{align*}
    &\ll_{\varepsilon,J}
    T\sum_{n\leq N}n^{-1+2\varepsilon}
    \sum_{1\leq |h|\leq n}
    \left(
        1+\frac{T|h|}{n\log T}
    \right)^{-J}.
\end{align*}
By \eqref{eqn:Ik1_length_bounds}, $N\leq T^{1-\eta_2}$.  Hence
$$
    \sum_{1\leq |h|\leq n}
    \left(
        1+\frac{T|h|}{n\log T}
    \right)^{-J}
    \ll_J
    \left(\frac{n\log T}{T}\right)^J
$$
uniformly for $n\leq N$.  The near-diagonal contribution is therefore
\begin{align*}
    &\ll_{\varepsilon,J}
    T^{1-J}(\log T)^J
    \sum_{n\leq N}n^{J-1+2\varepsilon}\\
    &\ll_{\varepsilon,J}
    T^{1-J}N^{J+2\varepsilon}(\log T)^J
    \ll
    T^{1-\eta_2J+O(\varepsilon)},
\end{align*}
which is $O_A(T^{-A})$ after choosing $J$ sufficiently large.

It remains to evaluate the diagonal contribution.  We have
\begin{align*}
    K_{\alpha_1}(0)
    &=
    T\left(\frac{T}{2\pi}\right)^{-\alpha_1}
    \int_{-\infty}^{\infty}\Phi(u)u^{-\alpha_1}\,du\\
    &=
    T\left(\frac{T}{2\pi}\right)^{-\alpha_1}
    \left(
        \widehat{\Phi}(0)
        +O\left(\frac1{\log T}\right)
    \right).
\end{align*}
Moreover, uniformly in the shifts,
$$
    \sum_{n\leq N}\frac{|A(n)B(n)|}{n}
    \ll
    \sum_{n\leq N}\frac{d_{k+1}(n)^2}{n}
    \ll
    (\log T)^{(k+1)^2}.
$$
It follows from \eqref{eqn:Ik1_kernel_expansion} that
\begin{align}
    \mathcal{I}_{k,1}(\alpha_1,\alpha_2)
    ={}&
    T\widehat{\Phi}(0)
    \left(\frac{T}{2\pi}\right)^{-\alpha_1}
    \sum_{n\leq N}\frac{A(n)B(n)}{n}\notag\\
    &+
    O\left(T(\log T)^{(k+1)^2-1}\right).
\label{eqn:Ik1_diagonal_corrected}
\end{align}

We may now remove the exponential factor on the diagonal.  Let
$B_0(n)$ be defined as $B(n)$ with $e^{-ab/T^3}$ replaced by $1$.
Since $ab\leq n$ whenever $abc_1\cdots c_{k-1}=n$, we have
$$
    |B(n)-B_0(n)|
    \ll
    \frac{n}{T^3}d_{k+1}(n).
$$
Consequently,
\begin{align*}
    T\sum_{n\leq N}
    \frac{|A(n)||B(n)-B_0(n)|}{n}
    &\ll
    T^{-2}\sum_{n\leq N}d_{k+1}(n)^2\\
    &\ll
    T^{-2}N(\log T)^{(k+1)^2-1}
    =o(1).
\end{align*}
Thus \eqref{eqn:Ik1_diagonal_corrected} becomes
\begin{align}
    \mathcal{I}_{k,1}(\alpha_1,\alpha_2)
    ={}&
    T\widehat{\Phi}(0)
    \left(\frac{T}{2\pi}\right)^{-\alpha_1}
    \sum_{\substack{
        c_1,\ldots,c_{k-1}\leq y\\
        d_1,\ldots,d_{k+1}\leq y
    }}
    \sum_{
        abc_1\cdots c_{k-1}
        =
        d_1\cdots d_{k+1}
    }\notag\\
    &\quad\times
    \frac{
        P[c_1]\cdots P[c_{k-1}]
        P[d_1]\cdots P[d_{k+1}]
    }{
        a^{1/2-\alpha_1}
        b^{1/2+\alpha_2}
        \sqrt{
            c_1\cdots c_{k-1}
            d_1\cdots d_{k+1}
        }
    }\notag\\
    &+
    O\left(T(\log T)^{(k+1)^2-1}\right).
\label{eqn:Ik1_diagonal_without_exponential}
\end{align}

We replace each $P[\cdot]$ term by its Mellin representation
\eqref{mellin_polynomial}.  The constrained Dirichlet series which
arises in this way factors as
\begin{align}
    &\sum_{
        abc_1\cdots c_{k-1}
        =
        d_1\cdots d_{k+1}
    }
    \frac{1}{
        a^{1/2-\alpha_1}
        b^{1/2+\alpha_2}
        c_1^{1/2+u_1}\cdots c_{k-1}^{1/2+u_{k-1}}
        d_1^{1/2+v_1}\cdots d_{k+1}^{1/2+v_{k+1}}
    }\notag\\
    &\quad=
    \mathcal{A}(\underline{u},\underline{v},\alpha_1,\alpha_2)
    \prod_{j=1}^{k+1}
    \zeta(1-\alpha_1+v_j)
    \zeta(1+\alpha_2+v_j)
    \prod_{i=1}^{k-1}\prod_{j=1}^{k+1}
    \zeta(1+u_i+v_j),
\label{eqn:Ik1_euler_product}
\end{align}
where $\mathcal{A}$ is holomorphic and absolutely convergent in a
product of half-planes containing the origin.  Comparing the local
factors at the origin gives
\begin{align}
    \mathcal{A}(\underline{0},\underline{0},0,0)
    &=
    \prod_p
    \left(1-\frac1p\right)^{(k+1)^2}
    \sum_{r=0}^{\infty}
    \frac{d_{k+1}(p^r)^2}{p^r}\notag\\
    &=a_{k+1}.
\label{eqn:Ik1_arithmetic_value}
\end{align}

We now move the $u_i$- and $v_j$-contours to
$\Re(u_i)=\Re(v_j)=\delta$, where
$\delta\asymp1/\log T$ is chosen to lie to the right of all the
shifted poles.  We next justify replacing the arithmetic factor by
the value in \eqref{eqn:Ik1_arithmetic_value}.  Since $\mathcal{A}$ converges
absolutely in a product of half-planes containing the origin, there
is a fixed $\delta_0>0$ for which it has a multiple Dirichlet series
expansion
\begin{equation}\label{eqn:Ik1_arithmetic_series}
    \mathcal{A}(\underline{u},\underline{v},\alpha_1,\alpha_2)
    =
    \sum_{\underline{r},\underline{s}}
    \frac{
        a_{\alpha_1,\alpha_2}(\underline{r},\underline{s})
    }{
        r_1^{u_1}\cdots r_{k-1}^{u_{k-1}}
        s_1^{v_1}\cdots s_{k+1}^{v_{k+1}}
    },
\end{equation}
such that, uniformly for the shifts under consideration,
\begin{equation}\label{eqn:Ik1_arithmetic_absolute}
    \sum_{\underline{r},\underline{s}}
    \left|
        a_{\alpha_1,\alpha_2}(\underline{r},\underline{s})
    \right|
    \left(
        r_1\cdots r_{k-1}s_1\cdots s_{k+1}
    \right)^{\delta_0}
    \ll1.
\end{equation}
In particular the same sum, with an additional factor
$$
    1+\sum_{i=1}^{k-1}\log r_i
      +\sum_{j=1}^{k+1}\log s_j,
$$
is bounded.

Insert \eqref{eqn:Ik1_arithmetic_series} together with the absolutely
convergent Dirichlet series for the factors
$\zeta(1+u_i+v_j)$.  For fixed $\underline r,\underline s$, the
subsequent Mellin inversions replace the row and column logarithmic
variables by
$$
    \frac{\log\left(y/(r_i\prod_jm_{i,j})\right)}{\log y},
    \qquad
    \frac{\log\left(y/(s_j\prod_im_{i,j})\right)}{\log y},
$$
respectively.  Lemmas \ref{bmthm4.1} and \ref{t_c}, applied exactly
as below, show that replacing the factors $r_i,s_j$ by $1$ changes
the resulting expression by
$$
    O\left(
        (\log y)^{(k+1)^2-1}
        \left(
            1+\sum_i\log r_i+\sum_j\log s_j
        \right)
    \right).
$$
The estimate is uniform because the functions to which
Lemma \ref{t_c} is applied have bounded $C^1$ norms.  The terms for
which one of the $r_i$ or $s_j$ exceeds $y$ contribute
$O(y^{-\delta_0/2})$ by
\eqref{eqn:Ik1_arithmetic_absolute}.  Summing the remaining error by
\eqref{eqn:Ik1_arithmetic_absolute} shows that the leading
term is multiplied by
$$
    \sum_{\underline r,\underline s}
    a_{\alpha_1,\alpha_2}(\underline r,\underline s)
    =
    \mathcal{A}(\underline0,\underline0,\alpha_1,\alpha_2).
$$
Finally, holomorphicity in the shifts gives
$$
    \mathcal{A}(\underline0,\underline0,\alpha_1,\alpha_2)
    =
    a_{k+1}+O\left(\frac1{\log T}\right).
$$
Thus $\mathcal A$ may be replaced by $a_{k+1}$ in the leading term,
at a total cost
$$
    O\left(T(\log T)^{(k+1)^2-1}\right).
$$

Expanding the factors $\zeta(1+u_i+v_j)$ into their Dirichlet
series and collecting terms with the same powers of $u_i$ and
$v_j$, we obtain
\begin{align}
    \mathcal{I}_{k,1}(\alpha_1,\alpha_2)
    ={}&
    a_{k+1}T\widehat{\Phi}(0)
    \left(\frac{T}{2\pi}\right)^{-\alpha_1}
    \sum_{i_1,\ldots,i_{k-1}}
    \sum_{j_1,\ldots,j_{k+1}}
    \frac{
        c_{i_1}\cdots c_{i_{k-1}}
        c_{j_1}\cdots c_{j_{k+1}}
        i_1!\cdots i_{k-1}!
        j_1!\cdots j_{k+1}!
    }{
        (\log y)^{
            i_1+\cdots+i_{k-1}
            +j_1+\cdots+j_{k+1}
        }
    }\notag\\
    &\times
    \sum_{\substack{
        m_{i,j}\geq1\\
        \prod_{j=1}^{k+1}m_{i,j}\leq y\ (1\leq i\leq k-1)\\
        \prod_{i=1}^{k-1}m_{i,j}\leq y\ (1\leq j\leq k+1)
    }}
    \left(
        \prod_{i=1}^{k-1}\prod_{j=1}^{k+1}\frac1{m_{i,j}}
    \right)\notag\\
    &\times
    \frac{1}{(2\pi i)^{2k}}
    \int_{(\delta)^{2k}}
    \prod_{i=1}^{k-1}
    \left(
        \frac{y}{\prod_{j=1}^{k+1}m_{i,j}}
    \right)^{u_i}
    \prod_{j=1}^{k+1}
    \left(
        \frac{y}{\prod_{i=1}^{k-1}m_{i,j}}
    \right)^{v_j}\notag\\
    &\times
    \prod_{j=1}^{k+1}
    \zeta(1-\alpha_1+v_j)
    \zeta(1+\alpha_2+v_j)
    \frac{du_1}{u_1^{i_1+1}}\cdots
    \frac{du_{k-1}}{u_{k-1}^{i_{k-1}+1}}
    \frac{dv_1}{v_1^{j_1+1}}\cdots
    \frac{dv_{k+1}}{v_{k+1}^{j_{k+1}+1}}\notag\\
    &+
    O\left(T(\log T)^{(k+1)^2-1}\right).
\label{eqn:Ik1_matrix_sum}
\end{align}
The boundary cases in which a row or column product is equal to $y$
contribute one logarithm less and are absorbed in the error term.

We first evaluate the $u_i$-blocks.  Mellin inversion gives, for
$1\leq\ell\leq k-1$,
\begin{align}
    &\sum_{i_\ell}
    \frac{c_{i_\ell}i_\ell!}{(\log y)^{i_\ell}}
    \frac{1}{2\pi i}
    \int_{(\delta)}
    \left(
        \frac{y}{\prod_{j=1}^{k+1}m_{\ell,j}}
    \right)^u
    \frac{du}{u^{i_\ell+1}}\notag\\
    &\qquad=
    P\left(
        \frac{
            \log\left(y/\prod_{j=1}^{k+1}m_{\ell,j}\right)
        }{\log y}
    \right).
\label{eqn:Ik1_row_block}
\end{align}
For $1\leq\ell\leq k+1$, the corresponding $v_\ell$-block is
\begin{align*}
    \sum_{j_\ell}
    \frac{c_{j_\ell}j_\ell!}{(\log y)^{j_\ell}}
    \frac{1}{2\pi i}
    \int_{(\delta)}
    \left(
        \frac{y}{\prod_{i=1}^{k-1}m_{i,\ell}}
    \right)^v
    \zeta(1-\alpha_1+v)
    \zeta(1+\alpha_2+v)
    \frac{dv}{v^{j_\ell+1}}.
\end{align*}
Applying Lemma \ref{bmthm4.1} with shifts
$-\alpha_1,\alpha_2$, this is
\begin{align}
    &(\log y)^2
    \int_{\substack{
        w_\ell,x_\ell\geq0\\
        w_\ell+x_\ell\leq1
    }}
    \left(
        \frac{y}{\prod_{i=1}^{k-1}m_{i,\ell}}
    \right)^{\alpha_1w_\ell-\alpha_2x_\ell}
    \left(
        \frac{
            \log\left(y/\prod_{i=1}^{k-1}m_{i,\ell}\right)
        }{\log y}
    \right)^2\notag\\
    &\qquad\times
    P\left(
        \frac{
            \log\left(y/\prod_{i=1}^{k-1}m_{i,\ell}\right)
        }{\log y}
        (1-w_\ell-x_\ell)
    \right)
    \,dw_\ell\,dx_\ell
    +O(\log y).
\label{eqn:Ik1_column_block}
\end{align}
If one of the $k+1$ column blocks is replaced by the error in
\eqref{eqn:Ik1_column_block}, the total power of $\log y$ is reduced
by one.  Thus all accumulated block errors are
$O(T(\log T)^{(k+1)^2-1})$.

Put
$$
    R_i=
    \frac{
        \log\left(y/\prod_{j=1}^{k+1}m_{i,j}\right)
    }{\log y},
    \qquad
    C_j=
    \frac{
        \log\left(y/\prod_{i=1}^{k-1}m_{i,j}\right)
    }{\log y}.
$$
Substituting \eqref{eqn:Ik1_row_block} and
\eqref{eqn:Ik1_column_block} into
\eqref{eqn:Ik1_matrix_sum} gives
\begin{align}
    \mathcal{I}_{k,1}(\alpha_1,\alpha_2)
    ={}&
    a_{k+1}T\widehat{\Phi}(0)
    \left(\frac{T}{2\pi}\right)^{-\alpha_1}
    (\log y)^{2k+2}
    \int_{\substack{
        w_j,x_j\geq0\\
        w_j+x_j\leq1\ (1\leq j\leq k+1)
    }}\notag\\
    &\times
    \sum_{\substack{
        m_{i,j}\geq1\\
        \prod_{j=1}^{k+1}m_{i,j}\leq y\\
        \prod_{i=1}^{k-1}m_{i,j}\leq y
    }}
    \left(
        \prod_{i=1}^{k-1}\prod_{j=1}^{k+1}\frac1{m_{i,j}}
    \right)
    \prod_{i=1}^{k-1}P(R_i)\notag\\
    &\times
    \prod_{j=1}^{k+1}
    \left[
        C_j^2
        y^{(\alpha_1w_j-\alpha_2x_j)C_j}
        P\left(C_j(1-w_j-x_j)\right)
    \right]
    dw_1\cdots dw_{k+1}\,
    dx_1\cdots dx_{k+1}\notag\\
    &+
    O\left(T(\log T)^{(k+1)^2-1}\right).
\label{eqn:Ik1_before_polytope}
\end{align}

We now apply Lemma \ref{t_c} with
$(k,\ell)\mapsto(k-1,k+1)$.  For fixed
$\underline w,\underline x$, the functions occurring in
\eqref{eqn:Ik1_before_polytope} have uniformly bounded $C^1$ norms,
since $|\alpha_i|\log y\ll1$.  It follows that
\begin{align}
    \mathcal{I}_{k,1}(\alpha_1,\alpha_2)
    ={}&
    a_{k+1}T\widehat{\Phi}(0)
    \left(\frac{T}{2\pi}\right)^{-\alpha_1}
    (\log y)^{(k+1)^2}
    \int_{\substack{
        w_j,x_j\geq0\\
        w_j+x_j\leq1
    }}
    \int_{\substack{
        0\leq t_{i,j}\leq1\\
        \sum_{j=1}^{k+1}t_{i,j}\leq1\\
        \sum_{i=1}^{k-1}t_{i,j}\leq1
    }}\notag\\
    &\times
    \prod_{i=1}^{k-1}
    P\left(
        1-\sum_{j=1}^{k+1}t_{i,j}
    \right)\notag\\
    &\times
    \prod_{j=1}^{k+1}
    \left[
        \left(
            1-\sum_{i=1}^{k-1}t_{i,j}
        \right)^2
        y^{
            (\alpha_1w_j-\alpha_2x_j)
            \left(1-\sum_{i=1}^{k-1}t_{i,j}\right)
        }\right.\notag\\
    &\hspace{40mm}\left.\times
        P\left(
            \left(
                1-\sum_{i=1}^{k-1}t_{i,j}
            \right)
            (1-w_j-x_j)
        \right)
    \right]\notag\\
    &\times
    dt_{1,1}\cdots dt_{k-1,k+1}\,
    dw_1\cdots dw_{k+1}\,
    dx_1\cdots dx_{k+1}\notag\\
    &+
    O\left(T(\log T)^{(k+1)^2-1}\right).
\label{eqn:Ik1_before_change_variables}
\end{align}

Finally, for $1\leq j\leq k+1$, make the change of variables
$$
    w_j'
    =
    w_j\left(
        1-\sum_{i=1}^{k-1}t_{i,j}
    \right),
    \qquad
    x_j'
    =
    x_j\left(
        1-\sum_{i=1}^{k-1}t_{i,j}
    \right).
$$
The Jacobian cancels the square factor in each column block.  The
conditions $w_j,x_j\geq0$ and $w_j+x_j\leq1$ become
$$
    w_j',x_j'\geq0,
    \qquad
    \sum_{i=1}^{k-1}t_{i,j}+w_j'+x_j'\leq1.
$$
Removing the primes, we obtain
\begin{align*}
    \mathcal{I}_{k,1}(\alpha_1,\alpha_2)
    &=
    a_{k+1}T
    \widehat{\Phi}(0)
    \left(\frac{T}{2\pi}\right)^{-\alpha_1}
    (\log y)^{(k+1)^2}
    \int_{\substack{
        0\leq t_{i,j}\leq1\\
        0\leq w_j,x_j\leq1
    }}
    \int_{\substack{
        \sum_{j=1}^{k+1}t_{i,j}\leq1\\
        \sum_{i=1}^{k-1}t_{i,j}+w_j+x_j\leq1
    }}\\
    &\quad\times
    y^{
        \alpha_1\sum_{j=1}^{k+1}w_j
        -
        \alpha_2\sum_{j=1}^{k+1}x_j
    }
    \prod_{i=1}^{k-1}
    P\left(
        1-\sum_{j=1}^{k+1}t_{i,j}
    \right)
    \prod_{j=1}^{k+1}
    P\left(
        1-\sum_{i=1}^{k-1}t_{i,j}-w_j-x_j
    \right)\\
    &\quad\times
    dt_{1,1}\cdots dt_{k-1,k+1}\,
    dw_1\cdots dw_{k+1}\,
    dx_1\cdots dx_{k+1}
    +
    O\left(T(\log T)^{(k+1)^2-1}\right),
\end{align*}
which completes the proof.

\section{Amplified fourth moments}\label{section:amplified_fourth}

\subsection{Proof of Theorem \ref{thm:Jk0}}

An application of Lemma \ref{bmthm5.1} decomposes $\mathcal{J}_{k,0}$ into six pieces, together with a manageable error term, we have

\begin{equation}\label{eqn:6pieces}
    \mathcal{J}_{k,0}(\alpha_1,\alpha_2,\alpha_3,\alpha_4)=\sum_{\ell=1}^6\mathcal{J}_{k,0,\ell}(\alpha_1,\alpha_2,\alpha_3,\alpha_4)+O_{\varepsilon}\left(T^{\frac{1}{2}+2k\vartheta_k+\varepsilon}+T^{\frac{3}{4}+k\vartheta_k+\varepsilon}\right).
\end{equation}

\
Given the restriction $\vartheta_k<\frac{1}{4k}$ this error term is $O(T^{1-\eta})$ for some $\eta>0$. 

We begin by considering $\mathcal{J}_{k,0,1}(\alpha_1,\alpha_2,\alpha_3,\alpha_4)$, and later discuss how to reintroduce the remaining five integrals by permuting the shifts appropriately. We have

\begin{align}
\mathcal{J}_{k,0,1}(\alpha_1,\alpha_2,\alpha_3,\alpha_4) =  & \sum_{\substack{a_1,\cdots, a_k \leq y \\ b_1,\cdots,b_k \leq y}} \frac{P[a_1]\cdots P[a_k]P[b_1]\cdots P[b_k]}{\sqrt{a_1\cdots a_k b_1\cdots b_k}} \int_{-\infty}^{\infty}  \Phi(\tfrac{t}{T}) \sum_{a_1\cdots a_km_1m_2 = b_1\cdots b_kn_1n_2} \notag\\
&  \times\frac{1}{m_1^{1/2+\alpha_1} m_2^{1/2+\alpha_2} n_1^{1/2+\alpha_3}  n_2^{1/2+\alpha_4}} \frac{1}{2\pi i} \int_{(2)} G(s) \bigg( \frac{t}{2\pi}\bigg)^{2s} \frac{1}{(m_1m_2n_1n_2)^s} \frac{ds}{s}\;dt, \label{eqn:H_{k+2,1}}\end{align} 

\
with $G(s)$ chosen to be an even entire function of rapid decay in vertical strips, with $G(0)=1$ and with zeros at $-\frac{\alpha_i+\alpha_j}{2}$ for $i\ne j$.

As above, the argument is first made for generic shifts and then extended by holomorphic continuation.  Uniformity on fixed polydisc boundaries, followed by Cauchy's formula, absorbs the harmless powers of $\log T$ introduced by the shift-dependent damping factor.

We appeal to the Mellin transform \eqref{mellin_polynomial} to rewrite our coefficients to give

\begin{align*}
\mathcal{J}_{k,0,1}(\alpha_1,\alpha_2,\alpha_3,\alpha_4)=    & \frac{1}{(2\pi i)^{1 + 2k}} \sum_{i_1,\cdots,i_k} \sum_{j_1,\cdots,j_k} \frac{c_{i_1} \cdots c_{i_k} c_{j_1}\cdots c_{j_k} i_1!\cdots i_k! j_1!\cdots j_k!}{ (\log y)^{ i_1+\cdots+i_k+j_1+\cdots+j_k}} \int_{(2)^{1+2k}} 
\\
& \times \int_{-\infty}^{\infty} y^{u_1+\cdots+u_k+v_1+\cdots+v_k}  \Phi(\tfrac{t}{T}) G(s) \bigg( \frac{t}{2\pi}\bigg)^{2s}     \sum_{a_1\cdots a_km_1m_2 = b_1\cdots b_kn_1n_2} 
\\
&  \times \frac{1}{a_1^{1/2+u_1}\cdots a_k^{1/2+u_k} b_1^{1/2+v_1}\cdots b_k^{1/2+v_k}}  \frac{1}{m_1^{1/2+\alpha_1+ s} m_2^{1/2+\alpha_2 + s} n_1^{1/2+\alpha_3 + s}  n_2^{1/2+\alpha_4 + s}}
\\
& \frac{du_1}{u_1^{i_1+1}}\cdots\frac{du_k}{u_k^{i_k+1}}\frac{dv_1}{v_1^{j_1+1}}\cdots \frac{dv_k}{v_k^{j_k+1}}\frac{ds}{s}dt.
\end{align*} 

Analogously to the proof of the twisted second moment, we observe that the constrained sum factors as an Euler product. Writing $\mathcal{A}(\underline{\alpha}, u_i, v_j, s)$ for an arithmetical factor converging absolutely in a product of half planes containing the origin we have

\begin{align*}
\mathcal{J}_{k,0,1}(\alpha_1,\alpha_2,\alpha_3,\alpha_4)&=\frac{1}{(2\pi i)^{1 + 2k}} \sum_{i_1,\cdots,i_k} \sum_{j_1,\cdots,j_k} \frac{c_{i_1} \cdots c_{i_k} c_{j_1}\cdots c_{j_k} i_1!\cdots i_k! j_1!\cdots j_k!}{ (\log y)^{ i_1+\cdots+i_k+j_1+\cdots+j_k}} \int_{(2)^{1+2k}} \int_{-\infty}^{\infty} 
\\
& \times y^{u_1+\cdots+u_k+v_1+\cdots+v_k} \Phi(\tfrac{t}{T}) G(s) \bigg( \frac{t}{2\pi}\bigg)^{2s}  \mathcal{A}(\underline{\alpha}, u_i, v_j, s)\prod_{i=1}^k\prod_{j=1}^k\zeta(1+u_i + v_j) 
\\
& \times \prod_{i=1}^k\left[\zeta(1+u_i+s+\alpha_3)\zeta(1+u_i+s+\alpha_4)\right]
\prod_{j=1}^k\left[\zeta(1+v_j+s+\alpha_1)\zeta(1+v_j+s+\alpha_2)\right]
\\
&  \times  \zeta(1+\alpha_1 + \alpha_3 + 2s)   \zeta(1+\alpha_1 + \alpha_4 + 2s) 
\\
&  \times  \zeta(1+\alpha_2 + \alpha_3 + 2s) \zeta(1+\alpha_2 + \alpha_4 + 2s)  \frac{du_1}{u_1^{i_1+1}}\cdots\frac{du_k}{u_k^{i_k+1}}\frac{dv_1}{v_1^{j_1+1}}\cdots \frac{dv_k}{v_k^{j_k+1}}\frac{ds}{s}dt.
\end{align*} 

We first shift the $u_i$- and $v_j$-contours to $\Re(u_i)=\Re(v_j)=\delta$, and then shift the $s$-contour to $\Re(s)=-\delta/2$. Here $\delta>0$ is a fixed constant for which the arithmetical piece $S$ converges absolutely. The poles from the factors $1+u_i+s+\alpha$ and $1+v_j+s+\alpha$ stay to the left of the new $s$-line, so we encounter only the potential poles at $s=0$ and at $s=-\frac{(\alpha_i+\alpha_j)}{2}$. However, the latter pole is cancelled by the zeros of $G(s)$, thus

$$\mathcal{J}_{k,0,1}(\alpha_1,\alpha_2,\alpha_3,\alpha_4)=\Res_{s=0}+I_{-\delta/2}.$$

\
We may bound the line integral as 

$$ I_{-\delta/2} \ll y^{2k \delta} T^{-\delta} T^{1+\varepsilon}\ll T^{\delta/2} T^{-\delta} T^{1+\varepsilon} \ll T^{1-\varepsilon}, $$

\
where we have used $y = T^{\vartheta_k} \ll T^{\tfrac{1}{4k}-\varepsilon^{'}}$. Therefore at the cost of the tolerable error $O(T^{1-\varepsilon})$ we have, after factoring out the terms independent of $i,j$ and performing the integration over $t$,

\begin{align*}
\mathcal{J}_{k,0,1}(\alpha_1,\alpha_2,\alpha_3,\alpha_4) &=  T \widehat{\Phi}(0)\zeta(1+\alpha_1 + \alpha_3) \zeta(1+\alpha_1 + \alpha_4) \zeta(1+\alpha_2 + \alpha_3) \zeta(1+\alpha_2 + \alpha_4)
\\
& \times\frac{1}{(2\pi i)^{2k}} \sum_{i_1,\cdots,i_k} \sum_{j_1,\cdots, j_k} \frac{c_{i_1} \cdots c_{i_k} c_{j_1}\cdots c_{j_k} i_1!\cdots i_k! j_1!\cdots j_k!}{ (\log y)^{ i_1+\cdots+i_k+j_1+\cdots+j_k}} \int_{(2)^{2k}} 
\\
& \times y^{u_1+\cdots+u_k+v_1+\cdots+v_k} \mathcal{A}(\underline{\alpha}, u_i, v_j) \prod_{i=1}^k\prod_{j=1}^k \zeta(1+u_i + v_j)
\\
& \times\prod_{i=1}^k\left[\zeta(1+u_i+\alpha_3)\zeta(1+u_i+\alpha_4)\right]\prod_{j=1}^k\left[\zeta(1+v_j+\alpha_1)\zeta(1+v_j+\alpha_2)\right]  \\&\frac{du_1}{u_1^{i_1+1}}\cdots\frac{du_k}{u_k^{i_k+1}}\frac{dv_1}{v_1^{j_1+1}}\cdots \frac{dv_k}{v_k^{j_k+1}}.
\end{align*} 

We again shift the contours to $\Re(u_i)=\Re(v_j)\asymp(\log T)^{-1}$.  The holomorphy and local boundedness of $\mathcal A$ in these polydiscs allow us to replace it by $\mathcal A(0)=a_{k+2}$: the holomorphic difference is $O(1/\log T)$, while the polar factors contribute at most the displayed power of $\log T$. Hence the difference contributes $O(T(\log T)^{(k+2)^2-1})$, exactly as in the proof of Theorem~\ref{thm:Ik0}.

\begin{align*}
\mathcal{J}_{k,0,1}(\alpha_1,\alpha_2,\alpha_3,\alpha_4) &=  T a_{k+2}\widehat{\Phi}(0)\zeta(1+\alpha_1 + \alpha_3) \zeta(1+\alpha_1 + \alpha_4) \zeta(1+\alpha_2 + \alpha_3) \zeta(1+\alpha_2 + \alpha_4)
\\
& \times \sum_{i_1,\cdots,i_k} \sum_{j_1,\cdots,j_k} \frac{c_{i_1}\cdots c_{i_k} c_{j_1}\cdots c_{j_k} i_1!\cdots i_k! j_1!\cdots j_k!}{ (\log y)^{ i_1+\cdots+i_k+j_1+\cdots+j_k}} \frac{1}{(2\pi i)^{2k}}\int_{(1/\log T)^{2k}}  
\\
&\times y^{u_1+\cdots+u_k+v_1+\cdots+v_k}  \prod_{i=1}^k\prod_{j=1}^k \zeta(1+u_i + v_j)
\\
& \times\prod_{i=1}^k\left[\zeta(1+u_i+\alpha_3)\zeta(1+u_i+\alpha_4)\right]\prod_{j=1}^k\left[\zeta(1+v_j+\alpha_1)\zeta(1+v_j+\alpha_2)\right]  
\\
&      \frac{du_1}{u_1^{i_1+1}}\cdots\frac{du_k}{u_k^{i_k+1}}\frac{dv_1}{v_1^{j_1+1}}\cdots \frac{dv_k}{v_k^{j_k+1}}+O\left(T(\log T)^{(k+2)^2-1}\right).
\end{align*} 

We now replace each entangled zeta function with its Dirichlet series, interchange summation and integration and collect terms with the same powers of $u_i$ and $v_j$. This gives

\begin{align*}
    \mathcal{J}_{k,0,1}(\alpha_1,\alpha_2,\alpha_3,\alpha_4)&=a_{k+2}T\widehat{\Phi}(0)\zeta(1+\alpha_1+\alpha_3)\zeta(1+\alpha_1+\alpha_4)\zeta(1+\alpha_2+\alpha_3)\zeta(1+\alpha_2+\alpha_4)
    \\
    &\times\sum_{i_1,\cdots,i_k}\sum_{j_1,\cdots,j_k}\frac{c_{i_1}\cdots c_{i_k}c_{j_1}\cdots c_{j_k}i_1!\cdots i_k!j_1!\cdots j_k!}{(\log y)^{i_1+\cdots+i_k+j_1+\cdots+j_k}}\sum_{\substack{m_{i,j}\ge 1\\ \prod_{j=1}^km_{i,j}\le y\\ \prod_{i=1}^km_{i,j}\le y}}\left(\prod_{i=1}^k\prod_{j=1}^k\frac{1}{m_{i,j}}\right)
    \\
    &\times\frac{1}{(2\pi i)^{2k}}\int_{(1/\log T)^{2k}}\prod_{i=1}^k\left(\frac{y}{\prod_{j=1}^km_{i,j}}\right)^{u_i}\prod_{j=1}^k\left(\frac{y}{\prod_{i=1}^km_{i,j}}\right)^{v_j}
    \\
    &\times\prod_{i=1}^k\left[\zeta(1+u_i+\alpha_3)\zeta(1+u_i+\alpha_4)\right]\prod_{j=1}^k\left[\zeta(1+v_j+\alpha_1)\zeta(1+v_j+\alpha_2)\right]
    \\
    &\times \frac{du_1}{u_1^{i_1+1}}\cdots\frac{du_k}{u_k^{i_k+1}}\frac{dv_1}{v_1^{j_1+1}}\cdots\frac{dv_k}{v_k^{j_k+1}}+O(T(\log T)^{(k+2)^2-1}).
\end{align*}

We may now evaluate the $u_i$ and $v_j$ blocks. For a fixed $1\le\ell\le k$, the $u_{\ell}$ block is

$$\sum_{r=0}^{\infty}\frac{c_rr!}{(\log y)^r}\frac{1}{2\pi i}\int_{(1/\log T)}\left(\frac{y}{\prod_{j=1}^k m_{\ell,j}}\right)^u\zeta(1+u+\alpha_3)\zeta(1+u+\alpha_4)\frac{du}{u^{r+1}}.$$

\
By Lemma \ref{bmthm4.1} with shifts $\alpha_3$ and $\alpha_4$, this equals

\begin{align*}
    (\log y)^2\int_{\substack{w_{\ell},x_{\ell}\ge 0\\ w_{\ell}+x_{\ell}\le 1}}&\left(\frac{y}{\prod_{j=1}^km_{\ell,j}}\right)^{-\alpha_3w_{\ell}-\alpha_4x_{\ell}}\left(\frac{\log(y/\prod_{j=1}^km_{\ell,j})}{\log y}\right)^2
    \\
    &\times P\left(\frac{\log(y/\prod_{j=1}^km_{\ell,j})}{\log y}(1-w_{\ell}-x_{\ell})\right)dw_{\ell}dx_{\ell}+O(\log y).
\end{align*}

\
Similarly, for a fixed $1\le \ell\le k$, the $v_{\ell}$ block is 

$$\sum_{r=0}^{\infty}\frac{c_rr!}{(\log y)^r}\frac{1}{2\pi i}\int_{(1/\log T)}\left(\frac{y}{\prod_{i=1}^k m_{i,\ell}}\right)^v\zeta(1+v+\alpha_1)\zeta(1+v+\alpha_2)\frac{dv}{v^{r+1}},$$

\
which equals

\begin{align*}
    (\log y)^2\int_{\substack{y_{\ell},z_{\ell}\ge 0\\ y_{\ell}+z_{\ell}\le 1}}&\left(\frac{y}{\prod_{i=1}^km_{i,\ell}}\right)^{-\alpha_1y_{\ell}-\alpha_2z_{\ell}}\left(\frac{\log(y/\prod_{i=1}^km_{i,\ell})}{\log y}\right)^2
    \\
    &\times P\left(\frac{\log(y/\prod_{i=1}^km_{i,\ell})}{\log y}(1-y_{\ell}-z_{\ell})\right)dy_{\ell}dz_{\ell}+O(\log y).
\end{align*}

\
These substitutions produce

\begin{align*}
    &\mathcal{J}_{k,0,1}(\alpha_1,\alpha_2,\alpha_3,\alpha_4)\\&=a_{k+2}T\widehat{\Phi}(0)\zeta(1+\alpha_1+\alpha_3)\zeta(1+\alpha_1+\alpha_4)\zeta(1+\alpha_2+\alpha_3)\zeta(1+\alpha_2+\alpha_4)\\&\times(\log y)^{4k}\int_{\substack{w_i,x_i\ge 0\\ w_i+x_i\le 1}}\int_{\substack{y_j,z_j\ge 0\\ y_j+z_j\le 1}}\sum_{\substack{m_{i,j}\ge 1\\\prod_{j=1}^km_{i,j}\le y\\ \prod_{i=1}^km_{i,j}\le y}}\left(\prod_{i=1}^k\prod_{j=1}^k\frac{1}{m_{i,j}}\right)\\&\times\prod_{i=1}^k\Bigg[\left(\frac{y}{\prod_{j=1}^km_{i,j}}\right)^{-\alpha_3w_i-\alpha_4x_i}\left(\frac{\log(y/\prod_{j=1}^km_{i,j})}{\log y}\right)^2P\left(\frac{\log(y/\prod_{j=1}^km_{i,j})}{\log y}(1-w_i-x_i)\right)\Bigg]\\&\times\prod_{j=1}^k\Bigg[\left(\frac{y}{\prod_{i=1}^km_{i,j}}\right)^{-\alpha_1y_j-\alpha_2z_j}\left(\frac{\log(y/\prod_{i=1}^km_{i,j})}{\log y}\right)^2P\left(\frac{\log(y/\prod_{i=1}^km_{i,j})}{\log y}(1-y_j-z_j)\right)\Bigg]\\& dw_1\cdots dw_kdx_1\cdots dx_k dy_1\cdots dy_kdz_1\cdots dz_k+O(T(\log T)^{(k+2)^2-1}).
\end{align*}

An application of  Lemma \ref{t_c} gives

\begin{align*}
    &\mathcal{J}_{k,0,1}(\alpha_1,\alpha_2,\alpha_3,\alpha_4)=
    \\
    & a_{k+2}T\widehat{\Phi}(0)\zeta(1+\alpha_1+\alpha_3)\zeta(1+\alpha_1+\alpha_4)  \zeta(1+\alpha_2+\alpha_3) 
    \\
    &\times\zeta(1+\alpha_2+\alpha_4)(\log y)^{k^2+4k}\int_{\substack{w_i,x_i\ge 0\\ w_i+x_i\le 1}}\int_{\substack{y_j,z_j\ge 0\\ y_j+z_j\le 1}}\int_{\substack{0\le t_{i,j}\le 1\\ \sum_{j=1}^kt_{i,j}\le 1\\\sum_{i=1}^kt_{i,j}\le 1}}
    \\
    &\times\prod_{i=1}^k\Bigg[\left(1-\sum_{j=1}^kt_{i,j}\right)^2y^{-(\alpha_3w_i+\alpha_4x_i)(1-\sum_{j=1}^kt_{i,j})}P\left(\left(1-\sum_{j=1}^kt_{i,j}\right)(1-w_i-x_i)\right)\Bigg]
    \\
    &\times\prod_{j=1}^k\Bigg[\left(1-\sum_{i=1}^kt_{i,j}\right)^2y^{-(\alpha_1y_j+\alpha_2z_j)(1-\sum_{i=1}^kt_{i,j})}P\left(\left(1-\sum_{i=1}^kt_{i,j}\right)(1-y_j-z_j)\right)\Bigg]
    \\
    &dt_{1,1}\cdots dt_{k,k}dw_1\cdots dw_{k}dx_1\cdots dx_kdy_1\cdots dy_kdz_1\cdots dz_k+O(T(\log T)^{(k+2)^2-1}).
\end{align*}

We now make the change of variables

\begin{align*}
    w_i'&=w_i\left(1-\sum_{j=1}^kt_{i,j}\right),\qquad x_i'=x_i\left(1-\sum_{j=1}^kt_{i,j}\right) \qquad (1\le i\le k),
    \\
    y_j'&=y_j\left(1-\sum_{i=1}^kt_{i,j}\right),\qquad z_j'=z_j\left(1-\sum_{i=1}^kt_{i,j}\right) \qquad (1\le j\le k).
\end{align*}

\
The Jacobian cancels the square factors in each block. Hence the conditions become

\begin{align*}
    &\sum_{j=1}^kt_{i,j}+w_i+x_i\le 1 \hspace{3mm} (1\le i\le k)\hspace{6mm} \text{and} \hspace{6mm}  \sum_{i=1}^kt_{i,j}+y_j+z_j\le 1 \hspace{3mm} (1\le j\le k).
\end{align*}

Thus this first piece becomes

\begin{align*}
    \mathcal{J}_{k,0,1}(\alpha_1,\alpha_2,\alpha_3,\alpha_4)&=a_{k+2}T\widehat{\Phi}(0)(\log y)^{k^2+4k}\zeta(1+\alpha_1+\alpha_3)\zeta(1+\alpha_1+\alpha_4)\zeta(1+\alpha_2+\alpha_3)
    \\
    &\times \zeta(1+\alpha_2+\alpha_4) \int_{\substack{0\le t_{i,j}\le 1\\ 0\le w_i,x_i,y_i,z_i\le 1\\ \sum_{j=1}^kt_{i,j}+w_i+x_i\le 1\\ \sum_{i=1}^kt_{i,j}+y_j+z_j\le 1}}y^{-\alpha_3\sum_{i=1}^kw_i-\alpha_4\sum_{i=1}^kx_i-\alpha_1\sum_{j=1}^ky_j-\alpha_2\sum_{j=1}^kz_j}
    \\
    &\times\prod_{i=1}^kP\left(1-\sum_{j=1}^kt_{i,j}-w_i-x_i\right)\prod_{j=1}^kP\left(1-\sum_{i=1}^kt_{i,j}-y_j-z_j\right)
    \\
    & dt_{1,1}\cdots dt_{k,k}dw_1\cdots dw_kdx_1\cdots dx_k dy_1\cdots dy_kdz_1\cdots dz_k+O(T(\log T)^{(k+2)^2-1}).
\end{align*}

We now insert the remaining five pieces in \eqref{eqn:6pieces}. Following the procedure in \cite[Section 5.2]{BM}, the product of four zeta factors is replaced by the corresponding six-term expression.  Denote this expression by $\mathcal W_0(\underline w,\underline x,\underline y,\underline z;\underline\alpha)$.  Thus

\begin{align}\label{eqn:Jk0_with_W0}
\mathcal J_{k,0}(\alpha_1,\alpha_2,\alpha_3,\alpha_4)
&=
a_{k+2}T\widehat\Phi(0)(\log y)^{k^2+4k}
\int_{\substack{
0\le t_{i,j}\le1\ \\
0\le w_i,x_i\le1\ \\
0\le y_j,z_j\le1\ 
}}
\int_{\substack{\sum_{j=1}^k t_{i,j}+w_i+x_i\le1\ \\
\sum_{i=1}^k t_{i,j}+y_j+z_j\le1\ }}
\mathcal W_0(\underline w,\underline x,\underline y,\underline z;\underline\alpha)
\nonumber\\
&\quad\times
\prod_{i=1}^k
P\left(1-\sum_{j=1}^k t_{i,j}-w_i-x_i\right)
\prod_{j=1}^k
P\left(1-\sum_{i=1}^k t_{i,j}-y_j-z_j\right)
\nonumber\\
&\quad 
dt_{1,1}\cdots dt_{k,k}
\,dw_1\cdots dw_k\,dx_1\cdots dx_k\,
dy_1\cdots dy_k\,dz_1\cdots dz_k
\nonumber 
+O\left(T(\log T)^{(k+2)^2-1}\right).
\end{align}

\
Explicitly,

\begin{align*}
&\mathcal W_0(\underline w,\underline x,\underline y,\underline z;\underline\alpha)
\\
&=
\frac{
y^{-\alpha_3\sum_{i=1}^k w_i
-\alpha_4\sum_{i=1}^k x_i
-\alpha_1\sum_{j=1}^k y_j
-\alpha_2\sum_{j=1}^k z_j}
}{
(\alpha_1+\alpha_3)(\alpha_1+\alpha_4)(\alpha_2+\alpha_3)(\alpha_2+\alpha_4)
}
\\
&\quad-
\frac{
T^{-(\alpha_1+\alpha_3)}
y^{\alpha_1\sum_{i=1}^k w_i
-\alpha_4\sum_{i=1}^k x_i
+\alpha_3\sum_{j=1}^k y_j
-\alpha_2\sum_{j=1}^k z_j}
}{
(\alpha_1+\alpha_3)(\alpha_4-\alpha_3)(\alpha_2-\alpha_1)(\alpha_2+\alpha_4)
}
\\
&\quad-
\frac{
T^{-(\alpha_1+\alpha_4)}
y^{-\alpha_3\sum_{i=1}^k w_i
+\alpha_1\sum_{i=1}^k x_i
+\alpha_4\sum_{j=1}^k y_j
-\alpha_2\sum_{j=1}^k z_j}
}{
(\alpha_3-\alpha_4)(\alpha_1+\alpha_4)(\alpha_2+\alpha_3)(\alpha_2-\alpha_1)
}
\\
&\quad-
\frac{
T^{-(\alpha_2+\alpha_3)}
y^{\alpha_2\sum_{i=1}^k w_i
-\alpha_4\sum_{i=1}^k x_i
-\alpha_1\sum_{j=1}^k y_j
+\alpha_3\sum_{j=1}^k z_j}
}{
(\alpha_1-\alpha_2)(\alpha_1+\alpha_4)(\alpha_2+\alpha_3)(\alpha_4-\alpha_3)
}
\\
&\quad-
\frac{
T^{-(\alpha_2+\alpha_4)}
y^{-\alpha_3\sum_{i=1}^k w_i
+\alpha_2\sum_{i=1}^k x_i
-\alpha_1\sum_{j=1}^k y_j
+\alpha_4\sum_{j=1}^k z_j}
}{
(\alpha_1+\alpha_3)(\alpha_1-\alpha_2)(\alpha_3-\alpha_4)(\alpha_2+\alpha_4)
}
\\
&\quad+
\frac{
T^{-(\alpha_1+\alpha_2+\alpha_3+\alpha_4)}
y^{\alpha_1\sum_{i=1}^k w_i
+\alpha_2\sum_{i=1}^k x_i
+\alpha_3\sum_{j=1}^k y_j
+\alpha_4\sum_{j=1}^k z_j}
}{
(\alpha_1+\alpha_3)(\alpha_2+\alpha_3)(\alpha_1+\alpha_4)(\alpha_2+\alpha_4)
}.
\end{align*}

The six terms above form the same fourfold divided-difference expression as in \cite[Section 5.2]{BM}. Repeated use of
$$
    \frac{e^A-e^B}{A-B}=\int_0^1 e^{uA+(1-u)B}\,du
$$
therefore rewrites it as

\begin{align*}
& \quad \frac{1}{2} (\log T)^4
y^{-\alpha_3\sum_{i=1}^k w_i
-\alpha_4\sum_{i=1}^k x_i
-\alpha_1\sum_{j=1}^k y_j
-\alpha_2\sum_{j=1}^k z_j} \int_{[0,1]^4}
\\
&\quad\times \left(1-\vartheta_k\sum_{i=1}^k w_i-\vartheta_k\sum_{j=1}^k y_j\right)
\left(1-\vartheta_k\sum_{i=1}^k x_i-\vartheta_k\sum_{j=1}^k z_j\right) 
\\
&\quad\times
\biggl[
\vartheta_k\left(\sum_{i=1}^k w_i-\sum_{i=1}^k x_i\right)
+
v_1\left(1-\vartheta_k\sum_{i=1}^k w_i-\vartheta_k\sum_{j=1}^k y_j\right)
-
v_2\left(1-\vartheta_k\sum_{i=1}^k x_i-\vartheta_k\sum_{j=1}^k z_j\right)
\biggr]
\\
&\quad\times
\biggl[
\vartheta_k\left(\sum_{j=1}^k y_j-\sum_{j=1}^k z_j\right)
+
v_1\left(1-\vartheta_k\sum_{i=1}^k w_i-\vartheta_k\sum_{j=1}^k y_j\right)
-
v_2\left(1-\vartheta_k\sum_{i=1}^k x_i-\vartheta_k\sum_{j=1}^k z_j\right)
\biggr]
\\
&\quad\times  
\left(Ty^{-\sum_{i=1}^k w_i-\sum_{j=1}^k y_j}\right)^{-(\alpha_1+\alpha_3)v_1} \left(
\frac{
y^{\sum_{j=1}^k y_j-\sum_{j=1}^k z_j}
\left(Ty^{-\sum_{i=1}^k w_i-\sum_{j=1}^k y_j}\right)^{v_1}
}{
\left(Ty^{-\sum_{i=1}^k x_i-\sum_{j=1}^k z_j}\right)^{v_2}
}
\right)^{-(\alpha_2-\alpha_1)v_3}
\\
&\quad\times \left(Ty^{-\sum_{i=1}^k x_i-\sum_{j=1}^k z_j}\right)^{-(\alpha_2+\alpha_4)v_2}
\left(
\frac{
y^{\sum_{i=1}^k w_i-\sum_{i=1}^k x_i}
\left(Ty^{-\sum_{i=1}^k w_i-\sum_{j=1}^k y_j}\right)^{v_1}
}{
\left(Ty^{-\sum_{i=1}^k x_i-\sum_{j=1}^k z_j}\right)^{v_2}
}
\right)^{-(\alpha_4-\alpha_3)v_4} 
\\
& \quad dv_1 dv_2 dv_3 dv_4,
\end{align*}

\
This step is purely algebraic and introduces no further error; substituting the identity completes the proof.

\subsection{Proof of Theorem \ref{thm:Jk1}}

We again work first in the unsmoothed dyadic normalisation.  More precisely, the calculation below is made with a sharp upper endpoint $X\asymp T$, uniformly for $X\in[T/4,2T]$; partial summation in $X$ against $\Phi(X/T)$ then gives the smoothed statement.  In passing to this smoothed form, the factor
$$
T\left(\frac{T}{2\pi}\right)^{-(\alpha_3+\alpha_4)}
$$
is replaced throughout by
$$
\int_{\mathbb R}\Phi(t/T)\left(\frac{t}{2\pi}\right)^{-(\alpha_3+\alpha_4)}dt
=T\widehat{\Phi}(0)\left(\frac{T}{2\pi}\right)^{-(\alpha_3+\alpha_4)}+O(T/\log T),
$$
uniformly for shifts of size $O(1/\log T)$, and the error is absorbed in the final error term.

Reversing the conjugation, introducing a $\chi(s)\chi(1-s)$ pair and applying Lemma \ref{lemma:bounds}, we may write $\mathcal{J}_{k,1}$ as  

\begin{align*}
    \mathcal{J}_{k,1}(\alpha_1,\alpha_2,\alpha_3,\alpha_4) & =\left(\frac{T}{2\pi}\right)^{-(\alpha_3+\alpha_4)}\frac{1}{i}\int_{1/2+i}^{1/2+iT}\zeta(s+\alpha_1)\zeta(s+\alpha_2)\zeta(s-\alpha_3)\zeta(s-\alpha_4)
    \\
    & \hspace{3.5cm} \chi(1-s)A(1-s)^{k+1}A(s)^{k-1}\;ds +O(T(\log T)^{(k+2)^2-1}).
\end{align*}

We now shift the contour to $\Re(s)=c:=1+1/\log T$. No poles are crossed: the possible zeta poles lie below the lower horizontal edge for generic shifts, and the final coincident-shift statement follows by holomorphic continuation.  The horizontal edges are $O(T(\log T)^{(k+2)^2-1})$ by Lemma~\ref{lemma:bounds} and the polynomial length bound; indeed $A(1-s)\ll y^{\sigma+\varepsilon}$, $A(s)\ll1$, and $\vartheta_k<1/(4k)$ give a power saving on these segments.  Hence

\begin{align*}
    \mathcal{J}_{k,1}(\alpha_1,\alpha_2,\alpha_3,\alpha_4) &=\left(\frac{T}{2\pi}\right)^{-(\alpha_3+\alpha_4)}\frac{1}{i}\int_{c+i}^{c+iT}\zeta(s+\alpha_1)\zeta(s+\alpha_2)\zeta(s-\alpha_3)\zeta(s-\alpha_4)
    \\& \hspace{3.5cm} \chi(1-s)A(1-s)^{k+1}A(s)^{k-1}\;ds +O(T(\log T)^{(k+2)^2-1}).
\end{align*}

We define the coefficients $a_k(m)$ by

$$\zeta(s+\alpha_1)\zeta(s+\alpha_2)\zeta(s-\alpha_3)\zeta(s-\alpha_4)A(s)^{k-1}=\sum_{m=1}^{\infty}\frac{a_k(m)}{m^s} \hspace{5mm} (\Re(s)>1).$$

Noting also that for $\Re(s)=c>1$ we have

\begin{equation*}
    A(1-s)^{k+1}=\sum_{n_1,\cdots,n_{k+1}\le y}\frac{P[n_1]\cdots P[n_{k+1}]}{(n_1\cdots n_{k+1})^{1-s}},
\end{equation*}

we may rewrite our integral as 

\begin{align*}
    \mathcal{J}_{k,1}(\alpha_1,\alpha_2,\alpha_3,\alpha_4)
    &=\left(\frac{T}{2\pi}\right)^{-(\alpha_3+\alpha_4)}\frac{1}{i}\int_{c+i}^{c+iT}\chi(1-s)\sum_{n_1,\cdots,n_{k+1}\le y}\frac{P[n_1]\cdots P[n_{k+1}]}{(n_1\cdots n_{k+1})^{1-s}}
    \\& \hspace{4cm} \sum_{m=1}^{\infty}\frac{a_k(m)}{m^s}\;ds +O(T(\log T)^{(k+2)^2-1}).
\end{align*}

Here $N\le y^{k+1}\le T^{(k+1)/(4k)-\varepsilon}<T^{1/2-\varepsilon}$ after decreasing $\varepsilon$ if necessary. Applying Lemma \ref{lemma:page_3.2} with $N=n_1\cdots n_{k+1}$ gives

\begin{align}
    \mathcal{J}_{k,1}(\alpha_1,\alpha_2,\alpha_3,\alpha_4) &=\left(\frac{T}{2\pi}\right)^{-(\alpha_3+\alpha_4)}2\pi\sum_{n_1,\cdots,n_{k+1}\le y}\frac{P[n_1]\cdots P[n_{k+1}]}{N} 
    \\
    &\hspace{5mm} \sum_{m\le\frac{NT}{2\pi}}a_k(m)e\left(-\frac{m}{N}\right)\notag +O(T(\log T)^{(k+2)^2-1}).\label{eqn:Jk1:CGG}
\end{align}

Our next step is to insert Dirichlet characters.  Following the approaches of Conrey, Ghosh and Gonek \cite{cgg} and Heap, Li and Zhao \cite{Heap_2022}, we write

\begin{equation}\label{eqn:expsum}
    e\left(-\frac{m}{N}\right)=\sum_{\substack{g|m\\g|N}}\frac{1}{\varphi(N/g)}\sum_{\chi\pmod{N/g}}\tau(\overline{\chi})\chi\left(\frac{-m}{g}\right).
\end{equation}

This separates the principal and non-principal characters. We use the following uniform consequence of the large-sieve and Siegel--Walfisz argument in \cite[Section 3]{page}: for moduli $N/g\leq N\leq y^{k+1}<T^{1/2-\varepsilon}$, shifts $|\alpha_i|\ll1/\log T$, and coefficients obtained by multiplying the fixed polynomial weights $P[n]$ with the shifted divisor coefficients above, the total non-principal contribution after summing over $g\mid N$ and $n_i\le y$ is $O_A(T(\log T)^{-A})$ for every fixed $A>0$. Thus
$$\mathcal{J}_{k,1}(\alpha_1,\alpha_2,\alpha_3,\alpha_4;\chi\ne\chi_0)\ll_A\frac{T}{(\log T)^A},$$
where $\chi_0$ denotes the principal character.

Thus we have (using the fact that $\tau(\chi_0)=\mu(N/g)$)

\begin{align*}
    \mathcal{J}_{k,1}(\alpha_1,\alpha_2,\alpha_3,\alpha_4)=&2\pi\left(\frac{T}{2\pi}\right)^{-(\alpha_3+\alpha_4)}\sum_{n_1,\cdots,n_{k+1}\le y}\frac{P[n_1]\cdots P[n_{k+1}]}{N}\sum_{m\le\frac{N T}{2\pi}}a_k(m)\\&\times \sum_{\substack{g|m\\ g|N}}\frac{\mu(N/g)}{\varphi(N/g)}\chi_0\left(\frac{-m}{g}\right)+O(T(\log T)^{(k+2)^2-1}).
\end{align*}

We can rewrite the inner sum by separating $g|N$ and $m=gm'$,

\begin{equation*}
    \sum_{m\le\frac{NT}{2\pi}}a_k(m)\sum_{\substack{g|m\\ g|N}}\frac{\mu(N/g)}{\varphi(N/g)}\chi_0\left(\frac{-m}{g}\right)=\sum_{g|N}\frac{\mu(N/g)}{\varphi(N/g)}\sum_{m'\le\frac{NT}{2\pi g}}a_k(gm')\chi_0(-m').
\end{equation*}

Applying Perron's formula in the shape of Lemma \ref{lem:perron} with $U=y^{k+1}$, the divisor bounds for $a_k(mg)$ and the inequality $N\le y^{k+1}<T^{1/2-\varepsilon}$ show that the horizontal and truncation errors are $O(T(\log T)^{(k+2)^2-1})$.  We have

\begin{align*}
 \mathcal{J}_{k,1}(\alpha_1,\alpha_2,\alpha_3,\alpha_4) = & 2\pi\left(\frac{T}{2\pi}\right)^{-(\alpha_3+\alpha_4)}\sum_{n_1,\cdots,n_{k+1}\le y}\frac{P[n_1]\cdots P[n_{k+1}]}{N} \frac{1}{2\pi i}\int_{c-iU}^{c+iU}\left(\frac{NT}{2\pi}\right)^s
 \\
 &\times \sum_{g|N}\frac{\mu(N/g)}{\varphi(N/g)}\sum_{m=1}^{\infty}\frac{a_k(mg)\chi_0(m)}{(mg)^s}\frac{ds}{s} +O(T(\log T)^{(k+2)^2-1}).  
\end{align*}
   
\vspace{5mm}

We now substitute the polynomial coefficients coming from $A(s)^{k-1}$. We use $v_1,\cdots,v_{k-1}$ for the Mellin variables attached to these $k-1$ copies. Using \eqref{mellin_polynomial} we have

\begin{align*}
    \mathcal{J}_{k,1}(\alpha_1,\alpha_2,\alpha_3,\alpha_4)&=2\pi\left(\frac{T}{2\pi}\right)^{-(\alpha_3+\alpha_4)}\sum_{j_1,\cdots,j_{k-1}}\frac{c_{j_1}\cdots c_{j_{k-1}}j_1!\cdots j_{k-1}!}{(\log y)^{j_1+\cdots+j_{k-1}}}
    \\
    &\times \sum_{n_1,\cdots,n_{k+1}\le y}\frac{P[n_1]\cdots P[n_{k+1}]}{N}\frac{1}{(2\pi i)^{k}}\int_{(2)^{k-1}}\int_{c-iU}^{c+iU}\left(\frac{NT}{2\pi}\right)^sy^{v_1+\cdots+v_{k-1}}
    \\
    &\times\sum_{g|N}\frac{\mu(N/g)}{\varphi(N/g)}\sum_{m=1}^{\infty}\frac{\chi_0(m)}{(gm)^s}\left(\sum_{ae_1\cdots e_{k-1}=gm}\sigma_{\alpha_1,\alpha_2,-\alpha_3,-\alpha_4}(a)e_1^{-v_1}\cdots e_{k-1}^{-v_{k-1}}\right)
    \\
    &\frac{ds}{s}\frac{dv_1}{v_1^{j_1+1}}\cdots\frac{dv_{k-1}}{v_{k-1}^{j_{k-1}+1}}+O(T(\log T)^{(k+2)^2-1}),
\end{align*}

where $\sigma_{\alpha_1,\alpha_2,-\alpha_3,-\alpha_4}(a)=\sum_{a=a_1a_2a_3a_4}a_1^{-\alpha_1}a_2^{-\alpha_2}a_3^{\alpha_3}a_4^{\alpha_4}$. Applying Lemma \ref{lemma:page_3.3} to the Dirichlet series in $m$, this series factors into polar zeta factors and a finite Euler factor. Thus

\begin{align*}
    & \sum_{g|N}\frac{\mu(N/g)}{\varphi(N/g)}\sum_{m=1}^{\infty}\frac{\chi_0(m)}{(gm)^s}\left(\sum_{ae_1\cdots e_{k-1}=gm}\sigma_{\alpha_1,\alpha_2,-\alpha_3,-\alpha_4}(a)e_1^{-v_1}\cdots e_{k-1}^{-v_{k-1}}\right)
    \\ 
    & = \zeta(s+\alpha_1)\zeta(s+\alpha_2)\zeta(s-\alpha_3)\zeta(s-\alpha_4)\prod_{j=1}^{k-1}\zeta(s+v_j)M(N,s,\underline{v}),
\end{align*}

where $M(N,s,\underline{v})$ denotes the finite Euler factor coming from the divisors $g|N$ and the associated coprimality conditions. Substituting this into the previous expression gives

\begin{align*}
    \mathcal{J}_{k,1}(\alpha_1,\alpha_2,\alpha_3, &\alpha_4)=2\pi\left(\frac{T}{2\pi}\right)^{-(\alpha_3+\alpha_4)}\sum_{j_1,\cdots,j_{k-1}}\frac{c_{j_1}\cdots c_{j_{k-1}}j_1!\cdots j_{k-1}!}{(\log y)^{j_1+\cdots +j_{k-1}}}\sum_{n_1,\cdots,n_{k+1}}\frac{P[n_1]\cdots P[n_{k+1}]}{N}
    \\
    &\times\frac{1}{(2\pi i)^k}\int_{(2)^{k-1}}\int_{c-iU}^{c+iU}\left(\frac{NT}{2\pi}\right)^sy^{v_1+\cdots+v_{k-1}}\zeta(s+\alpha_1)\zeta(s+\alpha_2)\zeta(s-\alpha_3)
    \\
    &\times \zeta(s-\alpha_4) \prod_{j=1}^{k-1}\zeta(s+v_j)M(N,s,\underline{v})\frac{ds}{s}\frac{dv_1}{v_1^{j_1+1}}\cdots\frac{dv_{k-1}}{v_{k-1}^{j_{k-1}+1}}+O(T(\log T)^{(k+2)^2-1}).
\end{align*}

We now substitute the remaining polynomial coefficients $P[n_1],\cdots,P[n_{k+1}]$. We use $u_1,\cdots,u_{k+1}$ for these variables. Writing $s=1+\omega$, we obtain

\begin{align*}
    \mathcal{J}_{k,1}(\alpha_1,\alpha_2,\alpha_3,\alpha_4)&=T\left(\frac{T}{2\pi}\right)^{-(\alpha_3+\alpha_4)}\sum_{i_1,\cdots,i_{k+1}}\sum_{j_1,\cdots,j_{k-1}}\frac{c_{i_1}\cdots c_{i_{k+1}}c_{j_1}\cdots c_{j_{k-1}}i_1!\cdots i_{k+1}!j_1!\cdots j_{k-1}!}{(\log y)^{i_1+\cdots +i_{k+1}+j_1+\cdots+j_{k-1}}}
    \\
    &\times\frac{1}{(2\pi i)^{2k+1}}\int_{(2)^{2k}}\int_{c-1-iU}^{c-1+iU}\left(\frac{T}{2\pi}\right)^{\omega}y^{u_1+\cdots+u_{k+1}+v_1+\cdots+v_{k-1}}
    \\
    &\times\zeta(1+\omega+\alpha_1)\zeta(1+\omega+\alpha_2)\zeta(1+\omega-\alpha_3)\zeta(1+\omega-\alpha_4)\prod_{j=1}^{k-1}\zeta(1+\omega+v_j)
    \\
    &\times\sum_{n_1,\cdots,n_{k+1}\le y}\frac{M(N,1+\omega,\underline{v})}{N^{-\omega}}\frac{1}{n_1^{u_1}\cdots n_{k+1}^{u_{k+1}}}
    \\
    &\frac{d\omega}{1+\omega}\frac{du_1}{u_1^{i_1+1}}\cdots \frac{du_{k+1}}{u_{k+1}^{i_{k+1}+1}}\frac{dv_1}{v_1^{j_1+1}}\cdots\frac{dv_{k-1}}{v_{k-1}^{j_{k-1}+1}}+O(T(\log T)^{(k+2)^2-1}).
\end{align*}

The sum over $n_1,\cdots,n_{k+1}$ may be written as a sum over $N$. The prime-wise calculation separates the polar factors below, while the remaining local factors are $1+O(p^{-2+\varepsilon})$; hence it factors as an Euler product of the form
\begin{align*}
    \sum_N\frac{M(N,1+\omega,\underline{v})\sigma_{u_2-u_1,\cdots,u_{k+1}-u_1}(N)}{N^{u_1-\omega}}&=G(\omega,\underline{u},\underline{v})\prod_{i=1}^{k+1}\prod_{j=1}^{k-1}\zeta(1+u_i+v_j)\\&\times\prod_{i=1}^{k+1}\frac{\zeta(1+u_i+\alpha_1)\zeta(1+u_i+\alpha_2)\zeta(1+u_i-\alpha_3)\zeta(1+u_i-\alpha_4)}{\zeta(1+u_i-\omega)},
\end{align*}

where $G(\omega,\underline{u},\underline{v})$ is holomorphic in a neighbourhood of the origin. Therefore

\begin{align*}
    \mathcal{J}_{k,1}(\alpha_1,\alpha_2,\alpha_3,\alpha_4)&=T\left(\frac{T}{2\pi}\right)^{-(\alpha_3+\alpha_4)}\sum_{i_1,\cdots,i_{k+1}}\sum_{j_1,\cdots,j_{k-1}}\frac{c_{i_1}\cdots c_{i_{k+1}}c_{j_1}\cdots c_{j_{k-1}}i_1!\cdots i_{k+1}!j_1!\cdots j_{k-1}!}{(\log y)^{i_1+\cdots+i_{k+1}+j_1+\cdots +j_{k-1}}}\\&\times\frac{1}{(2\pi i)^{2k+1}}\int_{(2)^{2k}}\int_{c-1-iU}^{c-1+iU}\left(\frac{T}{2\pi}\right)^{\omega}y^{u_1+\cdots+u_{k+1}+v_1+\cdots+v_{k-1}}\\&\times\zeta(1+\omega+\alpha_1)\zeta(1+\omega+\alpha_2)\zeta(1+\omega-\alpha_3)\zeta(1+\omega-\alpha_4)\\&\times\prod_{j=1}^{k-1}\zeta(1+\omega+v_j)G(\omega,\underline{u},\underline{v})\prod_{i=1}^{k+1}\prod_{j=1}^{k-1}\zeta(1+u_i+v_j)\\&\times\prod_{i=1}^{k+1}\frac{\zeta(1+u_i+\alpha_1)\zeta(1+u_i+\alpha_2)\zeta(1+u_i-\alpha_3)\zeta(1+u_i-\alpha_4)}{\zeta(1+u_i-\omega)}\\&\frac{d\omega}{1+\omega}\frac{du_1}{u_1^{i_1+1}}\cdots\frac{du_{k+1}}{u_{k+1}^{i_{k+1}+1}}\frac{dv_1}{v_1^{j_1+1}}\cdots \frac{dv_{k-1}}{v_{k-1}^{j_{k-1}+1}}+O(T(\log T)^{(k+2)^2-1}).
\end{align*}

We move the $\omega$-contour to $\Re(\omega)=-\varepsilon$. In doing so we cross the four poles at $-\alpha_1,-\alpha_2,\alpha_3$ and $\alpha_4$. On the new line the factor $(T/2\pi)^\omega$ gives a power saving, while $U=y^{k+1}<T^{1/2-\varepsilon}$ and the local zeta-factor bounds control the remaining variables; the horizontal segments are treated similarly. Hence the new contour contributes $O(T(\log T)^{(k+2)^2-1})$, and thus we may write
\begin{equation*}
    \mathcal{J}_{k,1}=\mathcal{J}_{k,1}^{-\alpha_1}+\mathcal{J}_{k,1}^{-\alpha_2} + \mathcal{J}_{k,1}^{\alpha_3}+\mathcal{J}_{k,1}^{\alpha_4}+O(T(\log T)^{(k+2)^2-1}).
\end{equation*}

We first evaluate the residue at $\omega=-\alpha_1$. As before, by holomorphicity we may replace $G(\omega,\underline{u},\underline{v})$ by its value at the origin, at the cost of an error term $O(T(\log T)^{(k+2)^2-1})$. Comparing the Euler product at zero with $\sum_{n=1}^{\infty}\frac{d_{k+2}(n)^2}{n^{1+2s}}$ shows that $G(0,\underline{0},\underline{0})=a_{k+2}$. At $\omega=-\alpha_1$, the factor $\zeta(1+\omega+\alpha_1)$ supplies the residue. Moreover, the factor $\zeta(1+u_i+\alpha_1)$ is cancelled by the denominator $\zeta(1+u_i-\omega)=\zeta(1+u_i+\alpha_1)$. Approximating the zeta functions we have

\begin{align*}
    \mathcal{J}_{k,1}^{-\alpha_1}&=a_{k+2}T\left(\frac{T}{2\pi}\right)^{-\alpha_1-\alpha_3-\alpha_4}\frac{1}{(\alpha_2-\alpha_1)(\alpha_1+\alpha_3)(\alpha_1+\alpha_4)} \sum_{i_1,\cdots,i_{k+1}}\sum_{j_1,\cdots,j_{k-1}}
    \\
    &\times\frac{c_{i_1}\cdots c_{i_{k+1}}c_{j_1}\cdots c_{j_{k-1}}i_1!\cdots i_{k+1}!j_1!\cdots j_{k-1}!}{(\log y)^{i_1+\cdots+i_{k+1}+j_1+\cdots+j_{k-1}}} \frac{1}{(2\pi i)^{2k}}\int_{(1/\log T)^{2k}}y^{u_1+\cdots+u_{k+1}+v_1+\cdots +v_{k-1}}
    \\
    &\times\prod_{i=1}^{k+1}\left[\zeta(1+u_i+\alpha_2)\zeta(1+u_i-\alpha_3)\zeta(1+u_i-\alpha_4)\right] \prod_{j=1}^{k-1}\zeta(1+v_j-\alpha_1)\prod_{i=1}^{k+1}\prod_{j=1}^{k-1}\zeta(1+u_i+v_j)
    \\
    &\times\frac{du_1}{u_1^{i_1+1}}\cdots\frac{du_{k+1}}{u_{k+1}^{i_{k+1}+1}}\frac{dv_1}{v_1^{j_1+1}}\cdots\frac{dv_{k-1}}{v_{k-1}^{j_{k-1}+1}}+O(T(\log T)^{(k+2)^2-1}).
\end{align*}

We now express any entangled zeta factor by the corresponding Dirichlet series and interchange the sums and integrals and collecting powers of the $u_i$ and $v_j$ terms, we obtain

\begin{align*}
    \mathcal{J}_{k,1}^{-\alpha_1}&=a_{k+2}T\left(\frac{T}{2\pi}\right)^{-\alpha_1-\alpha_3-\alpha_4}\frac{1}{(\alpha_2-\alpha_1)(\alpha_1+\alpha_3)(\alpha_1+\alpha_4)} \sum_{i_1,\cdots,i_{k+1}}\sum_{j_1,\cdots,j_{k-1}}
    \\
    &\times\frac{c_{i_1}\cdots c_{i_{k+1}}c_{j_1}\cdots c_{j_{k-1}}i_1!\cdots i_{k+1}!j_1!\cdots j_{k-1}!}{(\log y)^{i_1+\cdots+i_{k+1}+j_1+\cdots+j_{k-1}}} \sum_{\substack{m_{i,j}\ge 1\\ \prod_{j=1}^{k-1}m_{i,j}\le y\\\prod_{i=1}^{k+1}m_{i,j}\le y}}\left(\prod_{i=1}^{k+1}\prod_{j=1}^{k-1}\frac{1}{m_{i,j}}\right)
    \\
    &\times\frac{1}{(2\pi i)^{2k}}\int_{(1/\log T)^{2k}}\prod_{i=1}^{k+1}\left(\frac{y}{\prod_{j=1}^{k-1}m_{i,j}}\right)^{u_i}\prod_{j=1}^{k-1}\left(\frac{y}{\prod_{i=1}^{k+1}m_{i,j}}\right)^{v_j}
    \\
    &\times\prod_{i=1}^{k+1}\left[\zeta(1+u_i+\alpha_2)\zeta(1+u_i-\alpha_3)\zeta(1+u_i-\alpha_4)\right]\prod_{j=1}^{k-1}\zeta(1+v_j-\alpha_1)
    \\
    &\frac{du_1}{u_1^{i_1+1}}\cdots\frac{du_{k+1}}{u_{k+1}^{i_{k+1}+1}}\frac{dv_1}{v_1^{j_1+1}}\cdots\frac{dv_{k-1}}{v_{k-1}^{j_{k-1}+1}}+O(T(\log T)^{(k+2)^2-1}).
\end{align*}

We now evaluate the $u_i$ and $v_j$ pieces. For each fixed $1\le\ell\le k+1$, the $u_{\ell}$ piece is 
$$\sum_{r=0}^{\infty}\frac{c_rr!}{(\log y)^r}\frac{1}{2\pi i}\int_{(1/\log T)}\left(\frac{y}{\prod_{j=1}^{k-1}m_{\ell,j}}\right)^u\zeta(1+u+\alpha_2)\zeta(1+u-\alpha_3)\zeta(1+u-\alpha_4)\frac{du}{u^{r+1}}.$$
By Lemma \ref{bmthm4.1} with shifts $\alpha_2, -\alpha_3, -\alpha_4$ this equals
\begin{multline*}
    (\log y)^3\int_{\substack{x_{\ell},y_{\ell},z_{\ell}\ge0\\x_{\ell}+y_{\ell}+z_{\ell}\le 1}}\left(\frac{y}{\prod_{j=1}^{k-1}m_{\ell,j}}\right)^{-\alpha_2x_{\ell}+\alpha_3y_{\ell}+\alpha_4z_{\ell}}\left(\frac{\log(y/\prod_{j=1}^{k-1}m_{\ell,j})}{\log y}\right)^3\\\times P\left(\frac{\log(y/\prod_{j=1}^{k-1}m_{\ell,j})}{\log y}(1-x_{\ell}-y_{\ell}-z_{\ell})\right)dx_{\ell}dy_{\ell}dz_{\ell}+O((\log y)^2).
\end{multline*}
Similarly for each $1\le\ell\le k-1$, the $v_{\ell}$ piece is
$$\sum_{r=0}^{\infty}\frac{c_rr!}{(\log y)^r}\frac{1}{2\pi i}\int_{(1/\log T)}\left(\frac{y}{\prod_{i=1}^{k+1}m_{i,\ell}}\right)^v\zeta(1+v-\alpha_1)\frac{dv}{v^{r+1}}.$$
Once again applying Lemma \ref{bmthm4.1} with shift $-\alpha_1$, this equals
$$(\log y)\int_0^1\left(\frac{y}{\prod_{i=1}^{k+1}m_{i,\ell}}\right)^{\alpha_1w_{\ell}}\left(\frac{\log(y/\prod_{i=1}^{k+1}m_{i,\ell})}{\log y}\right)P\left(\frac{\log(y/\prod_{i=1}^{k+1}m_{i,\ell})}{\log y}(1-w_{\ell})\right)dw_{\ell}+O(1).$$
The total error from all of these evaluations is $O(T(\log T)^{(k+2)^2-1})$.

Substituting the main terms yields
\begin{align*}
    \mathcal{J}_{k,1}^{-\alpha_1}&=a_{k+2}T\left(\frac{T}{2\pi}\right)^{-\alpha_1-\alpha_3-\alpha_4}\frac{(\log y)^{3(k+1)+(k-1)}}{(\alpha_2-\alpha_1)(\alpha_1+\alpha_3)(\alpha_1+\alpha_4)}\int_{\substack{x_i,y_i,z_i\ge 0\\ x_i+y_i+z_i\le 1}}\int_{0\le w_j\le 1}\sum_{\substack{m_{i,j}\ge 1\\ \prod_{j=1}^{k-1}m_{i,j}\le y\\ \prod_{i=1}^{k+1}m_{i,j}\le y}}\left(\prod_{i=1}^{k+1}\prod_{j=1}^{k-1}\frac{1}{m_{i,j}}\right)\\&\times\prod_{i=1}^{k+1}\Bigg[\left(\frac{\log(y/\prod_{j=1}^{k-1}m_{i,j})}{\log y}\right)^3P\left(\frac{\log(y/\prod_{j=1}^{k-1}m_{i,j})}{\log y}(1-x_i-y_i-z_i)\right)\left(\frac{y}{\prod_{j=1}^{k-1}m_{i,j}}\right)^{-\alpha_2x_i+\alpha_3y_i+\alpha_4z_i}\Bigg]\\&\times\prod_{j=1}^{k-1}\Bigg[\left(\frac{\log(y/\prod_{i=1}^{k+1}m_{i,j})}{\log y}\right)P\left(\frac{\log(y/\prod_{i=1}^{k+1}m_{i,j})}{\log y}(1-w_j)\right)\left(\frac{y}{\prod_{i=1}^{k+1}m_{i,j}}\right)^{\alpha_1w_j}\Bigg]\\&dx_1\cdots dx_{k+1}dy_1\cdots dy_{k+1}dz_1\cdots dz_{k+1}dw_1\cdots dw_{k-1}+O(T(\log T)^{(k+2)^2-1}).
\end{align*}

For $k=1$ the matrix sum is empty and the following polytope integral is obtained directly.  For $k\geq2$, applying Lemma \ref{t_c} and interchanging the order of summation and integration gives the corresponding $(k+1)\times(k-1)$ polytope integral, with an error $O(T(\log T)^{(k+2)^2-1})$.  Finally we make the change of variables

\begin{align*}
    &x_i'=x_i\left(1-\sum_{j=1}^{k-1}t_{i,j}\right) \hspace{1cm}
     y_i' =y_i\left(1-\sum_{j=1}^{k-1}t_{i,j}\right),\\
    & z_i'=z_i\left(1-\sum_{j=1}^{k-1}t_{i,j}\right) \hspace{1cm}
    w_j'=w_j\left(1-\sum_{i=1}^{k+1}t_{i,j}\right) 
\end{align*}

this gives 

\begin{align*}
    \mathcal{J}_{k,1}^{-\alpha_1}&=a_{k+2}T\left(\frac{T}{2\pi}\right)^{-\alpha_1-\alpha_3-\alpha_4}\frac{(\log y)^{k^2+4k+1}}{(\alpha_2-\alpha_1)(\alpha_1+\alpha_3)(\alpha_1+\alpha_4)}\int_{\substack{0\le t_{i,j}\le 1\\ 0\le x_i,y_i,z_i\le 1\\ 0\le w_j\le 1}} 
    \\
    &\times  \int_{\substack{\sum_{j=1}^{k-1}t_{i,j}+x_i+y_i+z_i\le 1\\ \sum_{i=1}^{k+1}t_{i,j}+w_j\le 1}}y^{\alpha_1\sum_{j=1}^{k-1}w_j-\alpha_2\sum_{i=1}^{k+1}x_i+\alpha_3\sum_{i=1}^{k+1}y_i+\alpha_4\sum_{i=1}^{k+1}z_i}
    \\
    &\times \prod_{i=1}^{k+1}P\left(1-\sum_{j=1}^{k-1}t_{i,j}-x_i-y_i-z_i\right)\prod_{j=1}^{k-1}P\left(1-\sum_{i=1}^{k+1}t_{i,j}-w_j\right)
    \\
    & dt_{1,1}\cdots dt_{k+1,k-1}dx_1\cdots dx_{k+1}dy_1\cdots dy_{k+1}dz_1\cdots dz_{k+1}dw_1\cdots dw_{k-1}
    \\
    & +O(T(\log T)^{(k+2)^2-1}).
\end{align*}

The residues at $\omega=-\alpha_2, \omega=\alpha_3$ and $\omega=\alpha_4$ are treated in exactly the same way. In each case, the corresponding zeta factor supplies the residue, one of the four zeta factors in the $u_i$ piece is cancelled by the denominator $\zeta(1+u_i-\omega)$, and the same $(k+1)\times(k-1)$ matrix reduction applies. Combining the four residues yields the desired result.

\subsection{Proof of Theorem \ref{thm:Jk2}}

Upon allowing the two $\chi$-factors to act on the first two zeta-functions, and using Lemma \ref{lemma:bounds}, we have
\begin{align*}
\mathcal{J}_{k,2}(\alpha_1,\alpha_2,\alpha_3,\alpha_4)
= \bigg(\frac{T}{2\pi}\bigg)^{-\alpha_1-\alpha_2}
& \int_{-\infty}^{\infty}
\zeta(\tfrac12-\alpha_1-it)
 \zeta(\tfrac12-\alpha_2-it)
 \zeta(\tfrac12+\alpha_3-it)
 \zeta(\tfrac12+\alpha_4-it)
\\
&\times
A(\tfrac12+it)^{k+2}
A(\tfrac12-it)^{k-2}
\Phi(t/T)\,dt
+O\left(T(\log T)^{(k+2)^2-1}\right).
\end{align*}

For $t\asymp T$, Lemma \ref{lemma:bcy_sigma} gives

\begin{align*}
&\zeta(\tfrac12-\alpha_1-it)
 \zeta(\tfrac12-\alpha_2-it)
 \zeta(\tfrac12+\alpha_3-it)
 \zeta(\tfrac12+\alpha_4-it)
 =
\sum_{\ell=1}^{\infty}
\frac{\sigma_{-\alpha_1,-\alpha_2,\alpha_3,\alpha_4}(\ell)e^{-\ell/T^5}}
{\ell^{1/2-it}}
+O(T^{-1+\varepsilon}).
\end{align*}

Before substituting this expression, note that the $O(T^{-1+\varepsilon})$ error in Lemma \ref{lemma:bcy_sigma}, multiplied by the amplifier, contributes $O(T^{-1+\varepsilon})\int_0^T|A(1/2+it)|^{2k}dt$, which is absorbed by the mean-value bound for Dirichlet polynomials and the length restriction. Substituting the main term and expanding the Dirichlet polynomials gives
\begin{align*}
\mathcal{J}_{k,2}(\alpha_1,\alpha_2,\alpha_3,\alpha_4)
&= \bigg(\frac{T}{2\pi}\bigg)^{-\alpha_1-\alpha_2}
\sum_{a,b,c,d\ge1}
\sum_{\substack{e_1,\ldots,e_{k+2}\le y\\ f_1,\ldots,f_{k-2}\le y}}
\frac{
e^{-abcd/T^5}
P[e_1]\cdots P[e_{k+2}]
P[f_1]\cdots P[f_{k-2}]
}{
a^{1/2-\alpha_1}
b^{1/2-\alpha_2}
c^{1/2+\alpha_3}
d^{1/2+\alpha_4}
}
\\
&\times
\frac{1}{
(e_1\cdots e_{k+2})^{1/2}
(f_1\cdots f_{k-2})^{1/2}
}
\int_{-\infty}^{\infty}
\left(
\frac{abcdf_1\cdots f_{k-2}}{e_1\cdots e_{k+2}}
\right)^{it}
\Phi(t/T)\,dt
\\
&+O\left(T(\log T)^{(k+2)^2-1}\right).
\end{align*}

We separate the diagonal and off-diagonal contributions. Write $ E=e_1\cdots e_{k+2},\qquad F=f_1\cdots f_{k-2}. $ so that the diagonal is given by $abcdf_1\cdots f_{k-2}=e_1\cdots e_{k+2},$ or equivalently $abcdF=E$ hence the diagonal contribution is

\begin{align*}
\mathcal{D}
&=
T\widehat{\Phi}(0)
\bigg(\frac{T}{2\pi}\bigg)^{-\alpha_1-\alpha_2}
\sum_{\substack{e_1,\ldots,e_{k+2}\le y\\ f_1,\ldots,f_{k-2}\le y\\ abcdf_1\cdots f_{k-2}=e_1\cdots e_{k+2}}}
\frac{
e^{-abcd/T^5}
P[e_1]\cdots P[e_{k+2}]
P[f_1]\cdots P[f_{k-2}]
}{
a^{1/2-\alpha_1}
b^{1/2-\alpha_2}
c^{1/2+\alpha_3}
d^{1/2+\alpha_4}
}
\\
&\hspace{7cm} \times 
\frac{1}{
(e_1\cdots e_{k+2})^{1/2}
(f_1\cdots f_{k-2})^{1/2}
}.
\end{align*}

We now show that the off-diagonal contribution is negligible. By repeated integration by parts,
$$
\int_{-\infty}^{\infty}
\left(\frac{M}{E}\right)^{it}\Phi(t/T)\,dt
\ll_A
T\left(1+T\left|\log\left(\frac{M}{E}\right)\right|\right)^{-A},
$$
where $M=abcdF$.
Thus, using the divisor bound and the fact that the coefficients $P[n]$ are polynomial weights, the off-diagonal contribution satisfies the following estimate.  The shifted factors $a^{\pm\alpha_i},b^{\pm\alpha_i},c^{\pm\alpha_i},d^{\pm\alpha_i}$ are harmless here, since $|\alpha_i|\ll1/\log T$ and all variables in the effective range are $T^{O(1)}$, so they are absorbed into $m^\varepsilon$-type factors.
\begin{align*}
\mathcal{O}
\ll_{\varepsilon,A}
T
\sum_{\substack{a,b,c,d\ge1\\ e_1,\ldots,e_{k+2}\le y\\ f_1,\ldots,f_{k-2}\le y\\ M\ne E}}
&\frac{e^{-abcd/T^5}}
{a^{1/2-\varepsilon}
 b^{1/2-\varepsilon}
 c^{1/2-\varepsilon}
 d^{1/2-\varepsilon}
 E^{1/2-\varepsilon}
 F^{1/2-\varepsilon}}
\left(1+T\left|\log\left(\frac{M}{E}\right)\right|\right)^{-A}.
\end{align*}
If $M\le E/2$ or $M\ge 2E$, then
$$
\left|\log(M/E)\right|\ge \log 2,
$$
and this range contributes $O_A(T^{-100})$, on choosing $A$ sufficiently large. It remains to consider the near-diagonal range
$$
E/2<M<2E,
\qquad
M\ne E.
$$
In this range,
$$
\left|\log(M/E)\right|\asymp \frac{|M-E|}{E}.
$$
Fixing $e_1,\ldots,e_{k+2}$ and $f_1,\ldots,f_{k-2}$, using $d_4(m)\ll_\varepsilon m^\varepsilon$, and writing
$m=abcd$,
the inner sum over $a,b,c,d$ is bounded by
$$
\sum_{\substack{m\ge1\\ E/2<mF<2E\\ mF\ne E}}
m^{-1/2+\varepsilon}
\left(1+\frac{T|mF-E|}{E}\right)^{-A}.
$$
Since $E/2<mF<2E$, we have $m\asymp E/F$, so this is
$$
\ll_{\varepsilon}
\left(\frac{E}{F}\right)^{-1/2+\varepsilon}
\sum_{\substack{m\asymp E/F\\ mF\ne E}}
\left(1+\frac{T|mF-E|}{E}\right)^{-A}.
$$
The quantities $mF-E$ are non-zero multiples of $F$, hence
$$
\sum_{\substack{m\asymp E/F\\ mF\ne E}}
\left(1+\frac{T|mF-E|}{E}\right)^{-A}
\ll_A
\sum_{r\ge1}\left(1+\frac{TFr}{E}\right)^{-A}
\ll_A
\frac{E}{TF}.
$$
Here we used
$$
E\le y^{k+2}\le y^{2k}\ll T^{1/2},
$$
which follows from $\vartheta_k<1/(4k)$. Therefore the inner sum over $a,b,c,d$ is
$$
\ll_{\varepsilon}
\frac{E^{1/2+\varepsilon}}{TF^{1/2+\varepsilon}}.
$$
Substituting this back into the off-diagonal contribution gives
$$
\mathcal{O}
\ll_{\varepsilon}
T
\sum_{e_1,\ldots,e_{k+2}\le y}
\sum_{f_1,\ldots,f_{k-2}\le y}
\frac{1}{E^{1/2-\varepsilon}F^{1/2-\varepsilon}}
\frac{E^{1/2+\varepsilon}}{TF^{1/2+\varepsilon}}
\ll_{\varepsilon}
\sum_{e_1,\ldots,e_{k+2}\le y}
\sum_{f_1,\ldots,f_{k-2}\le y}
\frac{E^{2\varepsilon}}{F}.
$$

Now

$$
\sum_{f_1,\ldots,f_{k-2}\le y}\frac{1}{f_1\cdots f_{k-2}}
=
\left(\sum_{f\le y}\frac1f\right)^{k-2}
\ll(\log y)^{k-2},
$$
while
$$
\sum_{e_1,\ldots,e_{k+2}\le y}
(e_1\cdots e_{k+2})^{2\varepsilon}
\ll y^{k+2+\varepsilon}.
$$
Therefore
$$
\mathcal{O}
\ll_{\varepsilon}
y^{k+2+\varepsilon}(\log y)^{k-2}
\ll T^{1/2+\varepsilon}.
$$
It follows that
\begin{align*}
\mathcal{J}_{k,2}(\alpha_1,\alpha_2,\alpha_3,\alpha_4)
&=
T\widehat{\Phi}(0)
\bigg(\frac{T}{2\pi}\bigg)^{-\alpha_1-\alpha_2}
\sum_{\substack{e_1,\ldots,e_{k+2}\le y\\ f_1,\ldots,f_{k-2}\le y\\ abcdf_1\cdots f_{k-2}=e_1\cdots e_{k+2}}}
\frac{
e^{-abcd/T^5}
P[e_1]\cdots P[e_{k+2}]
P[f_1]\cdots P[f_{k-2}]
}{
a^{1/2-\alpha_1}
b^{1/2-\alpha_2}
c^{1/2+\alpha_3}
d^{1/2+\alpha_4}
}
\\
&\times
\frac{1}{
(e_1\cdots e_{k+2})^{1/2}
(f_1\cdots f_{k-2})^{1/2}
}+O\left(T(\log T)^{(k+2)^2-1}\right).
\end{align*}

By Mellin inversion,
$$
e^{-abcd/T^5}
=
\frac{1}{2\pi i}
\int_{(1)}
\Gamma(s)T^{5s}(abcd)^{-s}\,ds.
$$
Substituting this and interchanging the sum and integral gives
\begin{align*}
\mathcal{J}_{k,2}(\alpha_1,\alpha_2,\alpha_3,\alpha_4)
&=
\frac{T\widehat{\Phi}(0)}{2\pi i}
\bigg(\frac{T}{2\pi}\bigg)^{-\alpha_1-\alpha_2}
\int_{(1)}
\Gamma(s)T^{5s}
\sum_{\substack{e_1,\ldots,e_{k+2}\le y\\ f_1,\ldots,f_{k-2}\le y\\ abcdf_1\cdots f_{k-2}=e_1\cdots e_{k+2}}}
\\
& \times \frac{
P[e_1]\cdots P[e_{k+2}]
P[f_1]\cdots P[f_{k-2}]
}{
a^{1/2+s-\alpha_1}
b^{1/2+s-\alpha_2}
c^{1/2+s+\alpha_3}
d^{1/2+s+\alpha_4}
}
\\
&\times
\frac{1}{
(e_1\cdots e_{k+2})^{1/2}
(f_1\cdots f_{k-2})^{1/2}
}
\,ds+O\left(T(\log T)^{(k+2)^2-1}\right).
\end{align*}

We now substitute the Mellin representation \eqref{mellin_polynomial} for each occurrence of $P[e_i]$ and $P[f_j]$. This gives

\begin{align*}
\mathcal{J}_{k,2}(\alpha_1,\alpha_2,\alpha_3,\alpha_4)
&=
\frac{T\widehat{\Phi}(0)}{(2\pi i)^{2k+1}}
\bigg(\frac{T}{2\pi}\bigg)^{-\alpha_1-\alpha_2}
\sum_{i_1,\ldots,i_{k+2}}
\sum_{j_1,\ldots,j_{k-2}}
\frac{
c_{i_1}\cdots c_{i_{k+2}}
c_{j_1}\cdots c_{j_{k-2}}
i_1!\cdots i_{k+2}!
j_1!\cdots j_{k-2}!
}{
(\log y)^{i_1+\cdots+i_{k+2}+j_1+\cdots+j_{k-2}}
}
\\
&\times
\int_{(1)}\int_{(2)^{2k}}
\Gamma(s)T^{5s}
y^{u_1+\cdots+u_{k+2}+v_1+\cdots+v_{k-2}}
\sum_{abcdf_1\cdots f_{k-2}=e_1\cdots e_{k+2}}
\\
&\times
\frac{1}{
a^{1/2+s-\alpha_1}
b^{1/2+s-\alpha_2}
c^{1/2+s+\alpha_3}
d^{1/2+s+\alpha_4}
}
\frac{1}{
e_1^{1/2+u_1}\cdots e_{k+2}^{1/2+u_{k+2}}
f_1^{1/2+v_1}\cdots f_{k-2}^{1/2+v_{k-2}}
}
\\
&\times
ds\,
\frac{du_1}{u_1^{i_1+1}}\cdots
\frac{du_{k+2}}{u_{k+2}^{i_{k+2}+1}}
\frac{dv_1}{v_1^{j_1+1}}\cdots
\frac{dv_{k-2}}{v_{k-2}^{j_{k-2}+1}}
+O\left(T(\log T)^{(k+2)^2-1}\right).
\end{align*}

The inner constrained sum factors as an Euler product. Introducing a holomorphic arithmetic factor $\mathcal{A}$, this sum may be written as

\begin{align*}
&\mathcal{A}(\underline{u},\underline{v},\alpha_1,\alpha_2,\alpha_3,\alpha_4,s)
\prod_{\ell=1}^{k+2}\prod_{r=1}^{k-2}
\zeta(1+u_\ell+v_r)
\\
\times
&\prod_{\ell=1}^{k+2}
\zeta(1+u_\ell+s-\alpha_1)
\zeta(1+u_\ell+s-\alpha_2)
\zeta(1+u_\ell+s+\alpha_3)
\zeta(1+u_\ell+s+\alpha_4),
\end{align*}

where $\mathcal{A}$ converges absolutely in a product of half-planes containing the origin.

We now move the $u$ and $v$ contours to $\delta$ and the $s$ contour to $-\delta/2$. Since $\Gamma(s)$ has a simple pole at $s=0$ with residue $1$, we pick up the residue at $s=0$. The resulting line integral  contributes

$$
\ll T\,T^{-5\delta/2}y^{2k\delta}\ll T^{1-\varepsilon},
$$

Therefore

\begin{align*}
\mathcal{J}_{k,2}(\alpha_1,\alpha_2,\alpha_3,\alpha_4)
&=
\frac{T\widehat{\Phi}(0)}{(2\pi i)^{2k}}
\bigg(\frac{T}{2\pi}\bigg)^{-\alpha_1-\alpha_2}
\sum_{i_1,\ldots,i_{k+2}}
\sum_{j_1,\ldots,j_{k-2}}
\frac{
c_{i_1}\cdots c_{i_{k+2}}
c_{j_1}\cdots c_{j_{k-2}}
i_1!\cdots i_{k+2}!
j_1!\cdots j_{k-2}!
}{
(\log y)^{i_1+\cdots+i_{k+2}+j_1+\cdots+j_{k-2}}
}
\\
&\times
\int_{(\delta)^{2k}}
y^{u_1+\cdots+u_{k+2}+v_1+\cdots+v_{k-2}}
\mathcal{A}(\underline{u},\underline{v},\alpha_1,\alpha_2,\alpha_3,\alpha_4,0)
\\
&\times
\prod_{\ell=1}^{k+2}\prod_{r=1}^{k-2}
\zeta(1+u_\ell+v_r)
\prod_{\ell=1}^{k+2}
\zeta(1+u_\ell-\alpha_1)
\zeta(1+u_\ell-\alpha_2)
\zeta(1+u_\ell+\alpha_3)
\zeta(1+u_\ell+\alpha_4)
\\
&\times
\frac{du_1}{u_1^{i_1+1}}\cdots
\frac{du_{k+2}}{u_{k+2}^{i_{k+2}+1}}
\frac{dv_1}{v_1^{j_1+1}}\cdots
\frac{dv_{k-2}}{v_{k-2}^{j_{k-2}+1}}+O\left(T(\log T)^{(k+2)^2-1}\right).
\end{align*}

We may now replace the arithmetic factor by its value at the origin, by moving the contours to $\Re(u_\ell)=\Re(v_r)\asymp 1/\log T$ and using holomorphicity. On these contours the holomorphic difference is $O(1/\log T)$, while the polar zeta factors contribute only the displayed total power of $\log T$; hence the replacement introduces an error term of size $O(T(\log T)^{(k+2)^2-1})$. Comparison with the Euler product for $a_{k+2}$ gives $\mathcal{A}(0) = a_{k+2}.$

We may also expand the entangled zeta factors, interchange sums and integrals, valid by absolute convergence and collect terms corresponding to the same powers of $u_\ell$ and $v_r$ to obtain

\begin{align*}
\mathcal{J}_{k,2}(\alpha_1,\alpha_2,\alpha_3,\alpha_4)
&=
a_{k+2}T\widehat{\Phi}(0)
\bigg(\frac{T}{2\pi}\bigg)^{-\alpha_1-\alpha_2}
\sum_{i_1,\ldots,i_{k+2}}
\sum_{j_1,\ldots,j_{k-2}}
\frac{
c_{i_1}\cdots c_{i_{k+2}}
c_{j_1}\cdots c_{j_{k-2}}
i_1!\cdots i_{k+2}!
j_1!\cdots j_{k-2}!
}{
(\log y)^{i_1+\cdots+i_{k+2}+j_1+\cdots+j_{k-2}}
}
\\
&\times
\sum_{\substack{
m_{\ell,r}\ge1\\
\prod_{r=1}^{k-2}m_{\ell,r}\le y\ (1\le \ell\le k+2)\\
\prod_{\ell=1}^{k+2}m_{\ell,r}\le y\ (1\le r\le k-2)
}}
\left(\prod_{\ell=1}^{k+2}\prod_{r=1}^{k-2}\frac{1}{m_{\ell,r}}\right)
\frac{1}{(2\pi i)^{2k}}
\int_{(1/\log T)^{2k}}
\prod_{\ell=1}^{k+2}
\left(
\frac{y}{\prod_{r=1}^{k-2}m_{\ell,r}}
\right)^{u_\ell}
\\
&\times
\prod_{r=1}^{k-2}
\left(
\frac{y}{\prod_{\ell=1}^{k+2}m_{\ell,r}}
\right)^{v_r}
\prod_{\ell=1}^{k+2}
\zeta(1+u_\ell-\alpha_1)
\zeta(1+u_\ell-\alpha_2)
\zeta(1+u_\ell+\alpha_3)
\zeta(1+u_\ell+\alpha_4)
\\
&
\frac{du_1}{u_1^{i_1+1}}\cdots
\frac{du_{k+2}}{u_{k+2}^{i_{k+2}+1}}
\frac{dv_1}{v_1^{j_1+1}}\cdots
\frac{dv_{k-2}}{v_{k-2}^{j_{k-2}+1}}
+O\left(T(\log T)^{(k+2)^2-1}\right).
\end{align*}

We first evaluate the $u_\ell$-integrals. For each $1\le \ell\le k+2$, the corresponding block is

\begin{align*}
&\sum_{q\ge0}
\frac{c_q q!}{(\log y)^q}
\frac{1}{2\pi i}
\int_{(1/\log T)}
\left(
\frac{y}{\prod_{r=1}^{k-2}m_{\ell,r}}
\right)^u
\zeta(1+u-\alpha_1)
\zeta(1+u-\alpha_2)
\zeta(1+u+\alpha_3)
\zeta(1+u+\alpha_4)
\frac{du}{u^{q+1}}.
\end{align*}

By Lemma \ref{bmthm4.1}, with shifts $ -\alpha_1,\ -\alpha_2,\ \alpha_3,\ \alpha_4$, this block is

\begin{align*}
&(\log y)^4
\int_{\substack{w_\ell,x_\ell,y_\ell,z_\ell\ge0\\
w_\ell+x_\ell+y_\ell+z_\ell\le1}}
\left(
\frac{y}{\prod_{r=1}^{k-2}m_{\ell,r}}
\right)^{\alpha_1w_\ell+\alpha_2x_\ell-\alpha_3y_\ell-\alpha_4z_\ell}
\\
&\hspace{15mm}\times
\left(
\frac{
\log\left(y/\prod_{r=1}^{k-2}m_{\ell,r}\right)
}{\log y}
\right)^4
P\left(
\frac{
\log\left(y/\prod_{r=1}^{k-2}m_{\ell,r}\right)
}{\log y}
(1-w_\ell-x_\ell-y_\ell-z_\ell)
\right)
\\
&\hspace{15mm}
dw_\ell\,dx_\ell\,dy_\ell\,dz_\ell
+
O((\log y)^3).
\end{align*}

It remains to evaluate the $v_r$-integrals. For each $1\le r\le k-2$, the corresponding block contains no zeta factor and by Mellin inversion can be written

$$
\sum_{q\ge0}
\frac{c_q q!}{(\log y)^q}
\frac{1}{2\pi i}
\int_{(1/\log T)}
\left(
\frac{y}{\prod_{\ell=1}^{k+2}m_{\ell,r}}
\right)^v
\frac{dv}{v^{q+1}}
= 
P\left(
\frac{
\log\left(y/\prod_{\ell=1}^{k+2}m_{\ell,r}\right)
}{\log y}
\right)
$$

when
$$
\prod_{\ell=1}^{k+2}m_{\ell,r}\le y,
$$
and is zero otherwise.

The accumulated error from the $u_\ell$-blocks is
$
O\left(T(\log T)^{(k+2)^2-1}\right).
$
Thus
\begin{align*}
\mathcal{J}_{k,2}(\alpha_1,\alpha_2,\alpha_3,\alpha_4)
&=
a_{k+2}T\widehat{\Phi}(0)
\bigg(\frac{T}{2\pi}\bigg)^{-\alpha_1-\alpha_2}
(\log y)^{4(k+2)}
\\
&\times
\int_{\substack{
w_\ell,x_\ell,y_\ell,z_\ell\ge0\\
w_\ell+x_\ell+y_\ell+z_\ell\le1\ (1\le \ell\le k+2)
}}
\sum_{\substack{
m_{\ell,r}\ge1\\
\prod_{r=1}^{k-2}m_{\ell,r}\le y\ (1\le \ell\le k+2)\\
\prod_{\ell=1}^{k+2}m_{\ell,r}\le y\ (1\le r\le k-2)
}}
\left(\prod_{\ell=1}^{k+2}\prod_{r=1}^{k-2}\frac{1}{m_{\ell,r}}\right)
\\
&\times
\prod_{\ell=1}^{k+2}
\bigg[
\left(
\frac{
\log\left(y/\prod_{r=1}^{k-2}m_{\ell,r}\right)
}{\log y}
\right)^4
P\left(
\frac{
\log\left(y/\prod_{r=1}^{k-2}m_{\ell,r}\right)
}{\log y}
(1-w_\ell-x_\ell-y_\ell-z_\ell)
\right)
\\
&\hspace{40mm}\times
\left(
\frac{y}{\prod_{r=1}^{k-2}m_{\ell,r}}
\right)^{\alpha_1w_\ell+\alpha_2x_\ell-\alpha_3y_\ell-\alpha_4z_\ell}
\bigg]
\\
&\times
\prod_{r=1}^{k-2}
P\left(
\frac{
\log\left(y/\prod_{\ell=1}^{k+2}m_{\ell,r}\right)
}{\log y}
\right)
\\
& dw_1 \cdots dw_{k+2} \, dx_1 \cdots dx_{k+2} 
\,dy_1 \cdots dy_{k+2} \,dz_1 \cdots dz_{k+2}
+O\left(T(\log T)^{(k+2)^2-1}\right).
\end{align*}

For $k=2$ the sum over the $m_{\ell,r}$ is empty and the following formula is immediate.  For $k\geq3$, applying Lemma \ref{t_c} and interchanging the order of summation and integration gives

\begin{align*}
\mathcal{J}_{k,2}(\alpha_1,\alpha_2,\alpha_3,\alpha_4)
&=
a_{k+2}T\widehat{\Phi}(0)
\bigg(\frac{T}{2\pi}\bigg)^{-\alpha_1-\alpha_2}
(\log y)^{(k+2)^2}
\\
&\times
\int_{\substack{
w_\ell,x_\ell,y_\ell,z_\ell\ge0\\
w_\ell+x_\ell+y_\ell+z_\ell\le1\ (1\le \ell\le k+2)
}}
\int_{\substack{
0\le t_{\ell,r}\le1\\
\sum_{r=1}^{k-2}t_{\ell,r}\le1\\
\sum_{\ell=1}^{k+2}t_{\ell,r}\le1
}}
\\
&\times
\prod_{\ell=1}^{k+2}
\bigg[
\left(1-\sum_{r=1}^{k-2}t_{\ell,r}\right)^4
P\left(
\left(1-\sum_{r=1}^{k-2}t_{\ell,r}\right)
(1-w_\ell-x_\ell-y_\ell-z_\ell)
\right)
\\
&\hspace{36mm}\times
y^{\left(1-\sum_{r=1}^{k-2}t_{\ell,r}\right)
(\alpha_1w_\ell+\alpha_2x_\ell-\alpha_3y_\ell-\alpha_4z_\ell)}
\bigg]
\\
&\times
\prod_{r=1}^{k-2}
P\left(
1-\sum_{\ell=1}^{k+2}t_{\ell,r}
\right)
dt_{1,1} \cdots dt_{k+2, k-2}
dw_1 \cdots dw_{k+2} \, dx_1 \cdots dx_{k+2} 
\\
& \,dy_1 \cdots dy_{k+2} \,dz_1 \cdots dz_{k+2}  +O\left(T(\log T)^{(k+2)^2-1}\right).
\end{align*}

Making the change of variables
$$
w_\ell'
=
w_\ell\left(1-\sum_{r=1}^{k-2}t_{\ell,r}\right),
\qquad
x_\ell'
=
x_\ell\left(1-\sum_{r=1}^{k-2}t_{\ell,r}\right),
$$
$$
y_\ell'
=
y_\ell\left(1-\sum_{r=1}^{k-2}t_{\ell,r}\right),
\qquad
z_\ell'
=
z_\ell\left(1-\sum_{r=1}^{k-2}t_{\ell,r}\right),
$$

the Jacobian cancels the row factors, and this gives the formula in Theorem \ref{thm:Jk2}.

\section{Deduction of lower bounds for moments and optimisation}\label{section:corollaries}
In this section we discuss briefly the strategy behind the proofs of the main numerical lower bounds, including the choice of polynomial functions. The general strategy in each case is to apply H\"older's inequality to the twisted second and fourth moments, the choice of inequality depending on which particular moment we are estimating. After applying the relevant inequality, we have a lower bound for the relevant moment which depends on a quotient of twisted moments which we have evaluated in Theorems \ref{thm:O(T)_ge_2,4}-\ref{thm:Jk2}.

The quotients below are first interpreted on a fixed dyadic interval with the smooth weight inserted.  The powers of $\widehat{\Phi}(0)$ left by H\"older give the expected interval-mass factor; taking smooth minorants with $\widehat{\Phi}(0)\to 1$, and then using the smoothing argument described in the introduction, passes the bounds to the sharp interval.

All polytope integrals were evaluated in Wolfram Mathematica by expanding the relevant integrands as polynomials and applying exact monomial integral formulae on the corresponding simplices. The optimisation over the coefficients of $P$ is then carried out inside Mathematica, and the lower bounds are rounded downward from the values obtained in the auxiliary file.

For the joint moments, we use the notation
$$
 \mathcal J_{1,0}^{(r_1,r_2,r_3,r_4)}
 =
 \left.
 \frac{\partial^{r_1+r_2+r_3+r_4}}
 {\partial\alpha_1^{r_1}\partial\alpha_2^{r_2}
  \partial\alpha_3^{r_3}\partial\alpha_4^{r_4}}
 \mathcal J_{1,0}(\alpha_1,\alpha_2,\alpha_3,\alpha_4)
 \right|_{\underline{\alpha}=0},
$$
and
$$
 \mathcal I_{2,0}^{(r_1,r_2)}
 =
 \left.
 \frac{\partial^{r_1+r_2}}
 {\partial\alpha_1^{r_1}\partial\alpha_2^{r_2}}
 \mathcal I_{2,0}(\alpha_1,\alpha_2)
 \right|_{\underline{\alpha}=0}.
$$
In every quotient below, the same polynomial $P$ is used in the numerator and denominator.

\begin{itemize}
\item Theorem \ref{cor:6thmoment} follows from Cauchy--Schwarz
with the two-piece amplifier, giving
$$
M_3(T)\ge
\frac{\left(2\mathcal{J}_{1,0}(\underline{0})+
2\Re\mathcal{J}_{1,1}(\underline{0})\right)^2}
{6\mathcal{I}_{2,0}(\underline{0})
+8\Re\mathcal{I}_{2,1}(\underline{0})
+2\Re\mathcal{I}_{2,2}(\underline{0})}.
$$
By Theorem \ref{thm:O(T)_ge_2,4}, the term
$\mathcal{I}_{2,2}(\underline{0})$ contributes only a lower-order
term and hence does not enter the leading numerical coefficient.

For $j\geq 0$, let
$$
L_j(x)=\mathrm{P}_j(2x-1),
$$
where $\mathrm{P}_j$ is the $j$th Legendre polynomial. A numerical
search over polynomials of degree at most $24$ produced the following
explicit rational polynomial:
$$
P_{24}(x)
=
1+\frac{1}{10^8}
\sum_{j=1}^{24}m_j\left(L_j(x)-(-1)^j\right),
$$
where
$$
\begin{aligned}
(a_1,\ldots,a_{24})={}&
(1021570,\ 2227094,\ 1919933,\ 4336237,\ 2288571,\ 6635924,\\
&2206312,\ 8938549,\ 1786302,\ 11021434,\ 1162267,\ 12652488,\\
&474847,\ 13612604,\ -141737,\ 13705439,\ -572376,\ 12762071,\\
&-741166,\ 10693395,\ -638482,\ 7599276,\ -334182,\ 3795073).
\end{aligned}
$$
Since $L_j(0)=(-1)^j$, this normalisation gives $P_{24}(0)=1$. The shifted-Legendre basis is used as a numerically
stable parametrisation; $P_{24}$ is an ordinary fixed real polynomial
of degree $24$.

Take $y=T^\eta$, where $\eta<1/4$ is fixed. Substitution of
$P_{24}$ into the limiting leading-term quotient at $\eta=1/4$
shows that its denominator is positive and gives
$$
34.401208532194852234760578642045606\ldots.
$$
By continuity it follows that
$$
M_3(T)\ge
\left(34.4012+o(1)\right)c_3T(\log T)^9.
$$
    \item Theorem \ref{cor:6thmoment_deriv} follows from the same two-piece Cauchy--Schwarz inequality,
    $$
    \int_0^T|\zeta'(\tfrac12+it)|^6\,dt
    \ge
    \frac{\left(2\mathcal{J}'_{1,0}(\underline{0})+
    2\Re\mathcal{J}'_{1,1}(\underline{0})\right)^2}
    {6\mathcal{I}'_{2,0}(\underline{0})
    +8\Re\mathcal{I}'_{2,1}(\underline{0})
    +2\Re\mathcal{I}'_{2,2}(\underline{0})},
    $$
    where the primes denote the shift differentiations which put one derivative on each zeta factor.  With
    $$
    P(x)=1-18.88126644x+107.7366132x^2-261.1612304x^3
       +283.2956701x^4-113.1018636x^5,
    $$
    this gives
    $$
      \int_0^T|\zeta'(\tfrac12+it)|^6\,dt
      \ge 0.54958\,c_3T(\log T)^{15}.
    $$

    \item Theorem \ref{cor:6thmoment_secondderiv} follows similarly, with
    $$
    \int_0^T|\zeta''(\tfrac12+it)|^6\,dt
    \ge
    \frac{\left(2\mathcal{J}''_{1,0}(\underline{0})+
    2\Re\mathcal{J}''_{1,1}(\underline{0})\right)^2}
    {6\mathcal{I}''_{2,0}(\underline{0})
    +8\Re\mathcal{I}''_{2,1}(\underline{0})
    +2\Re\mathcal{I}''_{2,2}(\underline{0})},
    $$
    where the double primes denote the shift differentiations which put two derivatives on each zeta factor.  With
    $$
    P(x)=1-18.70519234x+106.5589085x^2-258.1529006x^3
       +280.0074491x^4-111.7956890x^5,
    $$
    this gives
    $$
      \int_0^T|\zeta''(\tfrac12+it)|^6\,dt
      \ge 0.023142\,c_3T(\log T)^{21}.
    $$

    \item For $\int_0^T|\zeta(\tfrac12+it)|^2|\zeta'(\tfrac12+it)|^4\,dt$, Cauchy--Schwarz with a one-piece amplifier gives
    $$
      \int_0^T|\zeta|^2|\zeta'|^4\,dt
      \ge
      \frac{\left(\mathcal J_{1,0}^{(0,1,1,0)}\right)^2}
      {\mathcal I_{2,0}^{(0,0)}}.
    $$
    Optimising over $\deg P\le 10$ gives the bound
    $$
      \int_0^T|\zeta|^2|\zeta'|^4\,dt
      \ge 2.2167\,c_3T(\log T)^{13}.
    $$

    \item For $\int_0^T|\zeta'(\tfrac12+it)|^2|\zeta(\tfrac12+it)|^4\,dt$, the quotient is
    $$
      \int_0^T|\zeta'|^2|\zeta|^4\,dt
      \ge
      \frac{\left(\mathcal J_{1,0}^{(0,1,1,0)}\right)^2}
      {\mathcal I_{2,0}^{(1,1)}}.
    $$
    Optimising over $\deg P\le 10$ gives
    $$
      \int_0^T|\zeta'|^2|\zeta|^4\,dt
      \ge 7.8288\,c_3T(\log T)^{11}.
    $$

    \item For $\int_0^T|\zeta|^2|\zeta''|^4\,dt$, we use
    $$
      \int_0^T|\zeta|^2|\zeta''|^4\,dt
      \ge
      \frac{\left(\mathcal J_{1,0}^{(0,2,2,0)}\right)^2}
      {\mathcal I_{2,0}^{(0,0)}}.
    $$
    Optimising over $\deg P\le 10$ gives
    $$
      \int_0^T|\zeta|^2|\zeta''|^4\,dt
      \ge 0.28647\,c_3T(\log T)^{17}.
    $$

    \item For $\int_0^T|\zeta''|^2|\zeta|^4\,dt$, we use
    $$
      \int_0^T|\zeta''|^2|\zeta|^4\,dt
      \ge
      \frac{\left(\mathcal J_{1,0}^{(0,2,2,0)}\right)^2}
      {\mathcal I_{2,0}^{(2,2)}}.
    $$
    Optimising over $\deg P\le 10$ gives
    $$
      \int_0^T|\zeta''|^2|\zeta|^4\,dt
      \ge 2.4485\,c_3T(\log T)^{13}.
    $$

    \item For $\int_0^T|\zeta'|^2|\zeta''|^4\,dt$, we use
    $$
      \int_0^T|\zeta'|^2|\zeta''|^4\,dt
      \ge
      \frac{\left(\mathcal J_{1,0}^{(1,2,2,1)}\right)^2}
      {\mathcal I_{2,0}^{(1,1)}}.
    $$
    Optimising over $\deg P\le 10$ gives
    $$
      \int_0^T|\zeta'|^2|\zeta''|^4\,dt
      \ge 0.072658\,c_3T(\log T)^{19}.
    $$

    \item For $\int_0^T|\zeta''|^2|\zeta'|^4\,dt$, we use
    $$
      \int_0^T|\zeta''|^2|\zeta'|^4\,dt
      \ge
      \frac{\left(\mathcal J_{1,0}^{(1,2,2,1)}\right)^2}
      {\mathcal I_{2,0}^{(2,2)}}.
    $$
    Optimising over $\deg P\le 10$ gives
    $$
      \int_0^T|\zeta''|^2|\zeta'|^4\,dt
      \ge 0.17573\,c_3T(\log T)^{17}.
    $$
\end{itemize}
\section{Proof of Theorem \ref{thm:sound_bounds_unconditional}}\label{sect:sound_unconditional}

\subsection{Removing the Lindel\texorpdfstring{\"o}{o}f Hypothesis from the needed \texorpdfstring{$J_4$}{J4}-estimate}

Let $k\geq 4$, put $r=k-1$, let $0<\theta<1$, and define
$$
        A_r(s,P;\theta)=
        \sum_{n\leq T^\theta}
        \frac{d_r(n)P(\log n/(\theta\log T))}{n^s},
$$
where $P$ is a fixed polynomial.  Soundararajan's $J$-amplifier contains
the oscillating term
$$
J_4=\frac{1}{i}\int_{\tfrac{1}{2}+iT}^{\tfrac{1}{2}+2iT}
\zeta(s)^2\chi(s)^{r-1}A_r(1-s,P;\theta)^2ds.
$$

\begin{prop}\label{prop:j4-kappa}
Assume that, for some $\kappa\geq0$,
$$\zeta(\tfrac{1}{2}+it)\ll_\varepsilon t^{\kappa+\varepsilon}.
$$
Then, for every fixed polynomial $P$, $J_4=O(T^{1-\delta})$ whenever $\theta<1-2\kappa$.

In particular, convexity gives this estimate throughout the range
$\theta<1/2$.  Weyl's bound gives the larger range $\theta<2/3$.
\end{prop}

\begin{proof}
We move the line of integration to $\Re s=\mu$, where $\mu>1$ is fixed. The contour has two horizontal pieces and the vertical side $\Re s=\mu$.

First let $s=\sigma+i\tau$ lie on a horizontal piece with
$T\leq \tau\leq 2T$ and $1/2\leq\sigma\leq 1$.  By Phragmen--Lindel\"of,
the assumed critical-line bound implies
$$
\zeta(\sigma+i\tau)\ll_\varepsilon
        T^{2\kappa(1-\sigma)+\varepsilon}.
$$
Moreover
$$
        A_r(1-s,P;\theta)\ll_\varepsilon T^{\theta\sigma+\varepsilon},
        \qquad
        \chi(s)^{r-1}\ll T^{(r-1)(1/2-\sigma)}.
$$
Therefore the integrand is $\ll_\varepsilon T^{E(\sigma)+\varepsilon}$,
where
$$
        E(\sigma)
        =
        4\kappa(1-\sigma)+2\theta\sigma
        +(r-1)(1/2-\sigma).
$$
Since $k\geq 4$, $r-1\geq 2$.  If $\theta<1-2\kappa$, then
$$
        E'(\sigma)=2\theta-4\kappa-(r-1)<0.
$$
Thus the maximum on $[1/2,1]$ is
$$
        E(1/2)=\theta+2\kappa<1.
$$
The contribution of this part of each horizontal side is therefore $O(T^{1-\delta})$ for some $\delta>0$.

On the remaining part $1\leq\sigma\leq\mu$, we use
$$
        \zeta(\sigma+i\tau)\ll_{\mu,\varepsilon}T^\varepsilon .
$$
The exponent to consider is then
$$
        2\theta\sigma+(r-1)(1/2-\sigma).
$$
This is decreasing in $\sigma$, and at $\sigma=1$ it equals
$2\theta-\frac{r-1}{2}$.
Since $\theta<1$ and $r-1\geq2$, this number is $<1$.  Hence this
part of each horizontal side is also $O(T^{1-\delta})$.

On the vertical line $\Re s=\mu$, $\zeta(\mu+it)\ll_\mu 1$, and the
Montgomery--Vaughan mean-value theorem for Dirichlet polynomials \cite{MontgomeryVaughan} gives
$$
\begin{aligned}
&\int_{\mu+iT}^{\mu+2iT}
\zeta(s)^2\chi(s)^{r-1}A_r(1-s,P;\theta)^2ds\\
&\quad\ll_{\mu,\varepsilon}
T^{1+(r-1)(1/2-\mu)+\varepsilon}
	\sum_{n\leq T^\theta}\frac{d_r(n)^2}{n^{2-2\mu}}\\
&\quad\ll_{\mu,\varepsilon}
T^{1+(r-1-2\theta)(1/2-\mu)+\varepsilon}.
\end{aligned}
$$
Because $k\geq 4$ and $\theta<1$, we have $r-1-2\theta=k-2-2\theta>0$.
Since $\mu>1$, the last displayed exponent is $<1$ after taking $\varepsilon$ sufficiently small.  Hence the vertical side is also
$O(T^{1-\delta})$.  Consequently $J_4=O(T^{1-\delta})$.
\end{proof}

\begin{cor}\label{cor:sound-k456-uncond}
Soundararajan's numerical bounds
$$
        M_4(T)\geq (410+o(1))c_4T(\log T)^{16},\qquad
        M_5(T)\geq (6484+o(1))c_5T(\log T)^{25},
$$
and
$$
        M_6(T)\geq (56260+o(1))c_6T(\log T)^{36}
$$
are unconditional.
\end{cor}

\begin{proof}
The asymptotic formulae for the corresponding $J_2$ and $J_3$ terms as in \cite{conrey_ghosh_mean_values_iii} are used only for $\theta<1/2$.  Convexity gives
$\kappa=1/4$, so Proposition~\ref{prop:j4-kappa} makes $J_4$ negligible
throughout exactly that range.  The same polynomial choices and numerical
evaluations as in Soundararajan's corollary therefore give the same bounds for
$\theta<1/2$.  Taking $\theta=1/2-\delta$, using constants rounded below,
and then letting $\delta\downarrow0$ gives the three displayed bounds without
assuming the Lindel\"of Hypothesis.
\end{proof}

\subsection{The two-piece refinement}

Let $k\ge 3$ be an integer, let $0<\theta<1/2$, and let $0<\alpha<1$.
Put
$$
        B(s)=A_{k-1}(s,P;\theta),
        \qquad
        A_k(s;\alpha)=\sum_{n\leq T^\alpha}\frac{d_k(n)}{n^s}.
$$
Define
$f(s)=\zeta(s)^k$,
$$
        U_\alpha(s)
        =
        A_k(s;\alpha)+\chi(s)^kA_k(1-s;\alpha),
$$
and
$$
        V_{\theta,P}(s)
        =
        \zeta(s)\{B(s)+\chi(s)^{k-1}B(1-s)\}.
$$
We define the inner product
$$
        \langle F,G\rangle
        =\frac{1}{i}\int_{\tfrac{1}{2}+iT}^{\tfrac{1}{2}+2iT}F(s)G(1-s)ds.
$$

Set $K=k(k-1)$.  For a polynomial $P$, define
$$
\mathcal A_{k,\theta}(P)=
\frac{\Gamma(k^2+1)\theta^K}{\Gamma(k+1)\Gamma(K)}
\int_0^1x^{K-1}
{}_2F_1(1-k,-k;K;-\theta x)P(x)dx,
$$
$$
\mathcal C_{k,\alpha,\theta}(P)=
\frac{\Gamma(k^2+1)\theta^K}{\Gamma(k+1)\Gamma(K)}
\int_0^1x^{K-1}(\alpha-\theta x)^kP(x)dx,
$$
and write $\mathcal C_{k,\theta}=\mathcal C_{k,1,\theta}$.  Finally define
$$
        h_P(x)=\int_x^1(z-x)^{k-1}P(z)dz
$$
and
$$
\mathcal B_{k,\theta}(P_1,P_2)=
\frac{\Gamma(k^2+1)\theta^{k^2}}{\Gamma(k)^2\Gamma((k-1)^2)}
\int_0^1x^{(k-1)^2-1}
\left(
\frac{h_{P_1}'(x)h_{P_2}'(x)}{\theta}
+2(k-1)h_{P_1}(x)h_{P_2}'(x)
\right)dx .
$$

\begin{lemma}\label{lem:mixed-products}
For fixed $k\ge 3$, $0<\theta<1/2$, $\theta<\alpha<1$, and fixed
polynomials $P,P_1,P_2$,
$$
        \langle f,U_\alpha\rangle
        =(2\alpha^{k^2}+o(1))c_kT(\log T)^{k^2},
$$
$$
        \langle U_\alpha,U_\alpha\rangle
        =(2\alpha^{k^2}+o(1))c_kT(\log T)^{k^2},
$$
$$
        \langle f,V_{\theta,P}\rangle
        =(2\mathcal A_{k,\theta}(P)+o(1))c_kT(\log T)^{k^2},
$$
$$
        \langle V_{\theta,P_1},V_{\theta,P_2}\rangle
        =(2\mathcal B_{k,\theta}(P_1,P_2)+o(1))c_kT(\log T)^{k^2},
$$
and
$$
        \langle U_\alpha,V_{\theta,P}\rangle
        =
        (2\mathcal C_{k,\alpha,\theta}(P)+o(1))c_kT(\log T)^{k^2}.
$$
\end{lemma}

\begin{proof}
The first two formulae are Soundararajan's $K_2,K_3$ formulae, in the normalisation of \cite[Section 2]{soundararajan1995mean}.  The next two
are the Conrey--Ghosh/Soundararajan $J_2,J_3$ formulae in the same normalisation of $c_k,\theta$, and $P$; the oscillating
$J_4$ term is negligible for $\theta<1/2$ by the unconditional $k=3$
argument in Soundararajan when $k=3$, and by
Proposition~\ref{prop:j4-kappa} when $k\geq 4$.

It remains to prove the mixed product.  Expanding
$U_\alpha(s)V_{\theta,P}(1-s)$ gives
$$
\begin{aligned}
U_\alpha(s)V_{\theta,P}(1-s)
={}&A_k(s;\alpha)\zeta(1-s)B(1-s)\\
&+\zeta(s)A_k(1-s;\alpha)B(s)\\
&+\chi(s)^{-k}\zeta(s)A_k(s;\alpha)B(s)\\
&+\chi(s)^{k-1}\zeta(s)A_k(1-s;\alpha)B(1-s).
\end{aligned}
$$
The first two terms are conjugate.  For the first, the diagonal calculation of \cite[Lemma 2.4]{soundararajan1995mean} applies because $\theta<\alpha<1$ and the present $P$ is fixed; it gives
$$
\begin{aligned}
&\frac{1}{i}\int_{\tfrac{1}{2}+iT}^{\tfrac{1}{2}+2iT}
A_k(s;\alpha)\zeta(1-s)B(1-s)ds\\
&\quad=
T\sum_{\substack{q\leq T^\theta\\ nq\leq T^\alpha}}
\frac{d_{k-1}(q)d_k(nq)}{nq}
P\left(\frac{\log q}{\theta\log T}\right)
+O(T^{1-\delta}).
\end{aligned}
$$
Using the averaged Selberg--Delange estimates quoted in Soundararajan's Lemma 2.4, uniformly in the same range of $q$ and with $D_k(q,1)$ denoting the arithmetic factor from that lemma,
$$
        \sum_{n\leq x}\frac{d_k(qn)}{n}
        =\frac{D_k(q,1)(\log x)^k}{\Gamma(k+1)}+R(q,x),
$$
where the weighted total contribution of $R(q,T^\alpha/q)$, after summing over $q\leq T^\theta$, is $O(T^{1-\delta})$, and
$$
        \sum_{q\leq x}\frac{d_{k-1}(q)D_k(q,1)}{q}
        =\frac{c_k\Gamma(k^2+1)}{\Gamma(K+1)}
        (\log x)^K+O((\log x)^{K-1}).
$$
Partial summation gives the main term
$\mathcal C_{k,\alpha,\theta}(P)c_kT(\log T)^{k^2}$.  The conjugate term gives the same
contribution.

It remains to show that the two oscillatory terms are $o(T(\log T)^{k^2})$.  Write
$$
        A_k(s;\alpha)B(s)
        =
        \sum_{m\leq T^{\alpha+\theta}}\frac{c_m}{m^s}
        =D(s),
        \qquad c_m\ll_{\varepsilon} m^{\varepsilon}.
$$
For the first oscillatory term,
$$
        \int_{\tfrac{1}{2}+iT}^{\tfrac{1}{2}+2iT}
        \chi(s)^{-k}\zeta(s)D(s)ds,
$$
move the contour to $\Re s=1/2-A$, where $A>0$ is fixed and large.  In this shift of contours, no pole
is crossed.  On the new vertical side,
$$
        |\chi(s)^{-k}|\ll T^{-kA},
        \qquad
        |\zeta(s)|\ll_A T^{A+\varepsilon},
$$
and
$$
        |D(s)|\ll_{\varepsilon} T^{(\alpha+\theta)(1/2+A)+\varepsilon}.
$$
Thus the vertical side is
$$
        \ll
        T^{1+(\alpha+\theta)/2-A(k-1-\alpha-\theta)+\varepsilon}.
$$
Since $k\ge 3$ and $\alpha+\theta<3/2$, the coefficient
$k-1-\alpha-\theta$ is positive.  Taking $A$ large gives
$O(T^{1-\delta})$.  On a horizontal side, put $a=1/2-\sigma$.  The
integrand is
$$
        \ll
        T^{(\alpha+\theta)/2-a(k-1-\alpha-\theta)+\varepsilon},
$$
whose maximum is $O(T^{(\alpha+\theta)/2+\varepsilon})=O(T^{3/4-\delta})$.  Hence
the horizontal sides are $O(T^{1-\delta})$.

For the second oscillatory term,
$$
        \int_{\tfrac{1}{2}+iT}^{\tfrac{1}{2}+2iT}
        \chi(s)^{k-1}\zeta(s)D(1-s)ds,
$$
move the contour to $\Re s=1/2+A$.  On the vertical side,
$$
        |\chi(s)^{k-1}|\ll T^{-(k-1)A},\qquad
        |\zeta(s)|\ll_A T^{\varepsilon},
$$
and
$$
        |D(1-s)|\ll_{\varepsilon} T^{(\alpha+\theta)(1/2+A)+\varepsilon}.
$$
The vertical side is again
$$
        \ll
        T^{1+(\alpha+\theta)/2-A(k-1-\alpha-\theta)+\varepsilon}
        =
        O(T^{1-\delta})
$$
for $A$ large.  On a horizontal side, put $b=\sigma-1/2\geq 0$.  Convexity
gives
$$
        |\zeta(\sigma+it)|\ll_{\varepsilon}
        T^{\max((1-\sigma)/2,0)+\varepsilon}.
$$
The integrand is at most
$$
T^{\max(1/4-b/2,0)+(\alpha+\theta)(1/2+b)-(k-1)b+\varepsilon}.
$$
At $b=0$ this exponent is $1/4+(\alpha+\theta)/2<1$, and it decreases as $b$ grows.  Thus these horizontal sides are also $O(T^{1-\delta})$.
Both oscillatory terms are $o(T(\log T)^{k^2})$, completing the proof.
\end{proof}

\begin{thm}\label{thm:two-approx}
Let $k\ge 3$, $0<\theta<1/2$, and let $P$ be real-valued on $[0,1]$.
Let
$$
        L_{k,\theta}(P)=
        \mathcal A_{k,\theta}(P)-\mathcal C_{k,\theta}(P)
$$
and
$$
        Q_{k,\theta}(P_1,P_2)=
        \mathcal B_{k,\theta}(P_1,P_2)
        -
        \mathcal C_{k,\theta}(P_1)\mathcal C_{k,\theta}(P_2).
$$
For every polynomial $P$ with $Q_{k,\theta}(P,P)>0$,
$$
        M_k(T)\geq
        \left(
        2+
        2\frac{L_{k,\theta}(P)^2}{Q_{k,\theta}(P,P)}
        +o(1)
        \right)c_kT(\log T)^{k^2}.
$$
\end{thm}

\begin{proof}
Replace $V_{\theta,P}$ by
$$
        W_{\alpha,\theta,P}
        =
        V_{\theta,P}
        -\frac{\mathcal C_{k,\alpha,\theta}(P)}{\alpha^{k^2}}U_\alpha.
$$
For fixed $\alpha<1$, Lemma~\ref{lem:mixed-products} gives the following asymptotics as $T\to\infty$; only after this limit do we let $\alpha\to1^-$. Thus
$$
        \langle U_\alpha,W_{\alpha,\theta,P}\rangle=o(T(\log T)^{k^2}),
$$
$$
        \langle f-U_\alpha,W_{\alpha,\theta,P}\rangle
        =
        (2(\mathcal A_{k,\theta}(P)-\mathcal C_{k,\alpha,\theta}(P))+o(1))c_kT(\log T)^{k^2},
$$
and
$$
        \langle W_{\alpha,\theta,P},W_{\alpha,\theta,P}\rangle
        =
        \left(2\left(\mathcal B_{k,\theta}(P,P)
        -\frac{\mathcal C_{k,\alpha,\theta}(P)^2}{\alpha^{k^2}}\right)+o(1)\right)c_kT(\log T)^{k^2}.
$$
Moreover,
$$
2\Re\langle f,U_\alpha\rangle-\langle U_\alpha,U_\alpha\rangle
=(2\alpha^{k^2}+o(1))c_kT(\log T)^{k^2}.
$$
Now use
$$
        0\leq \|f-U_\alpha-\lambda W_{\alpha,\theta,P}\|^2
$$
and optimize over real $\lambda$.  For $\alpha$ sufficiently close to $1$, the denominator is positive by the hypothesis $Q_{k,\theta}(P,P)>0$ and continuity.  This gives
$$
M_k(T)\geq
\left(
2\alpha^{k^2}
        +
2\frac{(\mathcal A_{k,\theta}(P)-\mathcal C_{k,\alpha,\theta}(P))^2}
{\mathcal B_{k,\theta}(P,P)-\mathcal C_{k,\alpha,\theta}(P)^2/\alpha^{k^2}}
+o(1)
\right)c_kT(\log T)^{k^2}.
$$
Letting $\alpha\to1^-$ gives the required lower bound.
\end{proof}

\subsection{Rational computation}\label{subsec:computation}
We now explain exactly how the numerical constants in the theorem are obtained.
Work in the finite-dimensional space
$$
        \operatorname{span}\{L_0,\ldots,L_D\},
        \qquad
        L_j(x)=P_j(2x-1),
$$
where $P_j(x)$ is the Legendre polynomial.  Explicitly,
$$
        L_j(x)=
        \sum_{m=0}^j
        (-1)^{j-m}\binom jm\binom{j+m}{m}x^m.
$$
For rational $\theta$, all matrix entries are rational.  The hypergeometric
polynomial is finite, with
$$
{}_2F_1(1-k,-k;K;-\theta x)
=
\sum_{m=0}^{k-1}
\frac{(1-k)_m(-k)_m}{(K)_m m!}(-\theta x)^m.
$$
Also
$$
        (1-\theta x)^k=\sum_{m=0}^k\binom km(-\theta x)^m.
$$
For each basis vector $L_j$, compute
$$
        h_j(x)=\int_x^1(z-x)^{k-1}L_j(z)dz
$$
as an exact polynomial.  Then compute
$$
        \ell_j=L_{k,\theta}(L_j)
$$
and the symmetric matrix
$$
        q_{ij}
        =
        \frac{1}{2}\left\{
        Q_{k,\theta}(L_i,L_j)+Q_{k,\theta}(L_j,L_i)
        \right\}.
$$
Only this symmetric part contributes to $Q(P,P)$.

Solve the exact rational linear system
$$
        q a=\ell.
$$
For the polynomial
$$
        P(x)=\sum_{j=0}^D a_jL_j(x)
$$
one has
$$
        L_{k,\theta}(P)=\ell^Ta,
        \qquad
        Q_{k,\theta}(P,P)=a^Tqa=\ell^Ta.
$$
Thus, if $\ell^Ta>0$, Theorem~\ref{thm:two-approx} gives the
constant
$$
        2+2\ell^Ta.
$$
This is the number given by the rational computation.

For the table above we used
$$
        \theta=\frac{99999}{200000},
        \qquad
        D=100.
$$

The exact degree $100$ endpoint computation at $\theta=1/2$ gives the
limiting values
$$
597.331516728188,\quad
139.324194681479,\quad
12.577507384068,\quad
2.259893101629,\quad
2.001975239588
$$
for $k=7,8,9,10,11$, respectively. The proof requires $\theta<1/2$, and by continuity we may achieve bounds arbitrarily close to those evaluated at $\theta=1/2$. Theorem \ref{thm:sound_bounds_unconditional} follows upon summing over dyadic intervals.

\section{Proof of Theorem \ref{thm:joint_lower}}\label{sect:joint_lower}
We begin with a lemma which enables us to evaluate a coefficient sum which arises naturally in the proof.
\begin{lemma}\label{coeff}
As $x\to\infty$,
$$
\sum_{m\leq x}\frac{\lambda_{a_1,\ldots,a_K}(m)^2}{m}=C(a_1,\ldots,a_K)(\log x)^{K^2+2N}+O\left((\log x)^{K^2+2N-1}\right).
$$
\end{lemma}
\begin{proof}
We use the following standard Ikehara--Delange consequence, in the form of \cite[II.5]{Tenenbaum}, followed by partial summation: if $b(n)\geq0$ and $\sum b(n)n^{-s}$ is meromorphic in $\Re s\geq1-\eta$, has polynomial vertical-strip growth, has no singularity on $\Re s\geq1$ except
$$
\sum_{n=1}^{\infty}b(n)n^{-s}=L(s-1)^{-R}+O(|s-1|^{-R+1})
$$
at $s=1$, then
\begin{equation}\label{ikehara_delange}
\sum_{n\leq x}\frac{b(n)}{n}=\frac{L}{\Gamma(R+1)}(\log x)^R+O((\log x)^{R-1}).
\end{equation}
For small complex $u_1,\ldots,u_K,v_1,\ldots,v_K$ and $\Re s>1$, the relevant Dirichlet series has the following explicit Euler product:
\begin{equation}\label{eq:joint_euler_product}
\begin{aligned}
&\sum_{m=1}^{\infty}\frac{\left(\sum_{m_1\cdots m_K=m}m_1^{-u_1}\cdots m_K^{-u_K}\right)
\left(\sum_{n_1\cdots n_K=m}n_1^{-v_1}\cdots n_K^{-v_K}\right)}{m^s}                 \\
&=\prod_{r=1}^{K}\prod_{q=1}^{K}\zeta(s+u_r+v_q)                                                   \\
&\quad\times\prod_p\Biggl[\left(\sum_{\ell=0}^{\infty}
\left(\sum_{e_1+\cdots+e_K=\ell}p^{-e_1u_1-\cdots-e_Ku_K}\right)
\left(\sum_{f_1+\cdots+f_K=\ell}p^{-f_1v_1-\cdots-f_Kv_K}\right)p^{-\ell s}\right)             \\
&\hspace{7.5em}\times\prod_{r=1}^{K}\prod_{q=1}^{K}(1-p^{-s-u_r-v_q})\Biggr].
\end{aligned}     
\end{equation}
The expression inside the square brackets is the local factor at $p$ of the holomorphic Euler product left after the polar zeta factors are removed.  If $|u_r|,|v_q|\leq\eta$ and $\Re s\geq1-\eta$, then this local factor is $1+O_K(p^{-2+C_K\eta})$: its first parenthesis begins as
$$
1+\sum_{r=1}^{K}\sum_{q=1}^{K}p^{-s-u_r-v_q}+O_K(p^{-2+C_K\eta}),
$$
whereas $\prod_{r,q}(1-p^{-s-u_r-v_q})=1-\sum_{r,q}p^{-s-u_r-v_q}+O_K(p^{-2+C_K\eta})$.  Taking $\eta$ small gives absolute and locally uniform convergence near $\Re s\geq1$.  At $u_r=v_q=0$ and $s=1$, the remaining Euler product in \eqref{eq:joint_euler_product} is exactly
\begin{equation}\label{eq:joint_euler_value}
\prod_p\left(1-\frac1p\right)^{K^2}\sum_{\ell=0}^{\infty}\frac{\binom{\ell+K-1}{K-1}^2}{p^\ell}.
\end{equation}
Apply $(-\partial_{u_1})^{a_1}\cdots(-\partial_{u_K})^{a_K}(-\partial_{v_1})^{a_1}\cdots(-\partial_{v_K})^{a_K}$ to \eqref{eq:joint_euler_product}, then put all shifts equal to zero.  The left side becomes $\sum_{m\geq1}\lambda_{a_1,\ldots,a_K}(m)^2m^{-s}$.  Since $\zeta(1+z+u_r+v_q)=(z+u_r+v_q)^{-1}+O(1)$, the highest pole at $z=s-1=0$ comes only from the principal parts of the zeta factors and from the value \eqref{eq:joint_euler_value} of the remaining Euler product; any derivative falling on that holomorphic Euler product, or on a regular part of one zeta factor, lowers the pole order by at least one.  Also
$$
\prod_{r,q}(z+u_r+v_q)^{-1}=z^{-K^2}\prod_{r,q}\left(1+\frac{u_r+v_q}{z}\right)^{-1}.
$$
Thus the leading term is precisely the derivative in \eqref{eqn:C_def}, multiplied by \eqref{eq:joint_euler_value}, and therefore
$$
\sum_{m=1}^{\infty}\frac{\lambda_{a_1,\ldots,a_K}(m)^2}{m^s}
=\frac{\Gamma(K^2+2N+1)C(a_1,\ldots,a_K)}{(s-1)^{K^2+2N}}+O(|s-1|^{-K^2-2N+1}).
$$
The required vertical-strip growth follows after the polar zeta factors have been separated: fixed derivatives are controlled by Cauchy estimates and the standard vertical-strip bounds for $\zeta$, while the remaining Euler product is locally uniformly convergent.  Applying \eqref{ikehara_delange} with $R=K^2+2N$ gives the lemma.
\end{proof}

We now evaluate the smoothed first moment.
\begin{lemma}\label{first}
Let $w\in C_c^\infty([1,2])$, $w\geq0$, and write $\widehat w(y)=\int_{\mathbb R}w(u)e^{-iyu}du$.  For fixed $0<\theta<1$,
\begin{align}
&\int_{\mathbb R}w(t/T)\prod_{r=1}^{K}\zeta^{(a_r)}(\tfrac{1}{2}+it)
\sum_{h\leq T^\theta}\frac{\lambda_{a_1,\ldots,a_K}(h)}{h^{1/2-it}}dt  \notag \\
&\qquad=(-1)^NT\widehat w(0)\sum_{h\leq T^\theta}\frac{\lambda_{a_1,\ldots,a_K}(h)^2}{h}
+o(T(\log T)^{K^2+2N}),\label{diffmom}
\end{align}
while
\begin{equation}
\int_{\mathbb R}w(t/T)\left|\sum_{h\leq T^\theta}\frac{\lambda_{a_1,\ldots,a_K}(h)}{h^{1/2+it}}\right|^2 dt
=T\widehat w(0)\sum_{h\leq T^\theta}\frac{\lambda_{a_1,\ldots,a_K}(h)^2}{h}
+o\left(T(\log T)^{K^2+2N}\right).                                      \label{ampsq}
\end{equation}
\end{lemma}
\begin{proof}
We first prove the shifted form from which \eqref{diffmom} follows.  Uniformly for $|\alpha_i|,|\beta_i|\leq c/\log T$,
\begin{align}
&\int_{\mathbb R}w(t/T)\prod_{i=1}^{K}\zeta(\tfrac{1}{2}+it+\alpha_i)
\sum_{h\leq T^\theta}\frac{\sum_{h_1\cdots h_K=h}h_1^{-\beta_1}\cdots h_K^{-\beta_K}}{h^{1/2-it}}dt  \notag\\
&=T\widehat w(0)\sum_{h\leq T^\theta}\frac{\left(\sum_{h_1\cdots h_K=h}h_1^{-\alpha_1}\cdots h_K^{-\alpha_K}\right)
\left(\sum_{n_1\cdots n_K=h}n_1^{-\beta_1}\cdots n_K^{-\beta_K}\right)}{h}+O_\varepsilon(T^{\theta+\varepsilon}).          \label{shifted}
\end{align}
Indeed, with $X=T^B$, Mellin inversion gives, for $s=1/2+it$ and $t\in[T,2T]$,
$$
\sum_{m=1}^{\infty}\frac{\sum_{m_1\cdots m_K=m}m_1^{-\alpha_1}\cdots m_K^{-\alpha_K}}{m^s}e^{-m/X}
=\frac1{2\pi i}\int_{(2)}\Gamma(z)X^z\prod_{i=1}^{K}\zeta(s+\alpha_i+z)dz.
$$
Moving the line to $\Re z=-\delta$, $0<\delta<1/4$, gives the product of zeta-functions from $z=0$; the poles $z=1-s-\alpha_i$, even if coincident, contribute only exponentially small residues because of the gamma factor, and the new line is $O(X^{-\delta}T^{K(1/4+\delta/2)+\varepsilon})$.  Choosing $B$ large makes the error $O_A(T^{-A})$.  Substituting this expansion into the left side of \eqref{shifted} gives
\begin{align*}
T\sum_{m\geq1}\sum_{h\leq T^\theta}&\frac{\left(\sum_{m_1\cdots m_K=m}m_1^{-\alpha_1}\cdots m_K^{-\alpha_K}\right)
\left(\sum_{h_1\cdots h_K=h}h_1^{-\beta_1}\cdots h_K^{-\beta_K}\right)e^{-m/X}}{(mh)^{1/2}} \\
&\times\widehat w\left(T\log\frac{m}{h}\right)+O_A(T^{-A}).
\end{align*}
The terms $m=h$ give the main term in \eqref{shifted}; replacing $e^{-h/X}$ by $1$ is negligible after increasing $B$.  For $m\neq h$, use
$$
\sum_{m_1\cdots m_K=m}m_1^{-\alpha_1}\cdots m_K^{-\alpha_K}\ll_\varepsilon m^\varepsilon,\qquad
\widehat w(y)\ll_A(1+|y|)^{-A}.
$$
If $m\asymp h$, then $|\log(m/h)|\gg |m-h|/h$, so the off-diagonal contribution is
$$
\ll T\sum_{h\leq T^\theta}h^{-1+2\varepsilon}\sum_{d\geq1}\left(1+\frac{Td}{h}\right)^{-A}
\ll\sum_{h\leq T^\theta}h^{2\varepsilon}\ll T^{\theta+\varepsilon};
$$
if $m$ is not comparable with $h$, then $|\log(m/h)|\gg1$, and rapid decay of $\widehat w$ together with $e^{-m/X}$ makes the contribution $O_A(T^{-A})$.  This proves \eqref{shifted}.

Now differentiate \eqref{shifted} $a_i$ times in each $\alpha_i$ and $a_i$ times in each $\beta_i$, then set all shifts to zero.  Cauchy's formula is applicable because \eqref{shifted} is uniform and holomorphic in polydiscs of radius $\asymp1/\log T$; differentiating the error costs at most $(\log T)^{2N}$, still $o(T(\log T)^{K^2+2N})$.  The $\beta$-derivatives of the Dirichlet polynomial supply the sign $(-1)^N$, and this gives \eqref{diffmom}.  Finally, \eqref{ampsq} follows by expanding the square; the diagonal is the displayed main term, and the off-diagonal is bounded by the same estimate just used, with both variables at most $T^\theta$.
\end{proof}

\begin{proof}[Proof of Theorem \ref{thm:joint_lower}]
Choose $w\in C_c^\infty([1,2])$ with $0\leq w\leq1$.  Since $w(t/T)$ is supported on $[T,2T]$, Cauchy's inequality and \eqref{diffmom}--\eqref{ampsq} give, for fixed $0<\theta<1$,
\begin{align*}
\int_T^{2T}\left|\prod_{r=1}^{K}\zeta^{(a_r)}(\tfrac{1}{2}+it)\right|^2 dt
&\geq
\frac{\left|\int_{\mathbb R}w(t/T)\prod_{r=1}^{K}\zeta^{(a_r)}(\tfrac{1}{2}+it)
\sum_{h\leq T^\theta}\lambda_{a_1,\ldots,a_K}(h)h^{-1/2+it}dt\right|^2}
{\int_{\mathbb R}w(t/T)\left|\sum_{h\leq T^\theta}\lambda_{a_1,\ldots,a_K}(h)h^{-1/2-it}\right|^2dt}  \\
&=(1+o(1))T\widehat w(0)\sum_{h\leq T^\theta}\frac{\lambda_{a_1,\ldots,a_K}(h)^2}{h}.
\end{align*}
By Lemma \ref{coeff}, the last sum is $C(a_1,\ldots,a_K)\theta^{K^2+2N}(\log T)^{K^2+2N}+O((\log T)^{K^2+2N-1})$.  Hence
$$
\int_T^{2T}\left|\prod_{r=1}^{K}\zeta^{(a_r)}(\tfrac{1}{2}+it)\right|^2dt
\geq\left(C(a_1,\ldots,a_K)\widehat w(0)\theta^{K^2+2N}+o(1)\right)T(\log T)^{K^2+2N}.
$$
Choose $w$ with $\widehat w(0)>1-\varepsilon$, then choose $\theta<1$ with $\theta^{K^2+2N}>1-\varepsilon$, and let $\varepsilon\downarrow0$.  Since
$$
\left|\prod_{r=1}^{K}\zeta^{(a_r)}(\tfrac{1}{2}+it)\right|^2
=\prod_{\mu=1}^{j}\left|\zeta^{(n_\mu)}(\tfrac{1}{2}+it)\right|^{2k_\mu},
$$
the dyadic lower bound follows.

For the interval $[1,T]$, the preceding dyadic estimate is uniform with $T$ replaced by any $U\in[T^\eta,T]$, where $0<\eta<1$ is fixed; this follows from the uniformity of the errors above and the coefficient-sum error in Lemma \ref{coeff}.  Apply it to $U=T/2,T/4,T/8,\ldots$ while $U\geq T^\eta$.  The intervals $[U,2U]$ are disjoint and lie inside $[1,T]$, and
$$
\sum_{T/2^m\geq T^\eta}\frac{T}{2^m}\left(\log\frac{T}{2^m}\right)^{K^2+2N}=(1+o(1))T(\log T)^{K^2+2N}.
$$
Summing the dyadic bounds over these intervals and then letting $\varepsilon\downarrow0$ proves the lower bound on $[1,T]$.
\end{proof}

\addcontentsline{toc}{chapter}{Bibliography}

\bibliography{references}    
\bibliographystyle{plain} 
\end{document}